\documentclass{article}

\usepackage{amsmath,amssymb,amsthm}
\usepackage{stmaryrd,mathrsfs}
\usepackage{microtype}
\usepackage{bm}
\usepackage[lofdepth,lotdepth]{subfig}
\usepackage{mleftright}
\mleftright
\usepackage{authblk}
\usepackage{hyperref}
\usepackage{breakurl}
\usepackage{mathtools}
\usepackage[shortlabels]{enumitem}
\setlist{font=\normalfont,leftmargin=*}
\usepackage[ruled,linesnumbered,vlined]{algorithm2e}
\usepackage{listings}

\newtheorem{thm}{Theorem}[section]
\newtheorem{lem}[thm]{Lemma}
\newtheorem{cor}[thm]{Corollary}
\newtheorem{prop}[thm]{Proposition}
\newtheorem{sg}{SG}
\newtheorem{introthm}{Theorem}

\theoremstyle{definition}
\newtheorem{defn}[thm]{Definition}
\newtheorem{exmp}[thm]{Example}
\newtheorem{quest}[thm]{Question}

\theoremstyle{remark}
\newtheorem{rem}[thm]{Remark}
\newtheorem{discussion}[thm]{Discussion}
\newtheorem*{notation}{Notation}
\newtheorem*{convention}{Convention}

\numberwithin{equation}{section}

\DeclareMathOperator{\lct}{lct}
\DeclareMathOperator{\fpt}{fpt}
\DeclareMathOperator{\ft}{ft}
\newcommand{\ftb}{\ft^{\idealb}}
\newcommand{\ftm}{\ft^{\idealm}}
\newcommand{\ftlm}[1]{\ft^{\idealm}(\ell^{#1})}
\newcommand{\ftlb}[1]{\ft^{\idealb}(\ell^{#1})}
\newcommand{\Dlm}{\Delta}
\newcommand{\Dlb}{\Delta}
\newcommand{\Plm}{\Phi}
\newcommand{\Plb}{\Phi}

\newcommand{\down}[1]{\left\lfloor #1 \right\rfloor}
\newcommand{\up}[1]{\left\lceil #1 \right\rceil}%

\newcommand{\tr}[2]{\left \langle {#1} \right \rangle_{#2}} 
\newcommand{\lpr}[2]{ [ #1 \hspace{.5mm} \%  \hspace{.5mm} #2  ]} 

\newcommand{\abs}[1]{\left|{#1}\right|}
\newcommand{\norm}[1]{\left\|{#1}\right\|}
 
\newcommand{\Upper}{\mathscr{U}}
\newcommand{\Lower}{\mathscr{L}}
\newcommand{\Boundary}{\mathscr{B}}
\newcommand{\Trivial}{\mathscr{T}}

\DeclarePairedDelimiter{\ideal}{\langle}{\rangle}
\newcommand{\idealm}{\mathfrak{m}}
\newcommand{\ideala}{\mathfrak{a}}
\newcommand{\idealb}{\mathfrak{b}}

\newcommand{\vv}[1]{\mathbf{{#1}}}
\newcommand{\vvv}[1]{\bm{{#1}}}
\newcommand{\canvec}{\mathbf{e}}

\newcommand{\kk}{\Bbbk}
\newcommand{\FF}{\mathbb{F}}
\newcommand{\RR}{\mathbb{R}}
\newcommand{\RRpos}{\mathbb{R}_{>0}}
\newcommand{\RRnn}{\mathbb{R}_{\ge0}}
\newcommand{\ZZ}{\mathbb{Z}}
\newcommand{\QQ}{\mathbb{Q}}
\newcommand{\QQpos}{\mathbb{Q}_{>0}}
\newcommand{\QQnn}{\mathbb{Q}_{\ge0}}
\newcommand{\NN}{\mathbb{N}}
\newcommand{\NNpos}{\mathbb{N}_{>0}}
\newcommand{\CC}{\mathbb{C}}
\newcommand{\unitint}{[0,1]_{p^\infty}}
\newcommand{\infint}{(\QQnn)_{p^\infty}}

\newcommand{\closure}[1]{\overline{#1}}
\newcommand{\oball}[2]{B_{#1}\left({#2}\right)}
\newcommand{\cball}[2]{\closure{B}_{#1}\left({#2}\right)}
\newcommand{\boundary}{\partial}

\newcommand{\eg}{e.g., }
\newcommand{\ie}{i.e., }
\newcommand{\etc}{etc.}

\newcommand{\cell}[1]{\left({#1}\right)}
\newcommand{\iref}{\ref}
\newcommand{\radius}{r}
\DeclareMathOperator{\mult}{mult}

\begin{document}

\title{$F$-threshold functions:  syzygy gap fractals and the two-variable homogeneous case}
\renewcommand\Affilfont{\footnotesize}
\author[*]{Daniel J.~Hern\'andez}
\author[**]{Pedro Teixeira}
\affil[*]{\textsl{Department of Mathematics, University of Utah, Salt Lake City, UT 84112, USA}
\authorcr \textsl{E-mail address}: \texttt{dhernan@math.utah.edu}}
\affil[**]{\textsl{Department of Mathematics, Knox College, Galesburg, IL 61401, USA}
\authorcr \textsl{E-mail address}: \texttt{pteixeir@knox.edu}}
\date{}
\maketitle

\begin{abstract}
	In this article we study $F$-pure thresholds (and, more generally, $F$-thresholds) of homogeneous polynomials in two variables
		over a field of characteristic~$p>0$.
	Passing to a field extension, we factor such a polynomial into a product of powers of pairwise prime linear forms, 
		and to this collection of linear forms we associate a special type of function called a \emph{syzygy gap fractal}.  
	We use this syzygy gap fractal to study, at once, the collection of all $F$-pure thresholds 
		of all polynomials constructed with the same fixed linear forms. 
	This allows us to describe the structure of the denominator of such an $F$-pure threshold,
		showing in particular that whenever the $F$-pure threshold differs from its expected value 
		its denominator is a multiple of~$p$.
	This answers a question of Schwede in the two-variable homogeneous case.
	In addition, our methods give an algorithm to compute $F$-pure thresholds of homogenous polynomials in two variables.
\end{abstract}

\section{Introduction}

Fix an arbitrary field $\kk$ of characteristic $p>0$, and consider a polynomial $g$ in $\kk[x_1, \ldots, x_r]$ with $g(\vv{0}) = 0$.  
By utilizing properties of the Frobenius endomorphism on the ambient polynomial ring, one may show that 
	\[ \fpt(g)\coloneqq \inf \left \{ \frac{a}{p^e} : g^a \in \ideal{x_1^{p^e}, \ldots, x_r^{p^e}} \right \}\]
	is a well-defined nonzero real number contained in the unit interval.  
This invariant, called the \emph{$F$-pure threshold} of $g$ (at the origin), was originally introduced in 
	\cite{takagi+watanabe.F-pure_thresholds}, though the definition we give here follows \cite{mustata+takagi+watanabe.F-thresholds}.  
Though it is not obvious from this definition, it turns out that the $F$-pure threshold of a polynomial is always a rational number  
	\cite[Corollary~2.30, Theorem~3.1]{blickle+mustata+smith.discr_rat_FPTs}.

Note that if, in the description of $\fpt(g)$ given above, one replaces the Frobenius power 
	$\idealm^{[p^e]} = \ideal{x_1^{p^e}, \ldots, x_r^{p^e}}$ of the maximal ideal $\idealm = \ideal{x_1, \ldots, x_r}$ 
	with the ordinary power $\idealm^{p^e}$, one would instead obtain the reciprocal of the multiplicity of $g$ at the origin 
	(that is, the largest $N$ such that $g \in \idealm^N$).  
Thus, the $F$-pure threshold may be thought of as a sort of  ``Frobenius multiplicity'', with smaller values  
	corresponding to ``worse'' singularities at the origin. 

In this article we are motivated by the relationship between $F$-pure thresholds and another important invariant, 
	traditionally defined for polynomials over fields of characteristic zero.  
Consider a polynomial $g$ over a field of characteristic zero that vanishes at the origin.  
By referring to a log resolution of singularities, one may assign to $g$ the numerical invariant $\lct(g)$, called the 
	\emph{log canonical threshold} of $g$ (at the origin).   
Like the $F$-pure threshold, the log canonical threshold of a polynomial is always a nonzero rational number contained 
	in the unit interval and may be thought of as a measure of the singularity of $g$ (at the origin), with smaller values 
	corresponding to ``worse'' singularities.  
For more on this invariant, we refer the reader to the survey \cite{blickle+lazarsfeld.intro_multiplier_ideals} and the references cited therein.
Throughout the rest of this article, we shall always consider polynomials vanishing at the origin and shall omit the phrase 
	``at the origin'' when referring to $F$-pure and log canonical thresholds. 

Remarkably, $F$-pure and log canonical thresholds are intimately related. 
Consider a polynomial $g_0$ over $\QQ$ and, for $p \gg 0$, let $g_p$ denote the polynomial over $\FF_p$, 
	the field with $p$ elements, obtained by reducing the coefficients of~$g_0$ modulo~$p$.   
It  follows from work of Hara and Yoshida \cite{hara+yoshida.generalization_TC_multiplier_ideals} that  
	$\fpt(g_p) \leq \lct(g_0)$ and $\lim_{p \to \infty} \fpt(g_p) = \lct(g_0)$ 
	(see \cite[Theorem~3.4]{mustata+takagi+watanabe.F-thresholds}).  
In general, little is known about how $\fpt(g_p)$ varies with $p$, and an important open conjecture predicts that 
	$\fpt(g_p) = \lct(g_0)$ for infinitely many primes.  
Motivated by understanding the situation when $\fpt(g_p) \neq \lct(g_0)$, the following question was asked by 
	Karl Schwede.\footnote{Schwede's question was asked during the Computational Workshop on Frobenius Singularities 
	and Invariants, held in Ann Arbor, MI, in 2012.  His question, and others, can be found at 
	\url{https://sites.google.com/site/computingfinvariantsworkshop/open-questions}.}  

\begin{quest}[Schwede]\label{question: Schwede}
	Fix $g_0 \in \ZZ[x_1, \ldots, x_r]$ vanishing at the origin.
	Assume $\fpt(g_p) \neq \lct(g_0)$, for a prime $p\gg 0$. 
	Write $\fpt(g_p) = a/b$ in lowest terms. 
	Does $p$ divide $b$? 
\end{quest}

In recent work, Bhatt and Singh (and, in a subsequent generalization, N\'u\~nez-Betancourt, Witt, Zhang, and the first author) 
	have shown the following:  
Suppose $g_0$ is a polynomial over $\QQ$ that is homogeneous under some $\NN$-grading and such that 
	the ideal generated by the partial derivatives of $g_0$ is primary to the ideal generated by the variables.  
If $p \gg 0$ and $\fpt(g_p) \neq \lct(g_0)$, then the denominator of $\fpt(g_p)$ is a not just a multiple of $p$ but in fact 
	a \emph{power} of~$p$ \cite{bhatt+singh.FPT_calabi_yao,hernandez+others.fpt_quasi-homog.polys}.  
In particular, a ``stronger'' form of Schwede's question has a positive answer for such 
	polynomials.\footnote{The question of under what circumstances the denominator of $\fpt(g_p)$ must be a power of~$p$ 
	was asked by the first author during the aforementioned workshop.}  
As far as the authors are aware, there are no such general descriptions of $F$-pure thresholds of polynomials whenever 
	one relaxes the hypothesis on the ideal generated by the corresponding partial derivatives.

 In this article we shed some further light on Question~\ref{question: Schwede}, answering it in what is perhaps the simplest nontrivial case.

\begin{introthm}[see Theorem~\ref{thm: p in denominator}]  
	If $G_0 \in \QQ[x,y]$ is a non-constant homogeneous polynomial,\footnote{Ongoing work by the first author and 
		Emily Witt suggests that this theorem may be false for certain non-homogeneous polynomials in two variables.} 
		then Question~\ref{question: Schwede} has a positive answer for $G_0$.  
	More precisely, if $p \gg 0$ and $\fpt(G_p) \neq \lct(G_0)$, then the minimal denominator of $\fpt(G_p)$ is of the 
		form $kp^e$, where $e \geq 1$ and $k$ divides the multiplicity of some linear factor \textup(over $\CC$\textup) of $G_0$.
\end{introthm}

We now briefly describe the main ideas in this article.  
For the remainder of this introduction, $G$ will denote a homogeneous polynomial in $\kk[x,y]$, where $\kk$ is a field of characteristic $p>0$.  
Many of our results deal with a generalization of $F$-pure thresholds called, simply, \emph{$F$-thresholds};  
	this generality pays off later, allowing us to extend our main result to polynomials that are homogeneous under non-standard 
	$\NN$-gradings---see Theorem~\ref{thm: p in denominator QH}.
Given an ideal $\idealb \subseteq \kk[x,y]$,  the $F$-threshold of $G$ with respect to $\idealb$, denoted $\ftb(G)$, is a 
	numerical invariant describing the complexity of the hypersurface defined by $G$.  
$F$-thresholds generalize $F$-pure thresholds, in the sense that $\fpt(G) = \ft^{\ideal{x,y}}(G)$ 
	(see \cite{mustata+takagi+watanabe.F-thresholds}).  
Rather than considering $F$-thresholds with respect to \emph{arbitrary} ideals $\idealb$, we focus instead on the case when 
	$\idealb $ is generated by two non-constant, relatively prime forms.  
The motivation for this restriction is that it allows us to apply the theory of \emph{syzygy gap fractals}, 
	introduced by Han in her thesis \cite{han.thesis} and generalized and studied by the second author in \cite{mythesis,sg1}, and 
	related to the theory of $p$-fractals developed by Monsky and the second author \cite{mythesis, pfractals1, pfractals2}.

Over $\closure{\kk}$ there exists a collection of pairwise prime linear forms $\ell = (\ell_1, \ldots, \ell_n)$ 
	 such that $G = \ell_1^{a_1} \cdots \ell_n^{a_n}$, for some $a_1, \ldots, a_n\in \NNpos$.
In Section~\ref{s: F-threshold function} we define a continuous function $\ftlb{\bullet}: \RRpos^n \to \RR$ with the property that 
	$\vv{k} \mapsto \ftb(\ell_1^{k_1} \cdots \ell_n^{k_n})$ whenever $\vv{k} = (k_1, \ldots, k_n) \in \NNpos^n$.  
This function, called an \emph{$F$-threshold function} (see Definition~\ref{defn: F-threshold function}), encodes the 
	$F$-thresholds with respect to the fixed ideal~$\idealb$ of all homogeneous polynomials with the same linear 
	factors as $G$, and will play a key role in this article.  
The $F$-threshold function is described in terms of a syzygy gap fractal
	attached to the ideal $\idealb$ and the linear forms $\ell_1, \ldots, \ell_n$, and  properties of syzygy gap fractals
	worked out in \cite{sg1} allow us to understand $F$-threshold functions well enough to prove our main results.  
More precisely, we see that the $F$-threshold function  
	(and, hence, the $F$-thresholds of all homogeneous polynomials with the same linear factors as $G$) is completely 
	determined by a family of distinguished points, called \emph{critical points}.  
It turns out that every coordinate of a critical point is a rational number whose denominator is a power of $p$, and it is precisely 
	this fact that allows us to say something about the denominators of $F$-pure thresholds of homogeneous polynomials.

Finally, we point out that our methods are effective, and provide us with an algorithm to compute $F$-pure thresholds of homogeneous
	polynomials in two variables, which has been implemented by the second author in the \emph{Macaulay2} \cite{M2} 
	package \emph{PosChar} \cite{poschar}
	(see Appendix~\ref{appendix: algorithm}).
This implementation is remarkably efficient when the polynomial factors over a relatively small field.
For instance, if $a$ is algebraic over $\FF_5$, satisfying $a^3+a+1=0$, and 
	\[G=x^{420}y^{419}(x+y)^{417}(x+ay)^{390}(x+a^2y)^{402}(x+a^3y)^{438} \in \FF_5(a)[x,y],\] 
	then our current implementation takes only about 0.3 seconds to report that 
	\[\fpt(G)=\frac{46636216675556057485911762783799675605705641779512143}
		{2\cdot 3\cdot 5^{76}\cdot 73}.\]

\subsection{Outline}

This paper is organized as follows.  
In Section~\ref{s:  Background} we recall the basics of base~$p$ expansions of real numbers as well as the properties of 
	syzygy gap fractals needed in this article.  
In Section~\ref{s:  Delta and Phi section} we introduce and study a pair of functions $\Dlb$ and $\Plb$, the former 
	being a special instance of a syzygy gap fractal, and in Section~\ref{s: F-threshold function} we use these 
	functions to define $F$-threshold functions.  
In Section~\ref{s:  CP section} we define the notion of critical points and precisely describe in which ways these points 
	determine the values of $F$-threshold functions.  
In Section~\ref{s:  Sierpinksi} we consider $F$-threshold functions attached to three linear forms.  
In Sections~\ref{s:  fpt} and~\ref{s:  QH case} we apply our methods to the study and computations of $F$-thresholds of polynomials in two variables
	that are homogeneous under either standard or non-standard $\NN$-gradings.

\subsection{Notations and conventions}

The following are some notations and conventions used throughout:
\begin{itemize}
	\item	$p$ denotes a prime number and $q$ always denotes a (variable) power of $p$.
	\item $\kk$ is a field of characteristic $p$.
	\item If $\ideala$ is an ideal of $\kk[x,y]$, then $\ideala^{[q]}$ denotes the $q$th Frobenius power of~$\ideala$, 
			that is, $\ideala^{[q]} \coloneqq  \ideal{ f^q : f \in \ideala}$.
		Also, $\deg \ideala$ denotes the degree or colength of 
			$\ideala$, that is, $\deg \ideala \coloneqq \dim_{\kk}(\kk[x,y] / \ideala)$.  			
	\item The term \emph{form} is used as a synonym for  \emph{nonzero homogenous polynomial}.  
		In this context, $\deg H$ denotes the typical degree of a form $H$ under some fixed $\NN$-grading 
			(usually the standard $\NN$-grading) on the ambient polynomial ring.
	\item	If $S\subseteq \RR$, then $S_q$ denotes the set consisting of rational elements of $S$ with denominator $q$, 
			and $S_{p^\infty}$ denotes the union of all $S_{q}$.
	\item	Vectors in $\RR^n$ are denoted by bold face letters, and their components are denoted by the same letter
			in regular font (\eg $\vv{u}=(u_1,\ldots,u_n)$).
		The canonical basis vectors of $\RR^n$ are denoted by $\canvec_1,\ldots,\canvec_n$.
		The vectors $(0,\ldots,0)$ and $(1,\ldots,1)$ are denoted by $\vv{0}$ and $\vv{1}$.
	\item	Unary operations on real numbers are extended to vectors in a componentwise fashion 
			(\eg $\up{\vv{u}}=(\up{u_1},\ldots,\up{u_n})$).
	\item  Given positive integers $a$ and $d$,  $\lpr{a}{d}$ denotes the \emph{least positive residue} of $a$ modulo $d$; \ie $\lpr{a}{d}$ is the unique integer $1 \leq b \leq d$ such that $a \equiv b \bmod d$.
	
\end{itemize}

\section{Background}
\label{s:  Background}

\subsection{Expansions and truncations}

We review here some terminology and notation concerning real numbers.

\begin{defn}
\label{expansion: D}
By an \emph{expansion} (base $p$) of a real number $0 \leq \lambda \leq 1$, we mean any expression of the form \[ \lambda = \sum_{e=0}^{\infty} \frac{\lambda_e}{p^e},\] where the digits $\lambda_e$ are integers between $0$ and $p-1$.  We call such an expansion \emph{non-terminating} if the sequence of digits $\lambda_e$ is not eventually zero, and \emph{terminating} otherwise.
\end{defn}

\begin{rem}[Comparing expansions]
\label{Comparing expansions: R}  Clearly, the only expansion of $0$ is terminating.  Moreover, a real number $\lambda \in (0,1]$ always has a unique non-terminating expansion, and it also has a (unique) terminating expansion if and only if $\lambda \in \QQ_{p^{\infty}}$.  In this case, the two expansions of $\lambda$ are related as follows:   if $\lambda = \sum_{e=0}^{r} \frac{\lambda_e}{p^e}$ is the unique terminating expansion of $\lambda$, and $\lambda_r \neq 0$, then \[ \lambda = \sum_{0 \leq e < r} \frac{\lambda_e}{p^e} + \frac{\lambda_r-1}{p^r} + \sum_{e > r} \frac{p-1}{p^e}\] is the unique non-terminating expansion of $\lambda$.  
\end{rem}

\begin{defn}  
\label{truncation: D}
Consider a real number $\lambda > 0$.  If $e \geq 0 $ is an integer, we call \[ \tr{\lambda}{e} \coloneqq  \frac{ \up{\lambda p^e} - 1}{p^e} \] the $e$-th \emph{truncation} of $\lambda$ (base $p$). We adopt the convention that $\tr{0}{e} = 0$ for every $e \geq 0$, and given a point $\vv{u} \in \RR^n$ with nonnegative coordinates, we use $\tr{\vv{u}}{e}$ to denote the componentwise truncation of $\vv{u}$.
\end{defn}

\begin{rem}[Characterizations of truncations of positive numbers]
\label{Characterizations of truncations: R} 
Suppose $\lambda > 0$.  It is straightforward to verify that $\tr{\lambda}{e}$ is the unique element of $\QQ_{p^e}$ with \[ \tr{\lambda}{e} < \lambda \leq \tr{\lambda}{e}  + \frac{1}{p^e}.\]  
This leads to an important characterization (and one we will often use without mention):  if $\lambda > 0$ and $\mu \in \QQ_{p^e}$, then $\mu < \lambda$ if and only if $\mu \leq \tr{\lambda}{e}$.  Consequently, if $\lambda>0$,  then $\lambda = \tr{\lambda}{e} + \frac{1}{p^e}$ if and only if $\lambda \in \QQ_{p^e}$.
\end{rem}

\begin{rem}[Truncations in terms of expansions]
\label{truncations and expansions: R}
Fix the unique non-terminating expansion $\lambda = \sum_{e=1}^{\infty} \frac{\lambda_e}{p^e}$ of a real number $0 < \lambda \leq 1$.  
For such an expansion, the $s$th tail $\tau = \sum_{e>s} \frac{\lambda_e}{p^e}$ lies in $(0,1/p^s]$, and therefore
\[ \up{\lambda p^s} = \up{ \left( \sum_{e=1}^s \frac{\lambda_e}{p^e} \right) p^s +  \tau p^s} = \left( \sum_{e=1}^s \frac{\lambda_e}{p^e} \right) p^s  + \up{\tau p^s} = \left( \sum_{e=1}^s \frac{\lambda_e}{p^e} \right) p^s  + 1. \]
In other words, 
\[ \tr{\lambda}{s} = \frac{\lambda_1}{p} + \cdots + \frac{\lambda_s}{p^s}.\]
From this expression, we see that the truncations 
	$\tr{\lambda}{s}$ form a non-decreasing sequence that converges to $\lambda$.
\end{rem}

\begin{rem}[Truncations of rational numbers] 
\label{truncations of rationals: rem} Given positive integers $a$ and $d$, 
\[ \frac{a}{d} \cdot p^e = \frac{ ap^e - \lpr{ap^e}{d}}{d} + \frac{\lpr{ap^e}{d}}{d}, \] 
where $\lpr{ap^e}{d}$ denotes the least positive residue of $ap^e$ modulo $d$.
Since we are dealing with least \emph{positive} residues, the second summand on the right-hand side of the above equation
	lies in $(0,1]$, while the first summand is an integer.
Thus, substituting the above equation into Definition \ref{truncation: D} shows that
\[ \tr{\frac{a}{d}}{e} = \frac{a}{d} - \frac{\lpr{ap^e}{d}}{dp^e}.\]
Basic properties of congruences show that this expression depends only on  $a/d$, and not on the choice of the numerator and denominator.
\end{rem}
	 
\subsection{Syzygy gap fractals}

We gather here some definitions and results concerning syzygy gaps and syzygy gap fractals from \cite{sg1} and adapt
	them to suit our needs.   
In what follows, $F,G,H\in R\coloneqq \kk[x,y]$ are forms with no common factor---that is, there is no non-constant polynomial in 
	$\kk[x,y]$ that divides all of $F$, $G$, and $H$.

\begin{defn}\label{defn: syzygy gap definition}
  	Let $M = R(-\deg F) \oplus R(-\deg G) \oplus R(-\deg H)$, so that, by the Hilbert Syzygy Theorem, there exists an 
		exact sequence of graded $R$-modules 
		\[ 0 \to R(-m) \oplus R(-n) \to M \to R \to R / \ideal{F,G,H} \to 0. \]
	The \emph{syzygy gap} of $F$, $G$, and $H$ is the nonnegative integer $\delta(F,G,H)\coloneqq\abs{m-n}$.
\end{defn}

Syzygy gaps are easily computed using \emph{Macaulay2} \cite{M2} or similar software.
In \emph{Macaulay2}, the following code defines a function \texttt{delta} that computes syzygy gaps:
\lstset{basicstyle=\small\ttfamily}
\begin{lstlisting}
delta := (F,G,H) -> (                                                
       M:= ker matrix {{F,G,H}};                              
       d:= degrees source generators M;                       
       abs(d_0_0-d_1_0)               
)
\end{lstlisting}

Perhaps one of the most important aspects of syzygy gaps is their relation with the degrees of certain ideals, which we recall below.

\begin{sg}[{\cite[Proposition~2.2]{sg1}, \cite[Lemma~1(2)]{monsky.mason}}]\label{sg: relation between degree and delta}
	The syzygy gap $\delta(F,G,H)$ and the degree of the ideal $\ideal{F,G,H}$ are related as follows\textup: 
		\[4\deg\ideal{F,G,H}=Q(\deg F,\deg G,\deg H) +\delta(F,G,H)^2,\] 
		where  $Q(a,b,c)  =2ab+2ac+2bc-{a}^2-{b}^2-{c}^2$.
\end{sg}

The proof of SG~\ref{sg:  relation between degree and delta} relies on the fact that the Hilbert series of  $R/\ideal{F,G,H}$
	(and, consequently, $\deg \ideal{F,G,H}$) can be calculated from the free resolution of $R/\ideal{F,G,H}$ appearing in 
	Definition~\ref{defn: syzygy gap definition}.  
Though we omit the proof of SG~\ref{sg: relation between degree and delta}, we use a similar idea to establish the following identity.

\begin{lem}\label{lem: deg <F,G> = deg F deg G}
	Let $U, V\in \kk[x,y]$ be relatively prime forms.
	Then $\deg \ideal{U,V} = \deg U \deg V$.  
\end{lem}

\begin{proof}
	Set $u=\deg U$ and $v=\deg V$.  
	As $U$ and $V$ are relatively prime, the sequence
		\[ 0 \to R(-u-v) \to R(-u) \oplus R(-v) \to R \to R/\ideal{U,V} \to 0,\] 
		in which the second map is given by $1 \mapsto (V,-U)$ and the third map by $(A,B) \mapsto AU+BV$, is exact.  
	If $\operatorname{Hilb}(t)$ denotes the Hilbert series of $R/\ideal{U,V}$, then the above exact sequence shows that 
		\[ \operatorname{Hilb}(t) = \frac{1-t^{u}-t^{v} + t^{u+v}}{(t-1)^2}.\]  
	Applying l'H\^opital's rule twice, we find that 
		$\operatorname{Hilb}(1) = uv=\deg U \deg V$, and the lemma follows, as 
		$\deg \ideal{U,V} = \dim_{\kk} R/\ideal{U,V} = \operatorname{Hilb}(1)$.
\end{proof}

\begin{cor}\label{cor: restated relation between degree and delta}
	If $F$ and $G$ are relatively prime, then 
		\[\delta(F,G,H)^2 = 4 ( \deg \ideal{F,G,H} - \deg \ideal{F,G} ) + ( \deg H - \deg FG)^2.\]
\end{cor}

\begin{proof}
	Standard algebraic manipulations of SG~\ref{sg: relation between degree and delta} produce the identity 
		\[ \delta(F,G,H)^2 = 4 ( \deg \ideal{F,G,H} - \deg F \deg G ) + ( \deg H - \deg FG )^2,\] 
		and the claim then follows from Lemma~\ref{lem: deg <F,G> = deg F deg G}.
\end{proof}

\begin{sg}[{\cite[Remark~2.6]{sg1}, \cite[Lemma~2(1)]{monsky.mason}}]\label{sg: triangle ineqs}
	If $\deg H\ge \deg F +\deg G$ and $F$ and  $G$ are relatively prime,  then   $\delta(F,G,H)=\deg H -\deg F -\deg G.$
\end{sg}

\begin{sg}[{\cite[Proposition~2.7(2)]{sg1}, \cite[Lemma~2(2)]{monsky.mason}}]\label{sg: factor prime to H}   
	If a form $P\in \kk[x,y]$ is prime to $H$, then  $\delta(PF,PG,H)=\delta(F,G,H)$.
\end{sg}

\begin{sg}[{\cite[Equation~(2)]{sg1}, \cite[Lemma~2(3)]{monsky.mason}}]\label{sg: inserting a linear form}
	If $\ell\in \kk[x,y]$ is a linear form, then $\delta(F,G,H\ell)=\delta(F,G,H)\pm 1$.
\end{sg}

\begin{sg}[{\cite[Proposition~2.12]{sg1}}]\label{sg: 2Dconv}
	Let $\ell_{1}$ and $\ell_{2}$ be relatively prime linear forms, such that $F$, $G$ and $H\ell_{1}\ell_{2}$ have no common factor. 
	Suppose that $\delta(F,G,H)=\delta(F,G,H\ell_{1}\ell_{2})$ and $\delta(F,G,H\ell_{1})=\delta(F,G,H\ell_{2})$. 
	Then either $\delta(F,G,H)=0$ or $\delta(F,G,H\ell_{1})=0$.
\end{sg}

While the above results hold in arbitrary characteristic, from this point on the assumption that $\kk$ is a field of positive 
	characteristic $p$ will become essential.
Due to the flatness of the Frobenius map over $\kk[x,y]$, we have the following:

\begin{sg}[{\cite[Equation~(3)]{sg1}}]\label{sg: char p}
	For each $q=p^e$ we have 
		\[\delta(F^q,G^q,H^q)=q\cdot \delta(F,G,H).\]
\end{sg}

\begin{defn}
	Let $\ell_1,\ldots,\ell_n\in \kk[x,y]$ be pairwise prime linear forms.
	A \emph{cell} (with respect to $\ell_1,\ldots,\ell_n$)  is a triple of forms $C=\cell{F,G,H}$ 
		such that $F$, $G$, and $H\ell_1\cdots\ell_n$ have no common factor. 
	If $C=\cell{F,G,1}$, we shall dispense with the third component and simply write $C=\cell{F,G}$.	
\end{defn}

In the remainder of this section, $\ell_1,\ldots,\ell_n\in \kk[x,y]$ are fixed pairwise prime linear forms, $\ell=(\ell_1,\ldots,\ell_n)$, 
	and $C$ is a cell $\cell{F,G,H}$ with respect to the $\ell_i$.

\begin{notation}
	For each $\vv{a}=(a_1,\ldots,a_n)\in \NN^n$,  $\ell^\vv{a}$ denotes the product $\ell_1^{a_1}\cdots\ell_n^{a_n}$.
\end{notation}

\begin{defn}
	The \emph{syzygy gap fractal} $\delta_C: \infint^n\to\mathbb{Q}$ is defined as follows:
	\begin{equation*}
		\delta_C\left(\frac{\vv{a}}{q}\right)=\frac{1}{q}\cdot \delta(F^q,G^q,H^q\ell^\vv{a}),
	\end{equation*}
	for each $q$ and each  $\vv{a}\in \NN^n$.
	(SG~\ref{sg: char p} shows that this is well defined.)
\end{defn}

In \cite{sg1} these functions were defined on $\unitint^n$, but it will be convenient in this paper to extend them to $\infint^n$.

\begin{notation}
	The  \emph{taxicab metric} and \emph{norm} on $\RR^n$ are denoted by $d$ and $\norm{\cdot}$.
	That is, $d(\vv{t},\vv{u})=\sum_{i=1}^n\abs{t_i-u_i}$ and $\norm{\vv{t}}=\sum_{i=1}^n\abs{t_i}$, for each $\vv{t},\vv{u}\in\RR^n$.
\end{notation}

SG~\ref{sg: inserting a linear form} gives us the following result:

\begin{sg}[{\cite[Proposition~4.2]{sg1}}]\label{sg: lipschitz}
	For each $\vv{t},\vv{u}\in \infint^n$ we have
		\[\abs{\delta_C(\vv{t})-\delta_C(\vv{u})}\le d(\vv{t},\vv{u}).\]  
\end{sg}
 
This shows that $\delta_C$ is uniformly continuous, so it extends (uniquely) to a continuous function $\RRnn^n\to \RR$.
\textbf{Henceforth, $\boldsymbol{\delta_C}$ will denote this extension.}
 		
The next three results were stated in \cite{sg1} for the original $\delta_C$, defined on $\unitint^n$,  but also hold for the extension to $\infint^n$ 
	(with identical proofs, with one exception noted below) and extend to $\delta_C:\RRnn^n\to \RR$, via density and continuity.

\begin{sg}[{\cite[Proposition~3.4]{sg1}}]\label{sg: simple cells}	
	For each cell $\cell{F,G,H}$, there exists a cell $\cell{U,V}$ such that $\delta_{\cell{F,G,H}}=\delta_{\cell{U,V}}$.
\end{sg}

\begin{defn}
	Two points of $(\QQnn)_q^n$ are \emph{adjacent} if they differ by 
		$\pm\canvec_i/q$, for some $i$, where $\canvec_1,\ldots,\canvec_n$ denote the canonical basis vectors.
	Equivalently, two points of $(\QQnn)_q^n$ are adjacent if the taxicab distance between them is $1/q$. 
\end{defn}

\begin{sg}[{\cite[Theorem~II]{sg1}}]\label{sg: local max}
	Suppose the restriction of  $\delta_C$ to $(\QQnn)_q^n$ attains a local maximum at  $\vv{u}_0$, in the sense that the values 
		of $\delta_C$ at all  points of $(\QQnn)_q^n$ adjacent to $\vv{u}_0$ are  smaller  than $\delta_C(\vv{u}_0)$.  
	Then 
		\[\delta_C(\vv{t})=\delta_C(\vv{u}_0)-d(\vv{t},\vv{u}_0),\]
		for all $\vv{t}\in \RRnn^{n}$ with $d(\vv{t},\vv{u}_0)\le \delta_C(\vv{u}_0)$.  
	In particular, $\delta_C$ is piecewise linear on that region and has a local maximum at $\vv{u}_0$ in the usual sense. 
\end{sg}

The following result was first obtained by Monsky in the case where $C=\cell{x,y}$ \cite[Corollary~9]{monsky.mason}, 
	and subsequently generalized by the second author:

\begin{sg} [{\cite[Theorem~III]{sg1}}]\label{sg: upper bound}
	Suppose $\delta_C$ has a local maximum at $\vv{a}/q$, where $q>1$ and $\vv{a}\in \NN^n$ has some coordinate not divisible by $p$. 	
	Then    
		\[\delta_C\left(\frac{\vv{a}}{q}\right)\le \frac{n-2}{q}.\] 
\end{sg}

As the proof of \cite[Theorem~III]{sg1} uses the fact that the $\delta_C$ are defined on a unit hypercube in an essential way, 
	relying on symmetry and reflections, this requires explanation. 
Let $\vv{a}/q$ be as in the above statement. 
Set $\vv{b}=\down{\vv{a}/q}$ and $C_\vv{b}=\cell{F,G,H\ell^\vv{b}}$.
Then $\delta_C(\vv{t})=\delta_{C_\vv{b}}(\vv{t}-\vv{b})$ for each $\vv{t}$ with $\vv{t}-\vv{b}\in \RRnn^n$, 
	so $\delta_{C_\vv{b}}$ has a local maximum at $\vv{a}/q-\vv{b}$.
Since some $a_i$ is prime to $p$, so is  the corresponding numerator $a_i-b_iq$ of $\vv{a}/q-\vv{b}$. 
Thus, the hypotheses of \cite[Theorem~III]{sg1} hold for $\delta_{C_\vv{b}}$ at the point $\vv{a}/q-\vv{b}\in [0,1]_q^n$, 
	and that result shows that $\delta_C(\vv{a}/q)=\delta_{C_\vv{b}}(\vv{a}/q-\vv{b})\le(n-2)/q$.

\begin{rem}\label{rem: trivial region}
	The function $\delta_C$ is linear outside a bounded subset of $\RRnn^n$.
	To see that, we may assume that $C=\cell{U,V}$, by SG~\ref{sg: simple cells}, and SG~\ref{sg: factor prime to H} allows us to assume that 
		$U$ and $V$ are relatively prime;
		SG~\ref{sg: triangle ineqs} then shows  that $\delta_{C}(\vv{a}/q)=\norm{\vv{a}/q}-\deg UV$ whenever  $\norm{\vv{a}/q}\ge \deg UV$.  
	By continuity, $\delta_C(\vv{t})=\norm{\vv{t}}-\deg UV$, for each $\vv{t}\in \RRnn^n$  with $\norm{\vv{t}}\ge \deg UV$.
\end{rem}

\begin{defn}
	The \emph{trivial region} of $\delta_C$ is the set $\{\vv{t}\in \RRnn^n : \norm{\vv{t}}\ge \deg UV\}$, where 
		$U$ and $V$ are forms such that $\delta_C=\delta_{(U,V)}$ (see SG~\ref{sg: simple cells}).
 	The complement of the trivial region in $\RRnn^n$ is the \emph{nontrivial region} of $\delta_C$.
\end{defn}

We close  this section with some consequences of SG~\ref{sg: local max}.

\begin{prop}
	Suppose $\vv{t}_0$ lies in the nontrivial region of $\delta_C$ and $\delta_C(\vv{t}_0)>0$.
	Then $\delta_C$ attains a local maximum at a point $\vv{u}_0\in \QQ_{p^\infty}^n$ with $d(\vv{t}_0,\vv{u}_0)<\delta_C(\vv{u}_0)$.
	In particular,  $\delta_C(\vv{t})= \delta_C(\vv{u}_0)-d(\vv{t},\vv{u}_0)$ on a neighborhood of $\vv{t}_0$.
\end{prop}

\begin{proof}
	Using the continuity of the function $\vv{t}\mapsto \delta_C(\vv{t})-d(\vv{t},\vv{t}_0)$ and the density of $\QQ_{p^\infty}^n$ in $\RR^n$, 
		choose $\vv{u}$ in some $\QQ_q^n$  such that $\delta_C(\vv{u})-d(\vv{u},\vv{t}_0)>0$.
	Then $\delta_C(\vv{u})>d(\vv{u},\vv{t}_0)\ge 0$, and SG~\ref{sg: lipschitz} shows that $\delta_C>0$ on the line segment joining~$\vv{u}$ and 
		$\vv{t}_0$.
	Since $\delta_C$ vanishes on the boundary of its trivial region (see Remark~\ref{rem: trivial region}), 
		this shows that~$\vv{u}$ is also in the nontrivial region.
	Consider the equivalence relation on the set $S=\{\vv{v}\in (\QQnn)_q^n:\delta_C(\vv{v})>0\}$ generated by adjacency---that is, 
		the smallest equivalence relation on $S$ containing all adjacent pairs $(\vv{v},\vv{v}')\in S^2$. 
	We shall show that the equivalence class of $\vv{u}$ is contained in the nontrivial region, and is therefore finite.	
	For that, it suffices to show that points  $\vv{z},\vv{w}\in S$, one in the trivial region and the other in the nontrivial region,  
		cannot be adjacent.	
	Indeed, if they were adjacent, then $\abs{\delta_C(\vv{z})-\delta_C(\vv{w})}= 1/q=d(\vv{z},\vv{w})$, by SG~\ref{sg: inserting a linear form}, 
		and SG~\ref{sg: lipschitz} would imply that $\delta_C$ is linear on the line segment joining  $\vv{z}$ and $\vv{w}$.
	This is impossible, since $\delta_C(\vv{z})$ and $\delta_C(\vv{w})$ are both 
		positive, and $\delta_C$ vanishes on the boundary of the trivial region.

	Since the equivalence class of $\vv{u}$ is finite, we can choose a point   $\vv{u}_0$ in this class where $\delta_C$ is maximum.
	Then SG~\ref{sg: local max} shows that $\delta_C$ attains a local maximum at $\vv{u}_0$ and that $\delta_C(\vv{u})=
		\delta_C(\vv{u}_0)-d(\vv{u},\vv{u}_0)$.
	To complete the proof, note that $d(\vv{t}_0,\vv{u}_0)\le d(\vv{t}_0,\vv{u})+d(\vv{u},\vv{u}_0)<
		\delta_C(\vv{u})+(\delta_C(\vv{u}_0)-\delta_C(\vv{u}))=\delta_C(\vv{u}_0)$.
\end{proof}

\begin{rem}\label{rem: don't need to go to finer mesh to find local max}
	The above proof shows that if $\vv{u}\in \QQ_q^n$ lies in the nontrivial region of $\delta_C$ and $\delta_C(\vv{u})>0$, 
		then the local maximum $\vv{u}_0$ that determines the behavior of  $\delta_C$ near $\vv{u}$ \emph{is also in $\QQ_q^n$}.
\end{rem}

Two corollaries follow immediately: 

\begin{cor}
	The function $\delta_C$ is piecewise linear with coefficients in $\QQ_{p^\infty}$ on each connected component of 
		its positive locus.
\qed	
\end{cor}

\begin{cor}\label{cor: local maxima have coords in ZZ[1/p]}
	The local maxima of $\delta_C$ are attained at points in $\QQ_{p^\infty}^n$.
\qed	
\end{cor}

\section{\texorpdfstring{The functions $\bm{\Dlb}$ and $\bm{\Plb}$; the upper and lower regions}
	{The functions Delta and Phi; the upper and lower regions}}

\label{s:  Delta and Phi section}

Throughout this and the next two sections,  $\ell$ is an $n$-tuple $(\ell_1,\ldots,\ell_n)$ 
	of pairwise prime linear forms in $\kk[x,y]$, and $\idealb$ is an ideal of $\kk[x,y]$ generated by relatively prime non-constant forms $U$ and $V$.
	
\begin{defn}\label{defn: Phi and Delta definition}
	We use $\Dlb$ to denote $\delta_{\cell{U,V}}$, the syzygy gap fractal associated with the cell $\cell{U,V}$ with respect to the linear 
		forms $\ell_1, \ldots, \ell_n$.  
	That is, $\Dlb$ is the unique continuous function $\RR^n_{\geq 0} \to \RR$ such that 
		\[ \Dlb \left(   \frac{\vv{a}}{q} \right) = \frac{1}{q} \cdot \delta( U^q, V^q, \ell^{\vv{a}} ), \] 
		for each $q$ and for each $\vv{a} \in \NN^n$.  
	We use $\Plb$ to denote the unique continuous function $\RRnn^n\to \RR$ such that
		\begin{equation}\label{eq: definition of Phi}
			\Plb\left(\frac{\vv{a}}{q}\right)=\frac{1}{q^2}\cdot\deg\ideal{U^q,V^q,\ell^\vv{a}},
		\end{equation}
		for each $q$ and for each $\vv{a}\in \NN^n$.
\end{defn}

Because $\kk[x,y]$ is regular of dimension 2, for each ideal $\ideala$ of $\kk[x,y]$
	we have  $\deg\ideala^{[p]}=p^2\cdot \deg \ideala$.
Thus, $\Plb(\vv{a}/q)$ is well defined, and \eqref{eq: definition of Phi} gives us a function $\infint^n\to \QQ$.
Below we shall justify the existence of the continuous extension of this function to $\RRnn^n$, tacitly assumed above.  
Furthermore, we shall see that $\Dlb$ and $\Plb$ are both independent of the choice of the two homogeneous 
	generators of the ideal $\idealb$, so they can be thought of as functions attached to~$\ell$ and~$\idealb$. 

\begin{lem}
	There exists a unique continuous extension to $\RRnn^n$ of the function $\Plb:\infint^n\to \QQ$ defined by 
		\eqref{eq: definition of Phi}.
	Furthermore, both $\Dlb$ and $\Plb$ depend only on the ideal $\idealb$, and not on the particular 
		choice of generators $U$ and $V$.
\end{lem}

\begin{proof}
	Corollary~\ref{cor: restated relation between degree and delta} shows that   $\Dlb$ and $\Plb$ are related as follows on $\infint^n$:
	\begin{equation}\label{eq: relation between Delta and Phi}
		\Dlb(\vv{t})^2=4 ( \Plb(\vv{t}) - \deg \idealb ) + ( \norm{\vv{t}} - \deg UV )^2.
	\end{equation}
	Since $\Dlb$ is defined and continuous on $\RRnn^n$, \eqref{eq: relation between Delta and Phi} can be used to 
		(uniquely) extend $\Plb$ to a continuous function $\RRnn^n\to \RR$, establishing the first claim. 
	To prove the second claim, first note that $\deg UV$ depends only on $\idealb$, and not on the particular choice of generators 
		$U$ and $V$: as $\sqrt{\idealb} = \ideal{x,y}$, any two forms generating $\idealb$ must be a minimal set of 
		generators; consequently, their degrees are univocally determined by~$\idealb$.  
	Next, observe that, as the restriction of $\Plb$ to $\infint^n$ is clearly independent of the choice of $U$ and $V$, 
		so is $\Plb$ itself, by continuity.
	Finally, since both $\Plb$ and $\deg UV$ are independent of the choice of the generators $U$ and $V$, then so is $\Dlb$, 
		by~\eqref{eq: relation between Delta and Phi}.
\end{proof}

\begin{exmp}\label{exmp: Phi for n=2}
	The function $\Plm$ attached to $\ell=(x,y)$ and $\idealm=\ideal{x,y}$ is as follows:
		\begin{equation*}
			\Plm\left(\frac{\vv{a}}{q}\right)=\frac{1}{q^2}\cdot \deg\ideal{x^q,y^q,x^{a_1}y^{a_2}}=
				\begin{cases}
					\frac{a_1}{q}+\frac{a_2}{q}-\frac{a_1a_2}{q^2}&\text{if }a_1,a_2<q\\
					1&\text{otherwise}
				\end{cases}
		\end{equation*}
		for each $q$ and each $\vv{a}\in \NN^2$; so
		\begin{equation*}
			\Plm\left(\vv{t}\right)=
				\begin{cases}
					t_1 + t_2 - t_1 t_2&\text{if }t_1,t_2<1\\
					1&\text{otherwise}
				\end{cases}
		\end{equation*}
	for each $\vv{t}\in \RRnn^2$, by continuity.
\end{exmp}

\begin{defn}\label{defn: order}
	For $\vv{u},\vv{v}\in \RR^n$ we write $\vv{u}\le \vv{v}$ if $u_i\le v_i$, for each $i$.
	The relations $\ge$, $<$, and $>$ on $\RR^n$ are defined likewise.  
\end{defn}

\begin{prop}[Basic properties of $\Plb$]\label{prop: basic properties of Phi}\ 
	\begin{enumerate}[(1)]
		\item \label{item: Dlb weakly increasing}The function $\Plb$ is \textup(weakly\textup) increasing\textup: 
				$\vv{t}\le \vv{u}\ \Rightarrow\ \Plb(\vv{t})\le \Plb(\vv{u})$.
		\item $0\le \Plb(\vv{t})\le \deg \idealb$, for each $\vv{t}$.
		\item \label{item: description of Dlb on upper region}	
			$\Plb(\vv{t})=\deg \idealb\ \Leftrightarrow\ \Dlb(\vv{t})=\abs{\norm{\vv{t}}-\deg UV}$.
			In particular, $\Plb(\vv{t})= \deg \idealb$ for each $\vv{t}$ such that $\norm{\vv{t}}\geq\deg UV$.
	\end{enumerate}
\end{prop}

\begin{proof}
	Points (1) and (2) are clear for $\vv{t},\vv{u}\in\infint^n$, and extend to all $\vv{t},\vv{u}\in \RRnn^n$
		by continuity.  
	The first statement in (3) follows immediately from~\eqref{eq: relation between Delta and Phi}, while the second statement
		follows from the first and Remark~\ref{rem: trivial region}, which states that $\Dlb(\vv{t}) = \norm{\vv{t}} - \deg UV$ 
		whenever $\norm{\vv{t}} \geq \deg UV$.
\end{proof}

\begin{defn}
	The set $\Trivial=\bigl\{\vv{t}\in \RRnn^n: \norm{\vv{t}}\ge \deg UV\bigr\}$ is the \emph{trivial region} attached to $\idealb$, 
		and its complement in $\RRnn^n$ is the \emph{nontrivial region}.
\end{defn}

\begin{convention}
	All topological notions used will refer to the subspace topology of $\RRnn^n$ induced  by the standard topology of $\RR^n$.
\end{convention}

\begin{notation}
	If $X\subseteq \RRnn^n$, then  $\closure{X}$ and $\boundary X$ denote the closure and the boundary of $X$ in $\RRnn^n$.
\end{notation}

\begin{defn}
	The \emph{upper} and \emph{lower regions} attached to $\ell$ and $\idealb$ are the sets 
		\[\Upper= \bigl\{ \vv{t} \in \RR^n_{\geq 0} : \Plb(\vv{t}) = \deg \idealb \bigr\} = \Plb^{-1}(\deg \idealb )\]  
		and 
		\[ \Lower= \bigl\{ \vv{t} \in \RR^n_{\geq 0} : \Plb(\vv{t}) < \deg \idealb \bigr\} = \Plb^{-1}([0,\deg\idealb)).\]
	The set $\Boundary$ is the common boundary $\boundary\Upper=\boundary \Lower$ of those regions in $\RRnn^n$. 
\end{defn}

Very often the choice of $\ell$ and $\idealb$ will be clear from the context (or fixed in advance, as in this section), so we shall 
	omit the phrase ``attached to $\ell$ and $\idealb$'' and ask that the reader rely on the context to determine the exact setup. 

\begin{rem}\label{rem: regions for pts with coords in ZZ[1/p]}
	If $\vv{a} \in \NN^n$ and $q$ is a power of $p$, then 
		\[\vv{a}/q\in \Upper\ \Leftrightarrow\ \ell^\vv{a}\in \idealb^{[q]}\quad\text{and}\quad
			\vv{a}/q\in \Lower\ \Leftrightarrow\ \ell^\vv{a}\not\in \idealb^{[q]}.\]
\end{rem}

\begin{exmp}\label{exmp: regions for case n=2}
	In the setting of Example~\ref{exmp: Phi for n=2}, the lower region $\Lower$ is the square $[0,1)^2$.
	Less trivial instances can be seen in Examples~\ref{exmp: 2D example}, \ref{exmp: some sierpinski staircases}, and 
		\ref{exmp: computation of fpts}.
\end{exmp}

\begin{defn}
	Let $\vv{u},\vv{v}\in \RR^n$.
	Then 
		\[[\vv{u},\vv{v}]\coloneqq \{\vv{t}\in \RR^n: \vv{u}\le \vv{t}\le \vv{v}\}=[u_1,v_1]\times\cdots\times [u_n,v_n].\] 
	The ``intervals'' $(\vv{u},\vv{v})$, $[\vv{u},\vv{v})$, $[\vv{u},\vvv{\infty})$, \etc, are defined analogously. 
\end{defn}

Some properties of the regions $\Upper$, $\Lower$, and $\Boundary$ follow immediately from 
	Proposition~\ref{prop: basic properties of Phi} and the continuity of $\Plb$:

\begin{cor}[Basic properties of the regions $\Upper$, $\Lower$, and $\Boundary$]\label{cor: properties of the regions}\ 
	\begin{enumerate}[(1)]
		\item $\Lower$ is open and $\Upper$ and $\Boundary$ are closed in $\RRnn^n$.
		\item $\Upper$ contains the trivial region $\Trivial = \bigl\{ \vv{t} \in \mathbb{R}^n_{\geq 0} : \norm{ \vv{t} } \geq \deg UV \bigr\}$.
		\item $\Lower$ is contained in the nontrivial region, and is therefore bounded.
		\item If $\vv{u}\in \Upper$, then $[\vv{u},\vvv{\infty})\subseteq \Upper$.
		\item\label{item: pts below a point in lower region}
			If  $\vv{u}\in \Lower$, then $[\vv{0},\vv{u}]\subseteq \Lower$.
		\item\label{item: pts above and below a boundary point} 
			If $\vv{u}\in \Boundary$, then $[\vv{0},\vv{u})\subseteq\Lower$ and $[\vv{u},\vvv{\infty})\subseteq \Upper$.
		\qed
	\end{enumerate}
\end{cor}

\begin{rem}
	The lower region, in the case where $\idealb=\ideal{x,y}$, was studied by P\'erez \cite{perez.constancy-regions} 
		under a different guise---as the first constancy region of the mixed test ideals $\tau(\ell^\vv{t})$.
\end{rem}

\section{\texorpdfstring{The $\bm F$-threshold function}{The F-threshold function}}\label{s: F-threshold function}

Let $\ell=(\ell_1,\ldots,\ell_n)$ and $\idealb=\ideal{U,V} \subseteq \kk[x,y]$ be as in the previous section. 
As before, we use $\Dlb$ and $\Plb$ to denote the unique continuous functions $\mathbb{R}^n_{\geq 0} \to \mathbb{R}$ such that 
	$\Dlb(\vv{a}/q) = q^{-1}\cdot \delta(U^q, V^q, \ell^{\vv{a}})$ and $\Plb(\vv{a}/q) = 
	q^{-2}\cdot \deg \ideal{U^q, V^q, \ell^{\vv{a}}}$, for every $\vv{a} \in \NN^n$ and for every $q=p^e$.

\begin{discussion} 
\label{discussion:  FT description for polynomials} 
	The \emph{$F$-threshold} of a polynomial $G\in \idealm=\ideal{x,y}$ with respect to~$\idealb$ can be defined as follows: 
		\[\ftb(G)=\inf\bigl\{k/q\in (\QQpos)_{p^\infty}:G^k\in \idealb^{[q]}\bigr\}.\]
	Though it is not at all obvious from this definition, it turns out that $\ftb(G)$ is a rational number 
		(see \cite[Corollary~2.30, Theorem~3.1]{blickle+mustata+smith.discr_rat_FPTs}).
	When $\idealb=\idealm$, $\ftb(G)$ is the \emph{$F$-pure threshold} of $G$, denoted by $\fpt(G)$.

	The condition $G^k\in \idealb^{[q]}$ is equivalent to $\deg\ideal{U^q,V^q,G^k}=\deg\ideal{U^q,V^q}$ 
		(where $\deg\ideal{U^q,V^q}=\deg\idealb^{[q]}=q^2\cdot\deg \idealb$), so in the 
		special case where $G=\ell^{\vv{a}}$, for some $\vv{a}\in\NNpos^n$, we find that 
		\[\ftb(\ell^\vv{a})=\inf\bigl\{\tfrac{k}{q}\in (\QQpos)_{p^\infty}:\Plb\bigl(\tfrac{k}{q}\cdot \vv{a}\bigr)=\deg\idealb\bigr\}.\]
	The continuity of $\Plb$ and the density of $(\QQpos)_{p^\infty}$ in $\RRpos$ then allow us to write
		\[\ftb(\ell^\vv{a})=\min\left\{\lambda\in \RRpos:\Plb(\lambda\vv{a})=\deg\idealb\right\}= \min\{ \lambda\in \RRpos :\lambda\vv{a}\in \Upper\},\]
		which motivates the following definition.
\end{discussion}

\begin{defn}\label{defn: F-threshold function}
	For each $\vv{t}\in \RRpos^n$ we define
		\begin{equation*}
			\ftlb{\vv{t}}=\min\bigl\{\lambda \in \RRpos: \Plb(\lambda  \vv{t})=\deg\idealb\bigr\} = 
				\min \bigl\{ \lambda \in \RRpos : \lambda \vv{t} \in \Upper \bigr\}.
		\end{equation*}
	These descriptions agree, by definition of the upper region $\Upper$, and the latter description shows that 
		$\ftlb	{\vv{t}}$ is well defined: as $\Upper$ is closed and its complement $\Lower$ is bounded (see 
		Corollary~\ref{cor: properties of the regions}), the minima appearing in these definitions exist.
	The function $\vv{t}\mapsto \ftlb{\vv{t}}$ is the \emph{$F$-threshold function} attached to $\ell$ and $\idealb$.
\end{defn}

The following are alternate characterizations of $\ftlb{\vv{t}}$: 
	\begin{align*}
		\ftlb{\vv{t}}&=\min\bigl\{\lambda \in \RRpos: \Dlb(\lambda\vv{t})=\abs{\lambda\norm{\vv{t}}-\deg UV}\bigr\}\\
		&=\max\bigl\{\lambda \in \RRpos: \lambda\vv{t}\in\closure{\Lower}\bigr\}\\
		&=\text{the unique $\lambda \in \RRpos$ such that $\lambda\vv{t}\in \Boundary$.}
	\end{align*}
The last characterization, which depends on Corollary~\ref{cor: properties of the regions}\iref{item: pts above and below a boundary point}, 
	is perhaps the best way to think of the $F$-threshold function: $\ftlb{\vv{t}}$ is the exact factor by which  
	$\vv{t}$ needs to be scaled to obtain a point of the boundary~$\Boundary$.

\begin{rem} \label{rem: natural ft upper bound}
	Fix $\vv{t} \in \RR^n_{>0}$.  The observation that $\frac{ \deg UV }{\norm{\vv{t}}} \cdot \vv{t}$ has norm $\deg UV$ 
		(and hence lies in the trivial region $\Trivial$, and therefore in $\Upper$) shows that the factor by which $\vv{t}$ 
		needs to be scaled to obtain a  point of the boundary $\Boundary$ is at most $\deg UV/\norm{\vv{t}}$.  
	Restated, we see that 
		\[  \ftlb{\vv{t}} \leq \frac{ \deg UV }{\norm{\vv{t}}}.\] 
	In future sections we shall try to understand in what ways the value of $\ftlb{\vv{t}}$ differs from this natural upper bound.
\end{rem}

\begin{exmp}
	In the situation of Examples~\ref{exmp: Phi for n=2} and~\ref{exmp: regions for case n=2} ($\idealm=\ideal{x,y}$ and $\ell=(x,y)$), 
		the closure of the lower region is the unit square $[0,1]^2$, so 
		\begin{equation*}
			\ftlm{\vv{t}}=\max\bigl\{\lambda \in \RRpos: \lambda\vv{t}\in[0,1]^2\bigr\}=\min\left\{\frac{1}{t_1},\frac{1}{t_2}\right\}.
		\end{equation*}
\end{exmp}

Directly from the definition we have:

\begin{prop}\label{prop: scaling}
	For each $\lambda\in \RRpos$ and $\vv{t}\in \RRpos^n$ we have
		\[\ftlb{\lambda \vv{t}}=\frac{1}{\lambda}\cdot \ftlb{\vv{t}}.\qedhere\]
\end{prop}

The above result and the rationality of $F$-thresholds of polynomials imply that $\ftlb{\vv{t}}$ is rational whenever $\vv{t}\in \QQpos^n$.

\begin{prop}
	The $F$-threshold function is continuous.
\end{prop}

\begin{proof}
	To show continuity at $\vv{c}\in \RRpos^n$, Proposition~\ref{prop: scaling} allows us to scale~$\vv{c}$ and assume that 
		$\vv{c}\in \Boundary$ or, equivalently, $\ftlb{\vv{c}}=1$.
	If $\vv{t}\in \RRpos^n$, then the factors by which $\vv{t}$ needs to be scaled to obtain points of $\boundary [\vv{0},\vv{c}]$
		and $\boundary [\vv{c},\vvv{\infty})$ are, respectively, $\min\{c_i/t_i\}$ and $\max\{c_i/t_i\}$.
	By Corollary~\ref{cor: properties of the regions}\iref{item: pts above and below a boundary point}, 
		the scaling factor needed to obtain a point of $\Boundary$ lies somewhere in between: 
		\begin{equation}\label{eq: inequalities in proof of continuity of FT function}
			\min\left\{\frac{c_1}{t_1},\ldots,\frac{c_n}{t_n}\right\}\le \ftlb{\vv{t}}\le  \max\left\{\frac{c_1}{t_1},\ldots,\frac{c_n}{t_n}\right\}.
		\end{equation}
	(Figure~\ref{fig: illustration for continuity proof}  illustrates this for $n=2$.)
	\begin{figure}
		\centering
		\includegraphics[height=1.7in]{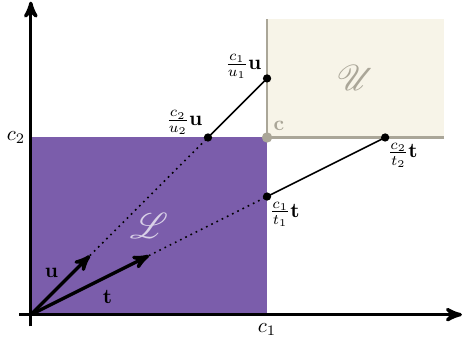}
		\caption{An illustration of inequalities (\ref{eq: inequalities in proof of continuity of FT function}) in the case $n=2$}
		\label{fig: illustration for continuity proof}
	\end{figure}
	As $\vv{t}\to \vv{c}$, both $\min\{c_i/t_i\}$ and $\max\{c_i/t_i\}$ tend to 1, so $\ftlb{\vv{t}}\to 1=\ftlb{\vv{c}}$, 
		showing continuity at~$\vv{c}$.
\end{proof}

\begin{notation}
	Let $\vv{x}\in \RRnn^n$ and $r\in \RRpos$. 
	Then $\oball{r}{\vv{x}}$ and $\cball{r}{\vv{x}}$ denote the open and closed taxicab balls (in $\RRnn^n$) of radius $r$ centered at $\vv{x}$.
\end{notation}

\begin{exmp}\label{exmp: 2D example}
	Let $\idealb=\ideal{x^3+y^3+xy^2,x^2y^3} \subseteq \FF_5[x,y]$ and  $\ell=(x,y)$.
	Figure~\ref{fig: 2D example}(a) shows a density plot of the function $\Dlb$ attached to $\ell$ and $\idealb$ on the square $[0,8]^2$, 
		where the color encodes the values of the function at each point---the lighter the color, the larger the value.
	Figure~\ref{fig: 2D example}(b) shows the corresponding upper and lower regions.
	\begin{figure}[!t]
		\centering		\subfloat[]{

			\includegraphics[height=2in]{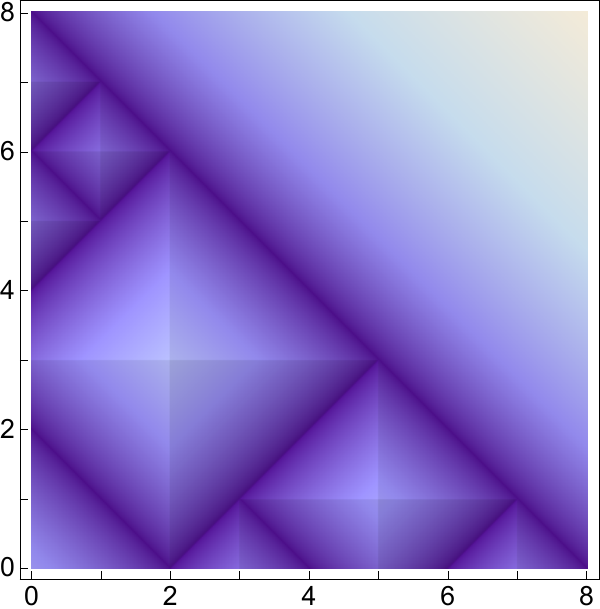}
		}
		\hskip 5mm
		\subfloat[]{
			\includegraphics[height=2in]{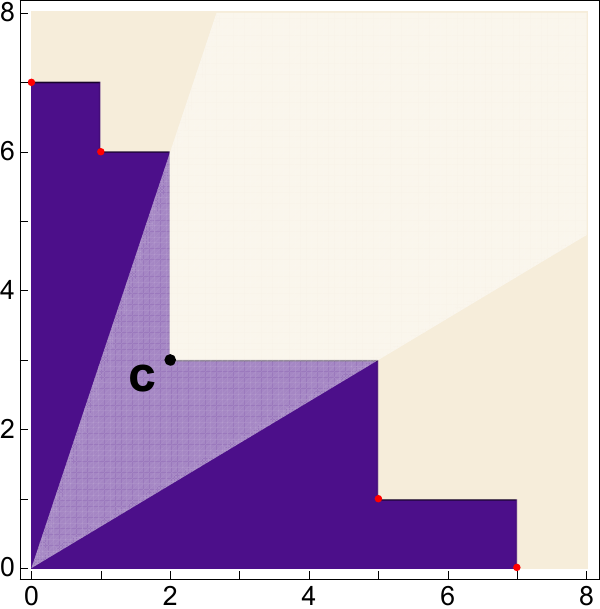}
		}
		\caption{The syzygy gap fractal of Example~\ref{exmp: 2D example} and its upper and lower regions}\label{fig: 2D example}
	\end{figure}
	This is a rather simple example of a syzygy gap fractal, that is entirely determined
		by its values at the points with integer coordinates.
	For instance, as the plot suggests, $\Dlb$ attains a local maximum at $\vv{c}=(2,3)$.
	Indeed, direct computations using SG~\ref{sg: triangle ineqs} and SG~\ref{sg: factor prime to H}
		show that $\Dlb(\vv{c})=\delta(x^3+y^3+xy^2,x^2y^3,x^2y^3)=3$, 
		while $\Dlb=2$ at each adjacent point in $\NN^2$.
	SG~\ref{sg: local max} then confirms that $\Dlb$ has a local maximum at $\vv{c}$, 
		and also shows that $\Dlb(\vv{t})=3-\abs{t_1-2}-\abs{t_2-3}$ on  $\cball{3}{\vv{c}}$.
	Similar formulas can be obtained for each other local maximum in $\NN^2$, and all such formulas put together give an
		explicit piecewise formula for $\Dlb$ on its nontrivial region.
	On the trivial region $\Trivial=\{\vv{t}\in \RRnn^2: \norm{\vv{t}}\ge 8\}$, on the other hand, we know that 
		$\Dlb(\vv{t})=\norm{\vv{t}}-8$.
	
	Returning to $\vv{c}=(2,3)$, note that on $\cball{3}{\vv{c}}$ we have $\Dlb(\vv{t})=8-\norm{\vv{t}}$ if and only if $\vv{t}\ge \vv{c}$.
	Proposition~\ref{prop: basic properties of Phi}\iref{item: description of Dlb on upper region} then shows that 
		$\Upper\cap \cball{3}{\vv{c}}=[\vv{c},\vvv{\infty})\cap \cball{3}{\vv{c}}$, and thus 
		$\Boundary\cap \cball{3}{\vv{c}}=\boundary[\vv{c},\vvv{\infty})\cap \cball{3}{\vv{c}}$.
	Consequently, $\ftlb{\vv{t}}=\max\{2/t_1,3/t_2\}$ on the cone over 
		$\boundary [\vv{c},\vvv{\infty})\cap \cball{3}{\vv{c}}$, which appears shaded in Figure~\ref{fig: 2D example}(b). 
	
	The point $\vv{c}$ plays a special role because it is a local maximum ``adjacent to the trivial region''.
	There are five such points in this example: $\vv{c}_1=\vv{c}=(2,3)$, $\vv{c}_2=(0,7)$,  $\vv{c}_3=(1,6)$, $\vv{c}_4=(5,1)$, and 
		$\vv{c}_5=(7,0)$, all shown in Figure~\ref{fig: 2D example}(b).
	Each of those points determines the $F$-threshold function locally, and together they determine the $F$-threshold function globally: 
		$\Upper=\bigcup_{i=1}^5[\vv{c}_i,\vvv{\infty})$, and therefore
		$\ftlb{\vv{t}}$ equals the minimum value of the set 
		\[ 
			\Bigl\{
				\max\bigl\{\tfrac{2}{t_1},\tfrac{3}{t_2}\bigr\}, 
				\max\bigl\{\tfrac{0}{t_1},\tfrac{7}{t_2}\bigr\}, 
				\max\bigl\{\tfrac{1}{t_1},\tfrac{6}{t_2}\bigr\}, 
				\max\bigl\{\tfrac{5}{t_1},\tfrac{1}{t_2}\bigr\}, 
				\max\bigl\{\tfrac{7}{t_1},\tfrac{0}{t_2}\bigr\}          
			\Bigr\}.
		\] 
\end{exmp}

Special points like those in the above example will be examined in detail in the next section.
As we shall see, Example~\ref{exmp: 2D example} is illustrative of the general situation in  that the lower region has a ``staircase'' aspect
	determined by such special points, and that the $F$-threshold function is determined by those points.
But Example~\ref{exmp: 2D example} is also misleadingly simple.
In more general situations there are typically infinitely many of those special points; very often those points do not have integer coordinates, 
	and $\Upper$ is not the union of boxes $[\vv{c},\vvv{\infty})$ determined by them.

\section{Critical points}
\label{s:  CP section}

Throughout this section we fix $\ell=(\ell_1,\ldots,\ell_n)$ and $\idealb=\ideal{U,V}$ as in Section~\ref{s: F-threshold function}, and again use 
	$\Dlb$ and $\Plb$ to denote the functions $\mathbb{R}^n_{\geq 0} \to \RR$ defined in terms of these choices, 
	as in Definition~\ref{defn: Phi and Delta definition}.

\begin{defn} \label{defn:  critical point definition}
	A point $\vv{c}\in \RRnn^n$ is a \emph{critical point} associated with $\ell$ and $\idealb$ if $\Dlb$ attains a local maximum at $\vv{c}$ 
		and that local maximum is adjacent to the trivial region~$\Trivial$, in the sense that $\Dlb(\vv{c})=\deg UV-\norm{\vv{c}}$
		(so the ball $\cball{\Dlb(\vv{c})}{\vv{c}}$ on which~$\Dlb$ is determined by $\vv{c}$ touches the trivial region; 
		see Figure~\ref{fig: CP diagram}).
	\begin{figure}
		\centering
		\includegraphics[height=2in]{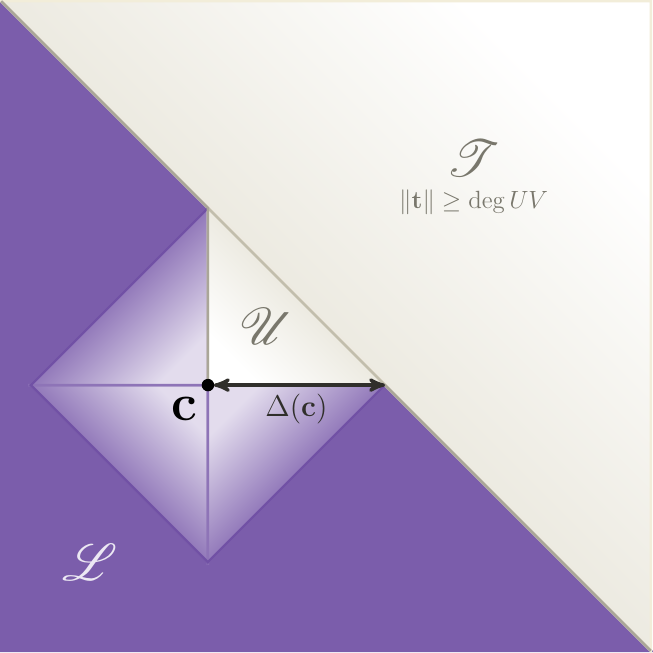}
		\caption{A critical point $\vv{c}$}\label{fig: CP diagram}
	\end{figure}	
\end{defn}

If $\ell$ and $\idealb$ have been fixed (as they were in this section) or are clear from the context, we shall call those points simply \emph{critical points}. 
By Corollary~\ref{cor: local maxima have coords in ZZ[1/p]}, \textbf{all critical points lie in $\bm{\QQ_{p^\infty}^n}$};
	this observation will be fundamental in the proofs of our main results, in Sections~\ref{s:  fpt} and~\ref{s: QH case}.
Critical points forge the (often extremely complex) boundary 
	$\Boundary$ locally---if $\vv{c}$ is a critical point, then $\Boundary$ agrees with the (very simple) boundary of $[\vv{c},\vvv{\infty})$
	near $\vv{c}$, as shown in the next proposition.
Consequently, the $F$-threshold function  agrees with the function $\vv{t} \mapsto \max \{c_1/t_1, \ldots, c_n/t_n\}$ 
	near~$\vv{c}$. 
(Proposition~\ref{prop: CPs vs normalized points} will make precise what ``near'' means.)

\begin{prop}\label{prop: basic properties of CPs}
	Let $\vv{c}$ be a critical point, and set $\radius=\Dlb(\vv{c})$.
	\begin{enumerate}[(1)]
		\item	For $\vv{t}\in \cball{\radius}{\vv{c}}$ we have $\Dlb(\vv{t})=\deg UV-\norm{\vv{t}}$ if and only  if $\vv{t}\ge \vv{c}$.
	\end{enumerate}	
	Consequently, 
	\begin{enumerate}[(1),resume]
		\item\label{item: upper points in critical region} 
			$\Upper\cap\cball{\radius}{\vv{c}}=[\vv{c},\vvv{\infty})\cap\cball{\radius}{\vv{c}}$,
		\item\label{item: boundary points in critical region}
			$\Boundary\cap\cball{\radius}{\vv{c}}=\boundary[\vv{c},\vvv{\infty})\cap \cball{\radius}{\vv{c}}$ 
				\textup(so, in particular, $\vv{c}\in \Boundary$\textup), and 
		\item\label{item: upper points}
			$[\vv{c},\vvv{\infty})\subseteq \Upper$.
	\end{enumerate}
\end{prop}

\begin{proof}
	If $\vv{t}\in \cball{\radius}{\vv{c}}$, then
		$\Dlb(\vv{t})=\Dlb(\vv{c})-d(\vv{t},\vv{c})=\deg UV-\norm{\vv{c}}-\norm{\vv{t}-\vv{c}}$, by SG~\ref{sg: upper bound} and 
		the definition of critical point, so $\Dlb(\vv{t})=\deg UV-\norm{\vv{t}}$ if and only if $\norm{\vv{c}}+\norm{\vv{t}-\vv{c}}=\norm{\vv{t}}$.
	As $\vv{t}$ and $\vv{c}$ have nonnegative coordinates, the latter condition is equivalent to $\vv{t}\ge \vv{c}$, giving (1). 	
	Each $\vv{t}\in  \cball{\radius}{\vv{c}}$ has norm less than or equal to $\deg UV$, so $\vv{t}\in \Upper$ if and only if 
		$\Dlb(\vv{t})=\deg UV-\norm{\vv{t}}$, 
		by Proposition~\ref{prop: basic properties of Phi}\iref{item: description of Dlb on upper region}; parts (2)--(4) now
		follow easily from (1).
\end{proof}

\begin{prop}\label{prop: CPs vs normalized points}
	If $\vv{t}\in \RRpos^n$ and $\vv{c}$ is a critical point, then 
		\[\ftlb{\vv{t}}=\max \left \{ \frac{c_1}{t_1}, \ldots, \frac{c_n}{t_n} \right \} 
			\Longleftrightarrow\ \vv{c} \leq \frac{ \deg UV}{\norm{\vv{t}}} \cdot \vv{t} .\]
\end{prop}

\begin{proof}  
	Note that $\vv{c} \leq (\deg UV/\norm{\vv{t}})\vv{t}$ if and only if $\max \{c_i/t_i\}\le \deg UV/\norm{\vv{t}}$.
	So, in what follows we shall prove the following equivalent assertion: 
		\begin{equation} \label{eq: CPs vs normalized points}
			\ftlb{\vv{t}} =  \max \left \{ \frac{c_1}{t_1}, \ldots, \frac{c_n}{t_n} \right \} \iff  
				\max \left \{ \frac{c_1}{t_1}, \ldots, \frac{c_n}{t_n} \right \} \leq \frac{ \deg UV}{\norm{\vv{t}}}.
		\end{equation}
	Set $\gamma =\max \{c_i/t_i\}$.   
	Assuming $\ftlb{\vv{t}} = \gamma$, Remark~\ref{rem: natural ft upper bound} implies that $\gamma = \ftlb{\vv{t}} \leq \deg UV / \norm{\vv{t}}$.  
	To establish the remaining implication, suppose $\gamma \leq \deg UV / \norm{\vv{t}}$.  
	By definition of~$\gamma$, we have $\gamma\vv{t}\in \boundary [\vv{c}, \vvv{\infty})$. 
	In particular, $\gamma\vv{t} \geq \vv{c}$, and combining this with our assumption that $\gamma \leq \deg UV / \norm{\vv{t}}$ we see that 
		\[ \norm{ \gamma\vv{t} - \vv{c} } = \gamma \norm{\vv{t}} - \norm{\vv{c}} \leq \deg UV - \norm{\vv{c}} = \Dlb(\vv{c}),\] 
		where the last equality follows from the fact that $\vv{c}$ is a critical point.  
	In summary, we have just seen that $\gamma \vv{t} \in \boundary [\vv{c}, \vvv{\infty}) \cap \cball{\Dlb(\vv{c})}{\vv{c}}$, and 
	Proposition~\ref{prop: basic properties of CPs}\iref{item: boundary points in critical region} then shows that $\gamma\vv{t} \in \Boundary$ as well.
	The fact that $\ftlb{\vv{t}} = \gamma$ then follows, as $\ftlb{\vv{t}} $ is the unique factor by which $\vv{t}$ is scaled to 
		obtain an element of $\Boundary$.
\end{proof}

We now move towards the proof of some simple characterizations of critical points.
For that, we need the following elementary result.

\begin{lem}\label{lem: ideal contains a power of maximal ideal}
	Let $F,G\in R\coloneqq\kk[x,y]$ be relatively prime forms of degrees~$a$ and~$b$, respectively.
	Then $\ideal{x,y}^{a+b-1}\subseteq \ideal{F,G}$.
\end{lem}

\begin{proof}
	As $F$ and $G$ are relatively prime, the $\kk$-linear map $R_{b-1}\times R_{a-1}\to R_{a+b-1}$,
		$(A,B)\mapsto AF+BG$, is injective. 
	But the $\kk$-dimensions of the domain and the codomain of the map are the same, so the map is surjective as well.
\end{proof}

\begin{rem}\label{rem: bounding box}
	Lemma~\ref{lem: ideal contains a power of maximal ideal} shows, in particular, that $(\deg UV-1)\canvec_i \in \Upper$, for each~$i$, 
		where $\canvec_1,\ldots,\canvec_n$ denote the canonical basis vectors.
	This and Corollary~\ref{cor: properties of the regions}(3)
		show that the lower region $\Lower$ and the boundary $\Boundary$ are contained in the intersection of 
		the cube $[0,\deg UV-1]^n$ and the half-space $\{\vv{t}\in \RR^n:\norm{\vv{t}}\le \deg UV\}$.
\end{rem}

\begin{prop}\label{prop: characterization of CP}
	Let $\vv{c}=\vv{a}/q\in (\QQnn)^n_q$.
	The following are equivalent\textup:
	\begin{enumerate}[(1)]
		\item\label{item: c is a CP}
			$\vv{c}$ is a critical point\textup;
		\item\label{item: first characterization of CP continuous}
			$\vv{c}\in \Upper$ and $\vv{c}-t\,\canvec_i\in \Lower$, for each $i$ such that $c_i>0$
			and $0<t\le c_i$\textup;
		\item\label{item: first characterization of CP discrete} 
			$\vv{c}\in \Upper$ and $\vv{c}-\canvec_i/q\in \Lower$, for each $i$ such that $c_i>0$\textup;
		\item\label{item: second characterization of CP} 
			$\ell^{\vv{a}}\in \idealb^{[q]}$ and  $\ell^{\vv{a}-\canvec_i}\not\in \idealb^{[q]}$, for each $i$ such that $a_i>0$.
	\end{enumerate}
\end{prop}

\begin{proof}
	Proposition~\ref{prop: basic properties of CPs}\iref{item: upper points in critical region} and 
		Corollary~\ref{cor: properties of the regions}\iref{item: pts below a point in lower region} show that 
		\iref{item: c is a CP}~$\Rightarrow$~\iref{item: first characterization of CP continuous};
		\iref{item: first characterization of CP continuous}~$\Rightarrow$~\iref{item: first characterization of CP discrete} is immediate;  
		Remark~\ref{rem: regions for pts with coords in ZZ[1/p]} shows that 
		\iref{item: first characterization of CP discrete}~$\Leftrightarrow$~\iref{item: second characterization of CP}.
	We now show that \iref{item: first characterization of CP discrete}~$\Rightarrow$~\iref{item: c is a CP}. 
	Rewrite equation~\eqref{eq: relation between Delta and Phi} as follows:
		\begin{equation}\label{eq: relation between Delta and Phi rearranged}
			\underbracket[.5pt][5pt]{{(\deg UV-\norm{\vv{t}}-\Dlb(\vv{t}))}}_A\underbracket[.5pt][5pt]{(\deg UV-\norm{\vv{t}}+\Dlb(\vv{t}))}_B=
			\underbracket[.5pt][5pt]{4(\deg\idealb-\Plb(\vv{t}))}_C.
		\end{equation}
	Substituting $\vv{t}=\vv{c}$, the first clause of~\iref{item: first characterization of CP discrete} shows that $C=0$. 
	Moreover, as $\ideal{x,y}^{q\deg UV-1}\subseteq \idealb^{[q]}$, by Lemma~\ref{lem: ideal contains a power of maximal ideal},
		the second clause of~\iref{item: first characterization of CP discrete} implies that $\norm{\vv{c}}<\deg UV$,
		so $B>0$.
	It follows that $A=0$, so $\Dlb(\vv{c})=\deg UV-\norm{\vv{c}}>0$.
	It remains to prove that $\Dlb$ attains a local maximum at~$\vv{c}$. 	
	Note that SG~\ref{sg: inserting a linear form} tells us that each $\Dlb(\vv{c}\pm \canvec_i/q)$ is either $\Dlb(\vv{c})-1/q$ or 
		$\Dlb(\vv{c})+1/q$. 
	We claim that the latter cannot happen; the result will then follow from SG~\ref{sg: local max}.
	
	Substituting $\vv{t}=\vv{c}+\canvec_i/q$ in~\eqref{eq: relation between Delta and Phi rearranged} and supposing 
		that $\Dlb(\vv{c}+\canvec_i/q)=\Dlb(\vv{c})+1/q$ we reach an impossibility: $A=-2/q$ and $B=2\Dlb(\vv{c})>0$, while $C=0$ 
		(since $\vv{c}+\canvec_i/q\in \Upper$).	
	If $c_i>0$, then substituting $\vv{t}=\vv{c}-\canvec_i/q$ in~\eqref{eq: relation between Delta and Phi rearranged} and supposing 
		that $\Dlb(\vv{c}-\canvec_i/q)=\Dlb(\vv{c})+1/q$ we again reach an impossibility: $A=0$, while  $C>0$, by
		the second clause of~\iref{item: first characterization of CP discrete}.
	This establishes the claim.
\end{proof}

The above proposition shows that a critical point is a minimal point of $\Upper$ with respect to the componentwise order $\le$
	introduced in Definition~\ref{defn: order}.
Although the converse does not hold---there are minimal points of $\Upper$ that are not critical points---it is true that a minimal point of 
	$\Upper\cap \QQ_{q}^n$ is a critical point.
This gives us a method, however impractical, to locate critical points: start with a point $\vv{u}\in\Upper\cap \QQ_{q}^n$, and
	successively subtract $1/q$ from the coordinates of $\vv{u}$, until a minimal point of $\Upper\cap \QQ_{q}^n$ is found.
For a more sophisticated method to locate critical points, see Appendix~\ref{appendix: algorithm}.

A consequence of the minimality discussed above is that if $\vv{b}$ and $\vv{c}$ are critical points with $\vv{b}\le\vv{c}$, then $\vv{b}=\vv{c}$. 
We shall use this fact often, without further comment.

\begin{cor}\label{cor: u in U_q iff u >= CP a/q}
	Let $\vv{u}\in (\QQnn)^n_q$.
	Then $\vv{u}\in \Upper$ if and only if there exists a critical point $\vv{c}\in (\QQnn)_q^n$ such that $\vv{c}\le \vv{u}$. 
\end{cor}

\begin{proof}
	If $\vv{u}\in \Upper$, choose $\vv{c}\in\Upper\cap[\vv{0},\vv{u}]\cap \QQ^n_q$ of minimum norm;
		then $\vv{c}$ satisfies condition~\iref{item: first characterization of CP discrete} of Proposition~\ref{prop: characterization of CP},
		and is therefore the desired critical point.
	The converse follows directly from Proposition~\ref{prop: basic properties of CPs}\iref{item: upper points}.
\end{proof}

\begin{rem}\label{rem: uniqueness of CP under point in boundary of trivial region}
	While a point in $\Upper$ may be strictly greater than several critical points, that is never the case for points in 
		$\boundary\Trivial = \bigl\{ \vv{t} \in \mathbb{R}^n_{\geq 0} : \norm{ \vv{t} } = \deg UV \bigr\}$.
	Indeed, suppose $\vv{b}\ne\vv{c}$ are critical points with $\vv{b},\vv{c}\le \vv{u}\in \boundary\Trivial$. 
	Consider the sum $d(\vv{u},\vv{b})+d(\vv{u},\vv{c})$.
	Since $\vv{u}\in\boundary\Trivial$,  this sum equals $\Dlb(\vv{b})+\Dlb(\vv{c})$, which is less than or equal to $d(\vv{b},\vv{c})$, 
		by SG~\ref{sg: local max}.
	The triangle inequality then implies that  $d(\vv{u},\vv{b})+d(\vv{u},\vv{c})=d(\vv{b},\vv{c})$.
	This fact, the assumption that $\vv{b},\vv{c}\le \vv{u}$, and the identity $\abs{b-c}+b+c=2\max\{b,c\}$ show that 
		$\sum_{i=1}^n u_i=\sum_{i=1}^n \max\{b_i,c_i\}$. 
	But $u_i\ge \max\{b_i,c_i\}$, for each $i$, and thus $u_i=\max\{b_i,c_i\}$, for each $i$.
	Since $\vv{u}\ne \vv{b},\vv{c}$, 
		this shows, in particular, that neither $\vv{b}$ nor $\vv{c}$ is strictly smaller than~$\vv{u}$.
\end{rem}

The facts proven above lead up to the following ``structure theorem'' for the $F$-threshold function.

\begin{thm}\label{thm: CPs and the FT function}
	Let $\ell=(\ell_1,\ldots,\ell_n)$, where the $\ell_i$ are pairwise prime linear forms in $\kk[x,y]$, and $\idealb=\ideal{U,V}$, 
		where $U$ and $V$ are relatively prime non-constant forms in $\kk[x,y]$.
	Fix $\vv{t}\in \RRpos^n$ and set $\lambda =\deg UV/\norm{\vv{t}}$. 
	If, for some $e$, the truncation $\tr{\lambda\vv{t}}{e}$ lies in $\Upper$, 
		then there exists a unique critical point $\vv{c}\le \tr{\lambda\vv{t}}{e}$ 
		associated with~$\ell$ and~$\idealb$ with coordinates in $\QQ_{p^e}$, and $\ftlb{\vv{t}}$ is determined by $\vv{c}$\textup:
		\[\ftlb{\vv{t}}= \max\left\{\frac{c_1}{t_1},\ldots,\frac{c_n}{t_n}\right\}<\lambda.\]
	Otherwise, $\ftlb{\vv{t}}$ is determined by the trivial region attached to $\idealb$\textup:
		\[\ftlb{\vv{t}}=\lambda=\frac{\deg UV}{\norm{\vv{t}}}.\]
\end{thm}

\begin{proof}  
	If  $\tr{\lambda\vv{t}}{e}\in\Lower$, for each $e$, then  $\lambda\vv{t}\in \closure{\Lower}$.
	As $\lambda\vv{t}\in \boundary\Trivial\subseteq\Upper$, $\lambda\vv{t}\in \Boundary$, whence $\ftlb{\vv{t}}=\lambda$.
	If  $\tr{\lambda\vv{t}}{e}\in\Upper$, for some $e$, then there exists a critical point 
		$\vv{c}=\vv{a}/p^e\le \tr{\lambda\vv{t}}{e}<\lambda \vv{t}$, by Corollary~\ref{cor: u in U_q iff u >= CP a/q}, 
		and Proposition~\ref{prop: CPs vs normalized points} shows that $\ftlb{\vv{t}}= \max \{c_i/t_i\}$.
\end{proof}

The way the $F$-threshold function is determined by critical points (and the trivial region) can be compactly stated as follows.

\begin{cor}
	For each $\vv{t}\in \RRpos^n$, 
		\[
			\pushQED{\qed} 
			\ftlb{\vv{t}}=\min \biggl( \biggl\{ \max \biggl\{\frac{c_1}{t_1},\ldots,\frac{c_n}{t_n}\biggr\} : 
				\vv{c}\text{\emph{\ is a critical point}}\biggr\}\cup\biggl\{ \frac{ \deg UV}{\norm{\vv{t}}} \biggr\} \biggl).\qedhere
			\popQED
		\]
\end{cor}

\begin{rem}\label{rem: bound on distance between FT and "LCT"} 
	Suppose that $\vv{c}$ is a critical point, so that, in particular, $\Dlb(\vv{c}) = \deg UV - \norm{\vv{c}}$.  
	It follows from the inequality $\norm{\vv{c}}/\norm{\vv{t}} \leq \max\{c_i/t_i\}$ that 
	\begin{equation}\label{eq: bound on distance between FT and "LCT"}
		\frac{\Dlb(\vv{c})}{\norm{\vv{t}}} = \frac{\deg UV}{\norm{\vv{t}}} - \frac{ \norm{\vv{c}}}{\norm{\vv{t}}} \geq 
			\frac{\deg UV}{\norm{\vv{t}}} - \max \left \{ \frac{c_1}{t_1}, \ldots, \frac{c_n}{t_n} \right \}.
	\end{equation}
	If, in addition, we suppose that $\max\{c_i/t_i\} \leq \deg UV/\norm{\vv{t}}$, 
		then \eqref{eq: CPs vs normalized points} and~\eqref{eq: bound on distance between FT and "LCT"} give us the following:
		\[\frac{\deg UV}{\norm{\vv{t}}}-\ftlb{\vv{t}}\le \frac{\Dlb(\vv{c})}{\norm{\vv{t}}}.\]
	This will be particularly relevant to us in the case where $\vv{c}$ is not in $\NN^n$; then we can write $\vv{c}=\vv{a}/q$, where $q>1$
		and some $a_i$ is prime to $p$, and SG~\ref{sg: upper bound} gives	
		\begin{equation*}
			\frac{\deg UV}{\norm{\vv{t}}}-\ftlb{\vv{t}}\le \frac{n-2}{q\norm{\vv{t}}}.
		\end{equation*}
\end{rem}

\section{\texorpdfstring{The case $\bm{n=3}$: the ``Sierpi\'nski staircases''}{The case n=3: the "Sierpi\'nski staircases"}}
\label{s:  Sierpinksi}

In this section we examine the case $n=3$ in detail. 
Let $\ell_1,\ell_2,\ell_3\in \kk[x,y]$ be pairwise prime linear forms, $\ell=(\ell_1,\ell_2,\ell_3)$, $\idealm=\ideal{x,y}$, and 
	$\idealb=\ideal{U,V}$, where $U, V\in \kk[x,y]$ are relatively prime non-constant forms.  

\begin{exmp}\label{exmp: some sierpinski staircases}
	Figure~\ref{fig: sierpinski staircases} shows the lower regions attached to $\ell=(\ell_1,\ell_2,\ell_3)$ and $\idealm$ in 
		various characteristics.
	Since $\idealm$ is invariant under linear changes of variables, those regions do not depend on the choice of $\ell$; in fact, 
		through a change of variables we may assume that $\ell=(x,y,x+y)$.
	\begin{figure}[!t]
		\centering
		\includegraphics[width=0.327\textwidth]{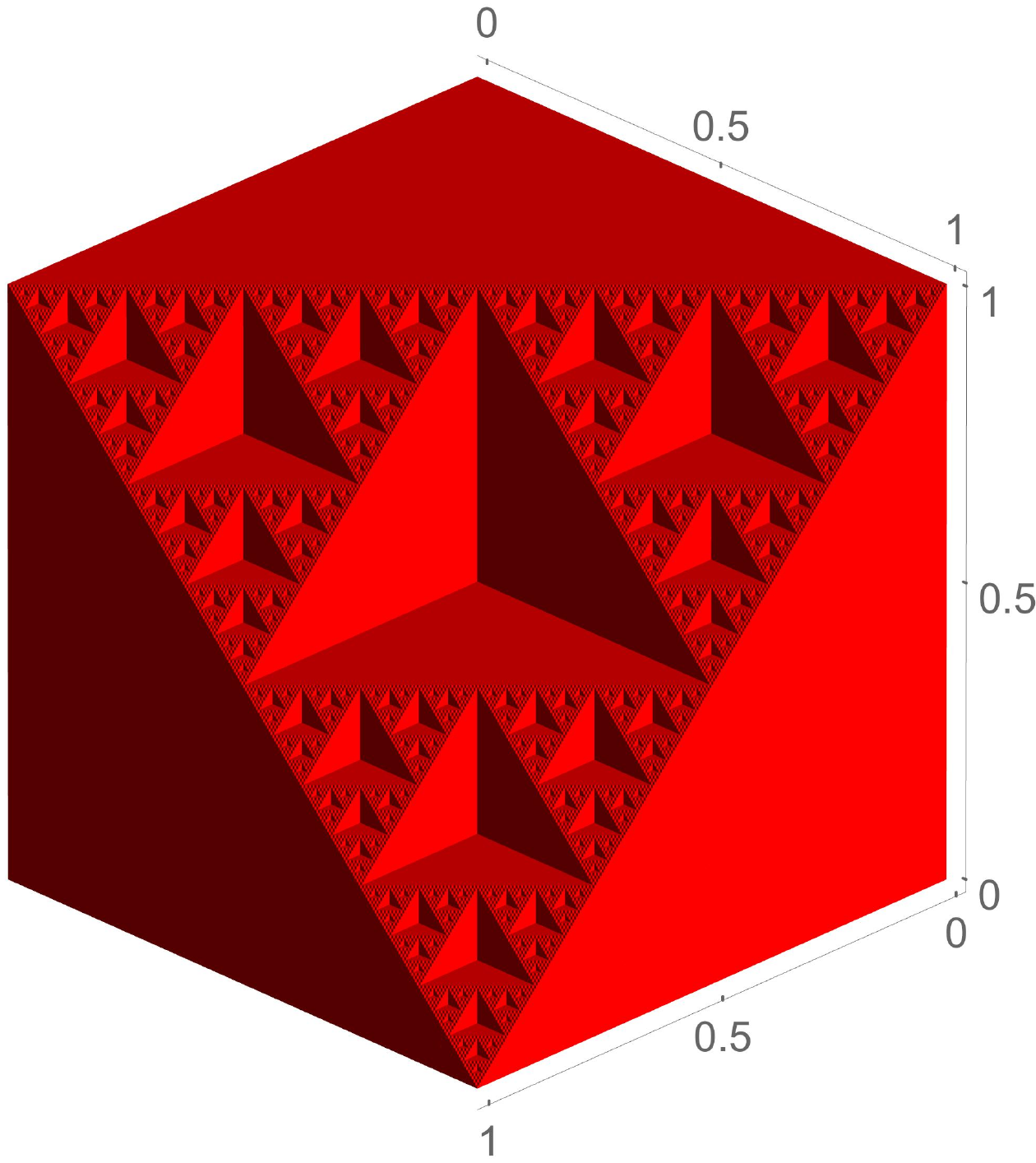}
		\includegraphics[width=0.327\textwidth]{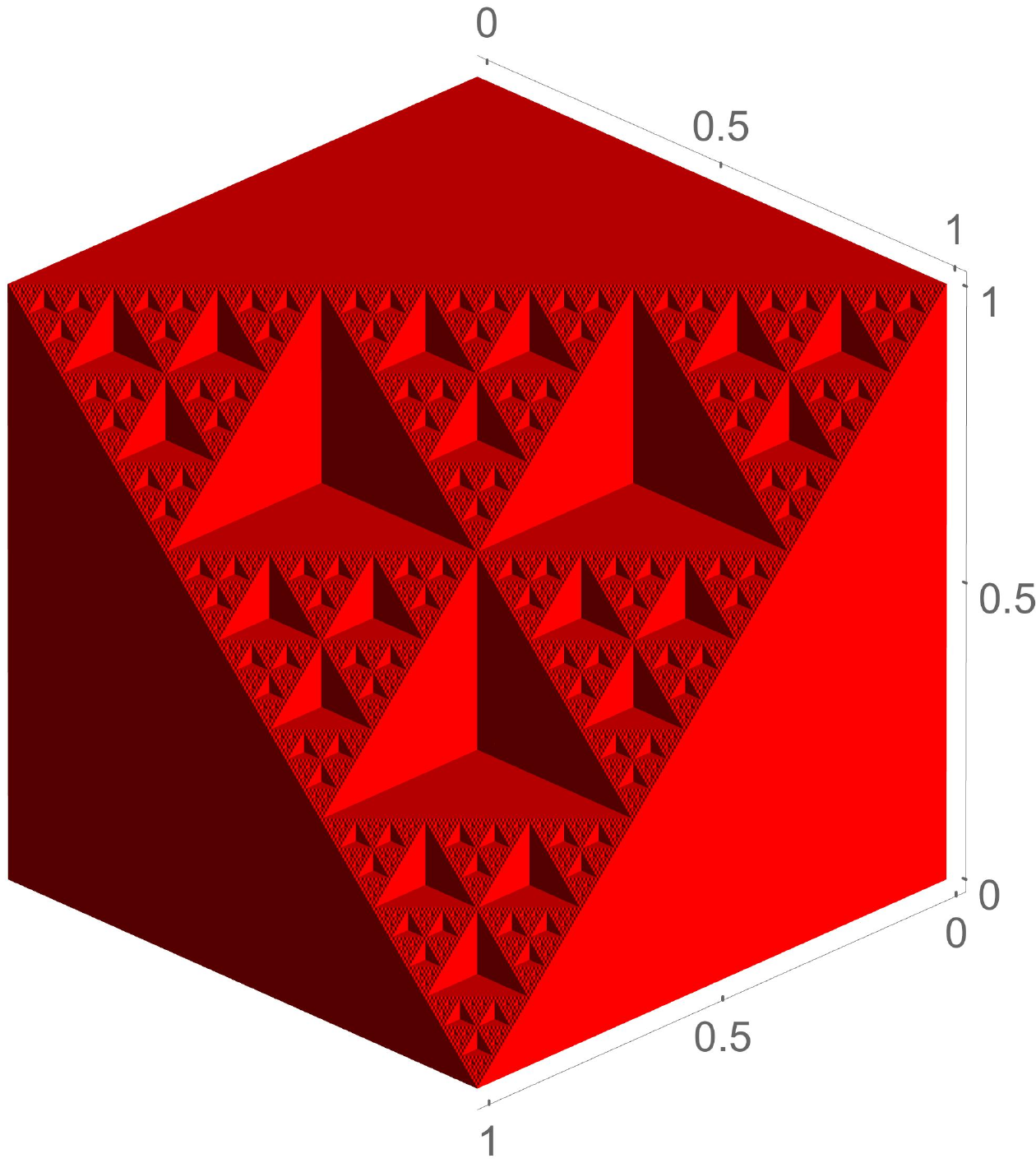}
		\includegraphics[width=0.327\textwidth]{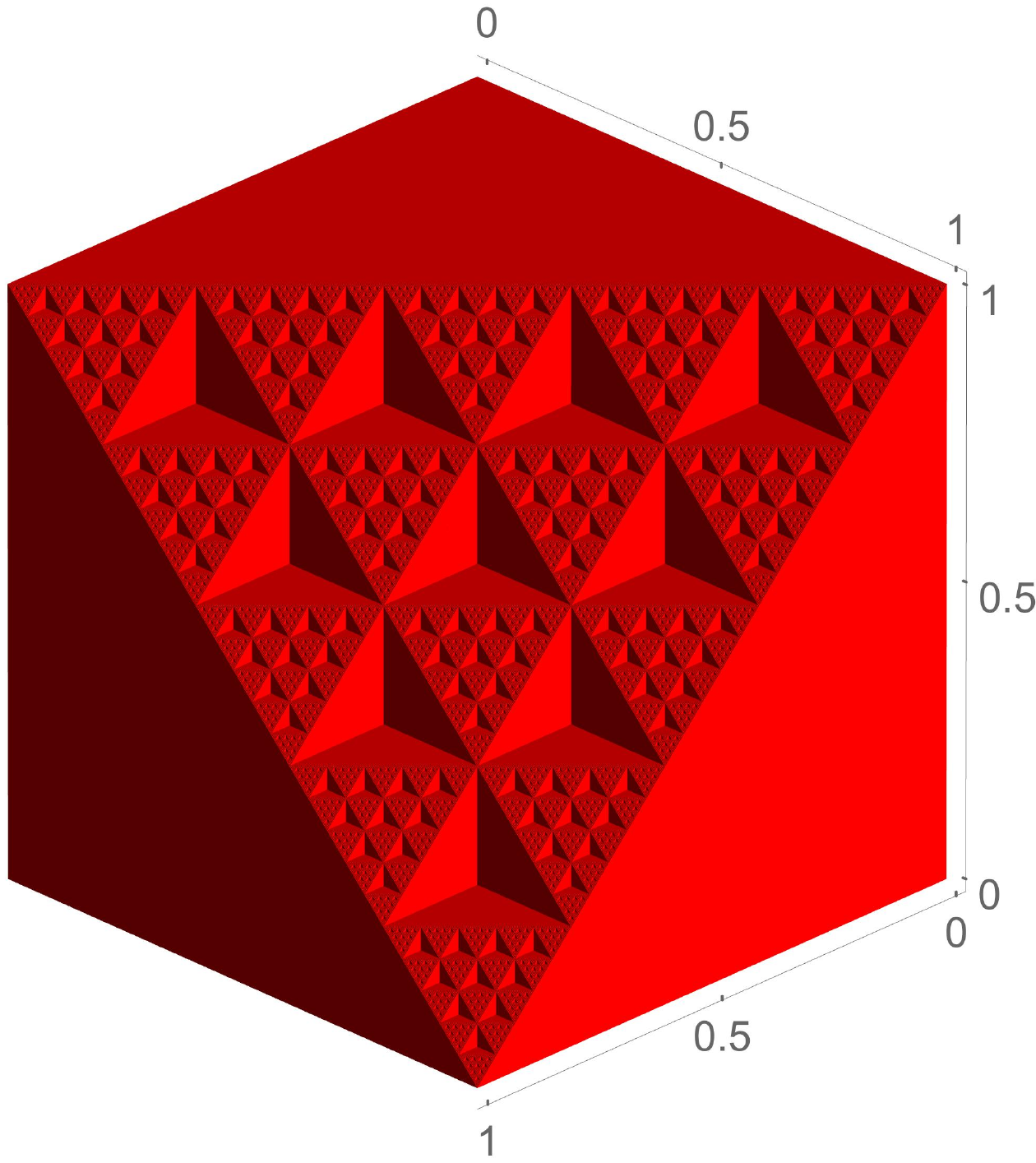}
		\caption{Sierpi\'nski staircases in characteristics 2, 3, and 5}\label{fig: sierpinski staircases}
	\end{figure}
	Figure~\ref{fig: sierpinski staircases wrt other ideals} shows the lower regions attached to $\ell=(x,y,x+y)$ and various ideals $\idealb$ in 
		characteristic~5.
	\begin{figure}
		\centering
		\includegraphics[width=1.4in]{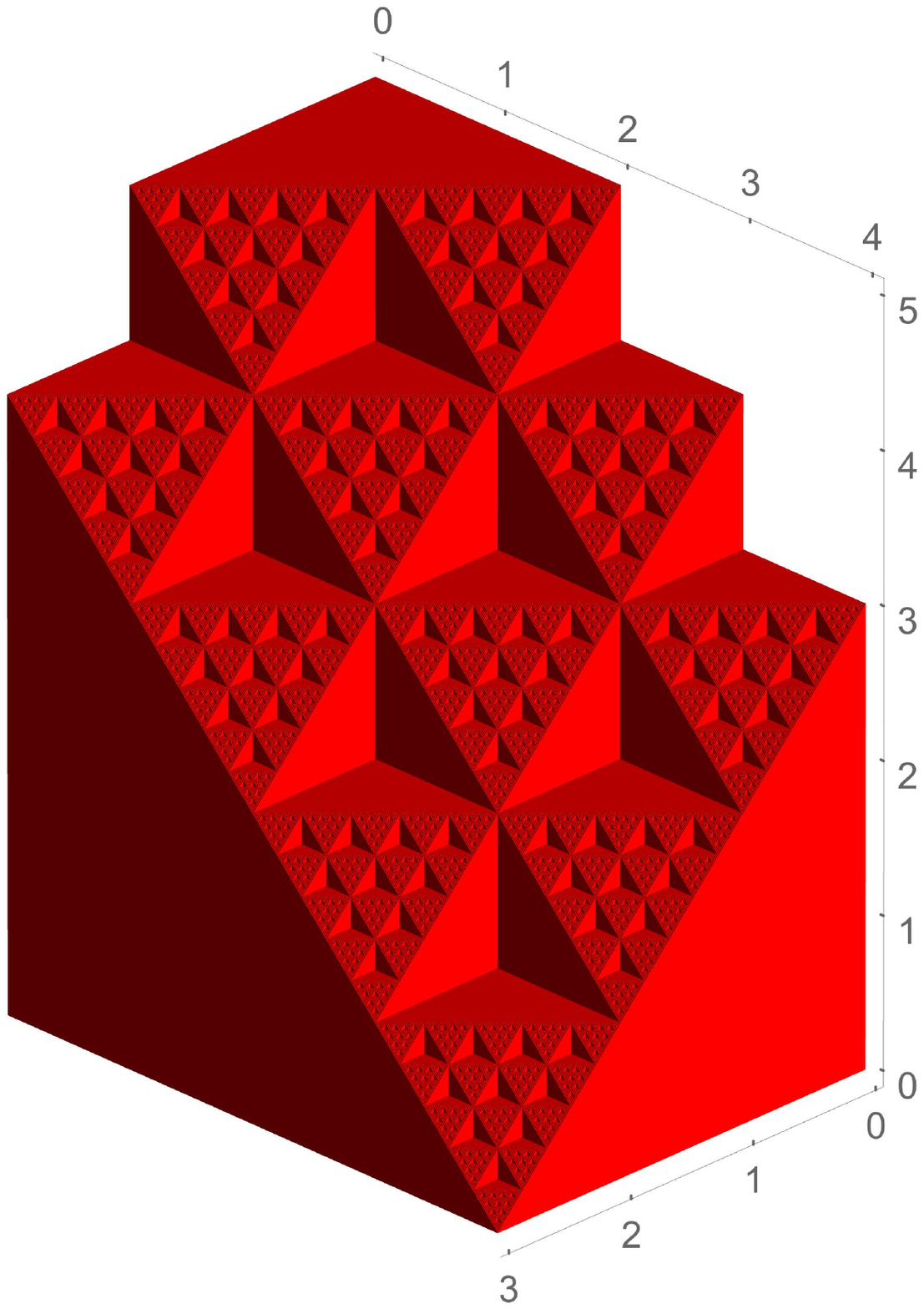}
		\includegraphics[width=1.564in]{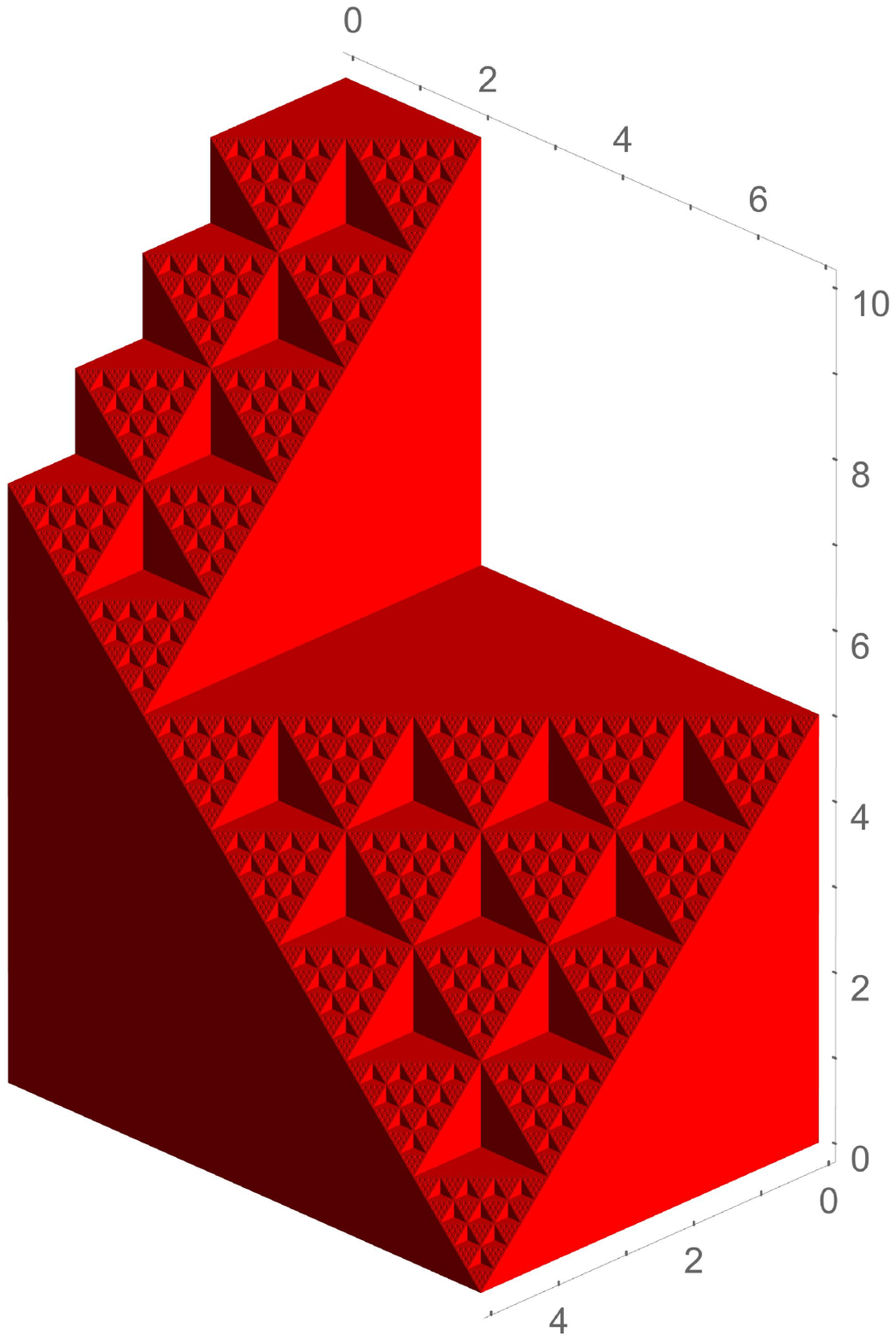}
		\includegraphics[width=1.7in]{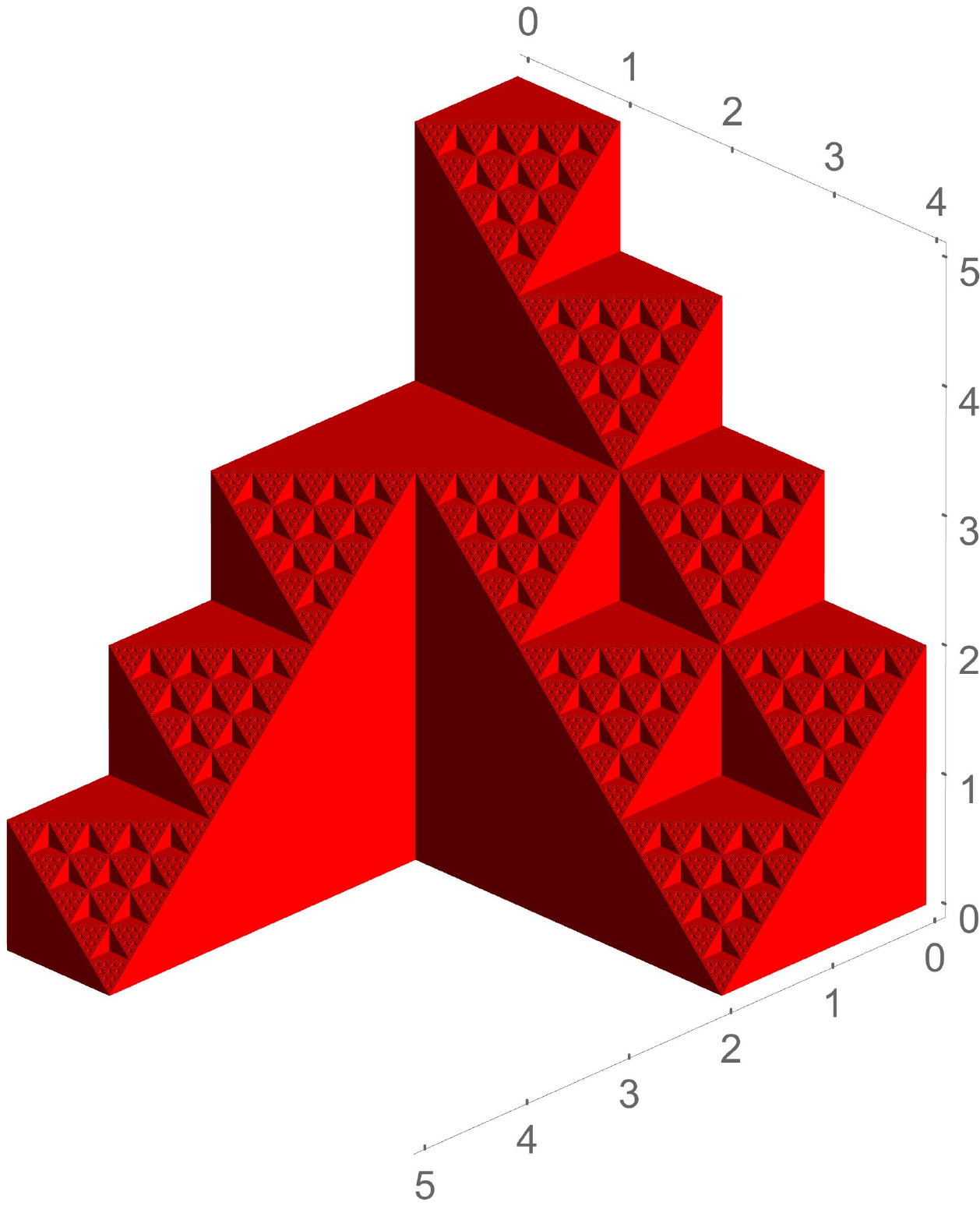}
		\caption{Sierpi\'nski staircases in characteristic 5 attached to $\ell=(x,y,x+y)$ and the ideals $\ideal{x^3,y^4}$, 
			$\ideal{x^5,y^7}$, and $\ideal{x^3+y^3,x^2 y}$}\label{fig: sierpinski staircases wrt other ideals}
	\end{figure}
	Because of the conspicuous presence of Sierpi\'nski gaskets and analogues, we call those regions ``Sierpi\'nski staircases''.
	These pictures were obtained by choosing a reasonably large $q$ and finding all maximal points $\vv{v}$ of 
		$\Lower\cap\QQ^3_{q}$, and then plotting all ``boxes'' $[\vv{0},\vv{\vv{v}}]$ together, where $\vv{v}$ ranges over all those 
		maximal points.  For more details (e.g., regarding the presence of the Sierpi\'nski gaskets and analogues), 
		see Theorem~\ref {thm: sierpinski} and Corollary~\ref{cor: sierpinski}.
		\end{exmp}
		
\begin{rem}\label{rem: remarks on the case n=3}
	SG~\ref{sg: upper bound}
		tells us that if the function $\Dlb$ attached to $\ell$ and $\idealb$ attains a local maximum at $\vv{a}/q$, 
		where $q>1$ and some $a_i$ is prime to $p$, then $\Dlb(\vv{a}/q)=1/q$.
	This is reflected in a unique feature of critical points in  the case $n=3$.
	Suppose $\vv{c}=\vv{a}/q$, with $\vv{a}$ and $q$ as above, is a critical point.  
	As $\Dlb(\vv{c})=1/q$,  
		$\norm{\vv{c}}+1/q=\deg UV$, by definition of critical point.
	That is, $\vv{c}$ is at distance $1/q$ from the trivial region.
\end{rem}

\begin{rem}[Coordinate sums of truncations]
\label{norms of truncations: R}
	Consider a point $\vv{u} \in \RR^3$ with positive coordinates.   
	As in Remark~\ref{Characterizations of truncations: R},  we have that 
		\[ \tr{\vv{u}}{s} < \vv{u} \leq \tr{\vv{u}}{s} + \frac{\vv{1}}{p^s}\] 
		for every $s \geq 0$, where $\vv{1}$ denotes the vector $(1,1,1)$.
	The second inequality is an equality in 
		some component if and only if that component lies in $\QQ_{p^s}$. 
	From this, it follows that when $\norm{\vv{u}}\in \NN$,
		$\norm{\tr{\vv{u}}{s}} = \norm{\vv{u}} -a/p^s$ 
		for some integer $1 \leq a \leq 3$, and that $a=3$ if and only if $\vv{u} \in \QQ^3_{p^s}$.
\end{rem}

With these observations, we are ready to begin our examination of the intersection of the boundary $\Boundary$ and the trivial region.

\begin{lem}  
\label{staircase: L}  Fix $\vv{u} \in \RR^3_{>0}$ with $\norm{\vv{u}} = \deg UV$ and $\tr{\vv{u}}{0} \in \Lower$.  If $\tr{\vv{u}}{e} \in \Upper$ for some $e \geq 1$, then there exists $1 \leq s \leq e$ such that $\tr{\vv{u}}{s}$ is a critical point with $\| \tr{\vv{u}}{s} \| = \deg UV - 1/p^s$.
\end{lem}

\begin{proof}  
	If $\tr{\vv{u}}{e} \in \Upper$, then Corollary~\ref{cor: u in U_q iff u >= CP a/q} shows that there exists a critical point 
		$\vv{c} \in \QQ^3_{p^e}$ with $\vv{c} \leq \tr{\vv{u}}{e} < \vv{u}$.  
	If $\vv{c}$ were integral, then we would have that $\vv{c} \leq \tr{\vv{u}}{0}$, which contradicts the assumption 
		that $\tr{\vv{u}}{0} \in \Lower$.  
	Consequently, the minimal $s$ such that $\vv{c} \in \QQ^3_{p^s}$ must satisfy $1 \leq s \leq e$, and so, 
		by Remark~\ref{rem: remarks on the case n=3}, 
		\[ \norm{\vv{c}} + \frac{1}{p^s}  = \deg UV = \norm{\vv{u}}.\]
	Because $\vv{u}-\vv{c}>\vv{0}$ and $\norm{\vv{u}-\vv{c}}=1/p^s$, we conclude that $\vv{u}-\vv{c}< \vv{1}/p^s$,	
		so $\vv{c}<\vv{u}<\vv{c}+\vv{1}/p^s$, and $\vv{c}=\tr{\vv{u}}{s}$, by Remark~\ref{Characterizations of truncations: R}.
\end{proof}

\begin{thm}
\label{n=3 main theorem: T}
Fix $\vv{u} \in \RR^3_{>0}$ with $\norm{\vv{u}} = \deg UV$ and $\tr{\vv{u}}{0} \in \Lower$, and set 
\[ \mu = \inf \left \{ s \geq 1 : \| \tr{\vv{u}}{s} \| = \deg UV - \frac{1}{p^s} \right \}.\]
If $\mu  = \infty$, then $\vv{u} \in \Boundary$. Otherwise, $\tr{\vv{u}}{{\mu}}$ is a critical point, and $\vv{u} \notin \Boundary$.
\end{thm}

\begin{proof}  
	If $\mu =\infty$, then $\| \tr{\vv{u}}{s} \| \neq \deg UV - 1/p^s$ for every $s \geq 1$; 
		Lemma~\ref{staircase: L} then implies that $\tr{\vv{u}}{e} \in \Lower$ for all $e \geq 1$, and so $\vv{u} \in \closure{\Lower}$.  
	Since $\vv{u}$ lies also in the trivial region, we conclude that $\vv{u}\in \Boundary$.
	
	If $\mu < \infty$, then $\ell^{ \, p^{\mu} \tr{\vv{u}}{\mu}}$ lies in $\idealm^{\deg UV \cdot p^{\mu} - 1}$, which by 
		Lemma~\ref{lem: ideal contains a power of maximal ideal} is contained in 
		$\langle  U^{p^{\mu}}, V^{p^{\mu}} \rangle = \idealb^{[p^{\mu}]}$. 
	In other words, $\tr{\vv{u}}{{\mu}} \in \Upper$, and the minimality of $\mu$ and 
		Lemma~\ref{staircase: L}  further show that $\tr{\vv{u}}{{\mu}}$ must be a critical point.
\end{proof}

\begin{defn}
The Sierpi\'nski $p$-gasket $\mathscr{S}_p$ consists of all points $\vv{v} \in [0,1]^3$ with $\norm{\vv{v}}=2$ for which there exists an expansion\footnote{As observed in Remark \ref{Comparing expansions: R}, there are at most $8$ possible expansions for the point $\vv{v}$.} 
	$\vv{v} = \sum_{e=1}^{\infty} \frac{\vv{w}_e}{p^e}$ 
	with $\norm{\vv{w}_e} = 2p-2$ for all $e \geq 1$.
\end{defn}

Alternatively, let  $\vv{L}=(0,1,1)$, $\vv{M}=(1,0,1)$, and $\vv{N}=(1,1,0)$.  
Any point $\vv{v}\in [0,1]^3$ with $\norm{\vv{v}}=2$ can be written uniquely as 	$\vv{v}=\lambda\vv{L}+\mu\vv{M}+\nu\vv{N}$, 
	where $\lambda, \mu, \nu\in [0,1]$ and  $\lambda+\mu+\nu=1$, and it is straightforward to verify that 
	$\vv{v} \in \mathscr{S}_p$ if and only if $\lambda, \mu$ and $\nu$ have expansions whose digits add to $p-1$ in every spot.  
Note that when $p=2$,  this characterization of $\mathscr{S}_p$ agrees with a well-known description of the Sierpi\'nski gasket 
	fractal with vertices $\vv{L}$, $\vv{M}$, and $\vv{N}$.

\begin{thm}
\label{thm: sierpinski}
	If $\idealb=\idealm$, then the intersection of the boundary $\Boundary$ and the hyperplane $X=\{\vv{t}\in \RR^3:\norm{\vv{t}}=2\}$
		 is the Sierpi\'nski $p$-gasket $\mathscr{S}_p$.
\end{thm}

\begin{proof}
	It is easy to see that the vertices $\vv{L}$, $\vv{M}$, and $\vv{N}$ of  $\mathscr{S}_p$ lie in $\Boundary$:  
		for instance, $\vv{L}\in \Upper$ and $\{(1-p^{-s})\vv{L}\}_{s\ge 1}$ is a sequence of points in $\Lower$ converging to $\vv{L}$.
	This and Remark~\ref{rem: bounding box} allow us to restrict our attention to points $\vv{u}\in X\cap  (0,1]^3$.
	For such $\vv{u}$, we show that either both $\vv{u} \in \Boundary$ and $\vv{u} \in \mathscr{S}_p$, 
		or neither of these hold.
	
	Let $\vv{u} = \sum_{e=1}^{\infty}\vv{u}_e/p^e$ be the unique expansion of $\vv{u}$ that is non-terminating in each coordinate.  Among all possible expansions of $\vv{u}$, this one is of particular interest, for by Remark \ref{truncations and expansions: R},
	\begin{equation} \label{truncation as a sum: e} \tr{\vv{u}}{e} = \frac{ \vv{u}_1}{p} + \cdots + \frac{ \vv{u}_e}{p^e} \end{equation} for every $e \geq 1$.

	Set $r= \inf \{ e \geq 1 : \norm{{\vv{u}}_e} \neq 2p-2 \}$.
	If $r = \infty$, then the componentwise non-terminating expansion shows that $\vv{u} \in \mathscr{S}_p$, while \eqref{truncation as a sum: e} also shows that  $\norm{\tr{\vv{u}}{e}} = 2-2/p^e$ for every $e \geq 1$, and so
		$\vv{u} \in \Boundary$ by Theorem~\ref{n=3 main theorem: T}.
	Suppose $r<\infty$.
	As $\norm{\vv{u}}=2$ and $\norm{{\vv{u}}_e} = 2p-2$, for each $1\le e<r$, 
		$\norm{\vv{u}_r}$ must be either $2p-1$ or $2p-3$.
		
	If $\norm{\vv{u}_r}=2p-3$, then \eqref{truncation as a sum: e} shows both that 
		 $\norm{\tr{\vv{u}}{e}}=2-2/p^e$ for each $e<r$, and that
		$\norm{\tr{\vv{u}}{r}}=2-3/p^r$.
	By Remark~\ref{norms of truncations: R}, $\vv{u}\in \QQ_{p^r}^3$, and so in particular, $\vv{u}_e=(p-1)\vv{1}$ for each $e>r$.  This shows that $\norm{\tr{\vv{u}}{e}}=2-3/p^e$, for each $e> r$ and Theorem~\ref{n=3 main theorem: T} then implies that $\vv{u}\in \Boundary$.
	On the other hand, taking one of the components of $\vv{u}$ whose $r$th digit is not $p-1$ and replacing its non-terminating expansion
		with its terminating expansion, the norm of the $r$th digit increases by 1,   the norm of all subsequent digits decrease by $p-1$,
		and the prior digits remain unchanged (see Remark~\ref{Comparing expansions: R})---so 
		 we obtain an expansion of $\vv{u}$ whose digits add to $2p-2$ in every spot, whence $\vv{u}\in \mathscr{S}_p$.
	
	If  $\norm{\vv{u}_r}=2p-1$, then \eqref{truncation as a sum: e} implies that $\norm{\tr{\vv{u}}{r}}=2-1/p^r$, so $\vv{u}\notin \Boundary$, by Theorem~\ref{n=3 main theorem: T}.
	We must show that no expansion of $\vv{u}$ satisfies the definition of the Sierpi\'nski $p$-gasket.
	This is the case for the non-terminating expansion $\sum_{e=1}^{\infty}\vv{u}_e/p^e$.
	Any other expansion is obtained from this one by replacing the non-terminating expansion of one or more components 
		with the terminating expansion (if at all possible).
	But  Remark~\ref{Comparing expansions: R} shows that any such replacement, 
		when performed to an expansion where the first digit with norm $\ne 2p-2$ has a \emph{larger} norm,
		will result in an expansion with the same property.	
	Consequently, performing any such changes to $\sum_{e=1}^{\infty}\vv{u}_e/p^e$ will not lead to an expansion satisfying 
		the definition of the Sierpi\'nski $p$-gasket.
\end{proof}

The next  result, which explains the presence of Sierpi\'nski $p$-gaskets  in  Figure~\ref{fig: sierpinski staircases wrt other ideals},
	 relates regions attached to $\idealb$ to regions attached to $\idealm$.
We temporarily  decorate our notations for the regions attached to $\idealm$ with the subscript $\idealm$. 
	
\begin{prop}\label{prop: symmetry}
	If $\vv{v}\in \Lower\cap \NN^3$ and $\norm{\vv{v}}=\deg UV-2$, 
		then $\Boundary \cap [\vv{v},\vv{v}+\vv{1}]=\vv{v}+\Boundary_\idealm$, the translation  of $\Boundary_\idealm$ by $\vv{v}$.
\end{prop}

\begin{proof}
	By Lemma~\ref{lem: ideal contains a power of maximal ideal},
		$\ell^{\vv{v}}\idealm \subseteq \idealm^{\deg UV-1}\subseteq \idealb$, and yet $\ell^{\vv{v}}\notin \idealb$, since
		$\vv{v}\in \Lower$.
	This shows that $(\idealb:\ell^{\vv{v}})=\idealm$.
	For each $\vv{a}\in \NN^3$ and each $q$ we have $(\idealb^{[q]}:\ell^{q\vv{v}+\vv{a}})=((\idealb^{[q]}:\ell^{q\vv{v}}):\ell^{\vv{a}})=
		((\idealb:\ell^{\vv{v}})^{[q]}:\ell^{\vv{a}})=(\idealm^{[q]}:\ell^{\vv{a}})$, where the second equality follows from the flatness
		of the Frobenius over $\kk[x,y]$. 
	Thus, $\ell^{q\vv{v}+\vv{a}}\in \idealb^{[q]}$ if and only if $\ell^{\vv{a}}\in \idealm^{[q]}$ or, equivalently, 
		$\vv{v}+\vv{a}/q\in \Upper$ if and only if $\vv{a}/q\in \Upper_\idealm$.
	It follows that $\Upper\cap [\vv{v},\vvv{\infty})\cap \QQ^3_{p^\infty}=(\vv{v}+\Upper_\idealm)\cap\QQ^3_{p^\infty}$.
	Taking closures we see that $\Upper \cap [\vv{v},\vvv{\infty})=\vv{v}+\Upper_\idealm$, and consequently
		$\Lower \cap [\vv{v},\vvv{\infty})=\vv{v}+\Lower_\idealm$.
	From this, it is easy to see that $\vv{v}+\Boundary_\idealm\subseteq \Boundary \cap [\vv{v},\vv{v}+\vv{1}]$ 
		and $\Boundary \cap (\vv{v},\vv{v}+\vv{1}]\subseteq \vv{v}+\Boundary_\idealm$.
	To complete the proof, let $\vv{u}\in \Boundary \cap [\vv{v},\vv{v}+\vv{1}]$ be a point with $u_i=v_i$, for some~$i$.
	As $\vv{u}\in \Upper$, the point $\vv{u}-\vv{v}$ lies in $\Upper_\idealm$.
	Since $\vv{u}-\vv{v}\in [0,1]^3$ and one of its coordinates is~0, it must be the case that some other coordinate equals~1 
		(see Example~\ref{exmp: regions for case n=2}), and $\vv{u}-\vv{v}\in \Boundary_\idealm$. 
\end{proof}

With a little  more work, we can characterize the intersection of the boundary with the trivial region, for arbitrary $\idealb=\ideal{U,V}$.  

\begin{cor} 
\label{cor: sierpinski}
	Let $\mathscr{P}= \Lower\cap \{ \vv{v}  \in \NN^3 : \norm{\vv{v}} = \deg UV - 2 \}$.
	Then the intersection the boundary $\Boundary$ and the hyperplane $X=\{\vv{t}\in \RR^3:\norm{\vv{t}}=\deg UV\}$ 
		 is the Minkowski sum $\mathscr{P}+\mathscr{S}_p=\{\vv{p}+\vv{s}:\vv{p}\in\mathscr{P}\text{ and }\vv{s}\in \mathscr{S}_p \}$.
\end{cor}

\begin{proof}
	Theorem~\ref{thm: sierpinski} and Proposition~\ref{prop: symmetry} show that 
		$\mathscr{P}+\mathscr{S}_p\subseteq \Boundary\cap X$.
	To show the reverse inclusion, 
		we show that  for each $\vv{u}\in \Boundary\cap X$ there exists $\vv{v}\in\mathscr{P}$ such that $\vv{v}\le \vv{u}\le \vv{v}+\vv{1}$.
	Proposition~\ref{prop: symmetry} then shows that $\vv{u}-\vv{v}\in \Boundary_\idealm$, 
		so $\vv{u}-\vv{v}\in \mathscr{S}_p$, by Theorem~\ref{thm: sierpinski}, 
		and $\vv{u}=\vv{v}+(\vv{u}-\vv{v})\in \mathscr{P}+\mathscr{S}_p$.

	Let $\vv{u}\in \Boundary\cap X$.
	Since $\norm{\vv{u}}=\deg UV$ and $\vv{u}\in \Boundary\subseteq [0,\deg UV-1]^3$, 
		$\vv{u}$ can have at most one zero coordinate.
	Suppose $\vv{u}$ has one zero coordinate.
	If one of the two nonzero coordinates of $\vv{u}$  is not an integer, then neither is the other one, as $\norm{\vv{u}}\in \NN$; 
		thus, $\down{\vv{u}}$ has norm $\deg UV-1$, and is therefore 
		in $\Upper$, by Lemma~\ref{lem: ideal contains a power of maximal ideal}.
	This implies that $\vv{u}$ lies in the interior of  $\Upper$, a contradiction.
	So $\vv{u}\in\NN^3$, and subtracting 1 from each of its nonzero coordinates we obtain the desired point in $\mathscr{P}$. 
	So we assume from now  on that $\vv{u}$ has positive coordinates.
	
	Suppose  $\vv{u}\notin \NN^3$.
	As $\tr{\vv{u}}{0}<\vv{u}\in \Boundary$, $\tr{\vv{u}}{0}\in \Lower$.
	Furthermore, Remark \ref{norms of truncations: R} and the assumption $\vv{u} \notin \NN^3$ show that 
		$\norm{\tr{\vv{u}}{0}}$ is either $\deg UV -1$ or $\deg UV -2$.  
	The first possibility, however, is ruled out by Lemma~\ref{lem: ideal contains a power of maximal ideal} and 
		the fact that $\tr{\vv{u}}{0} \in \Lower$.  
	Thus,  $\tr{\vv{u}}{0}\in \mathscr{P}$ is the desired point.

	Suppose now that $\vv{u}\in \NN^3$.	
	We claim that one of the points $\vv{u}-\canvec_1-\canvec_2$ and $\vv{u}-\canvec_1-\canvec_3$ lie in  $\Lower$, 
		and consequently in $\mathscr{P}$.
	To see that, we analyze some values of the function $\Dlb$ attached to $\ell$ and $\idealb$, noting that, 
		by Proposition~\ref{prop: basic properties of Phi}\iref{item: description of Dlb on upper region},
		\[\vv{t}\in \Upper \iff \Dlb(\vv{t})=\abs{\norm{\vv{t}}-\deg UV}.\]
	As $\norm{\vv{u}}=\deg UV$, $\Dlb(\vv{u})=0$.
	By SG~\ref{sg: inserting a linear form}, $\Dlb(\vv{u}-\canvec_1)=1$, the possible values for $\Dlb(\vv{u}-\canvec_1-\canvec_2)$ and 
		$\Dlb(\vv{u}-\canvec_1-\canvec_3)$ are $0$ or $2$, and the possible values for  $\Dlb(\vv{u}-\vv{1})$ are $1$ or $3$.
	But since $\vv{u}\in \Boundary$, $\vv{u}-\vv{1}\in \Lower$, and consequently $\Dlb(\vv{u}-\vv{1})=1$.
	Thus, SG~\ref{sg: 2Dconv} shows that one of $\Dlb(\vv{u}-\canvec_1-\canvec_2)$ and $\Dlb(\vv{u}-\canvec_1-\canvec_3)$ must be $0$, 
		and thus one of $\vv{u}-\canvec_1-\canvec_2$ and $\vv{u}-\canvec_1-\canvec_3$  lie in $\Lower$, 
		establishing our claim. 	
\end{proof}

Next, we turn our attention to the $F$-threshold function.  The following result is an immediate corollary of Theorem \ref{thm: CPs and the FT function} and Theorem \ref{n=3 main theorem: T}.

\begin{cor}  
\label{ft when n=3: C}
Let $\vv{t} \in \RR^3_{>0}$, and set $\vv{u} = (\deg UV/ \norm{\vv{t}}) \, \vv{t}$.  Suppose $\ell^{\up{\vv{u}}-\vv{1}} \notin \idealb$, and set $\mu = \inf \{ s \geq 1 : \| \tr{\vv{u}}{s} \| = \deg UV - 1/p^s \}$.
\begin{enumerate}[(1)]
\item If $\mu = \infty$, then $\ftlb{\vv{t}} = \deg UV / \norm{\vv{t}}$.
\item Otherwise, $\ftlb{\vv{t}} = \max \left \{ \frac{\tr{u_1}{{\mu}}}{t_1}, \frac{\tr{u_2}{{\mu}}}{t_2}, \frac{\tr{u_3}{{\mu}}}{t_3} \right \}$. \qed
\end{enumerate}
\end{cor}

For the remainder of this section, for simplicity we specialize to the case that $\idealb = \idealm$.  
We will be concerned with studying the value of the $F$-threshold function $\ftlm{\vv{t}}$ for some fixed $\vv{t} \in \RR^3_{>0}$ 
	as the characteristic varies, and we will be especially concerned with understanding when $\ftlm{\vv{t}}$ is determined 
	by the trivial region, that is, when $\ftlm{\vv{t}} = 2 / \norm{\vv{t}}$.  
Let $\vv{u}$ be the  ``normalized'' point $\vv{u}=2\vv{t}/\norm{\vv{t}}$.  
In a fixed characteristic $p$, Theorem \ref{thm: CPs and the FT function} and Theorem~\ref{thm: sierpinski} show that  $\ftlm{\vv{t}}$ is determined by the trivial region if and only if $\vv{u} \in \mathscr{S}_p$.   

If $\vv{u}\notin[0,1]^3$, then $\vv{u}$ cannot lie in $\mathscr{S}_p$ for any $p$, and so 
	$\ftlm{\vv{t}}$ can never equal $2/\norm{\vv{t}}$.  On the other hand, if $\vv{u}\in [0,1]^3$ and some coordinate of $\vv{u}$ equals zero or one, then $\vv{u}$ must lie on some edge of $\mathscr{S}_p$, and so $\ftlm{\vv{t}}$ is always determined by the trivial region.  Thus, the interesting case to consider is when $\vv{u} \in (0,1)^3$.  	
	However, the fractals $\mathscr{S}_p$ have area 0,  and so such a point $\vv{u}$ rarely lies in $\mathscr{S}_p$.  Thus, in some fixed characteristic, $\ftlm{\vv{t}}$ is ``almost never'' determined by the trivial region. Moreover, because the union of these countably many fractals still has area 0, for ``most'' $\vv{t}\in \RRpos^3$, $\ftlm{\vv{t}}$ is \emph{not} determined by the trivial region in \emph{any} characteristic.  The following result then comes as a surprise:

\begin{cor}
\label{n=3 integral case: C}
Fix a point $\vv{a} = (A,B,C) \in \NN^3$ with positive coordinates, set $D = \norm{\vv{a}}$, and suppose that the normalized point $\vv{u} = 2 \vv{a} / D$ lies in $(0,1)^3$.  Fix a prime $p$  that does not divide $D$, and let $\mathcal{O}$ denote the multiplicative order of $p$ modulo $D$.  If $s \geq 1$, then \[  \vv{z}_s = \left( \frac{\lpr{2A p^s}{D}}{D}, \frac{\lpr{2B p^s}{D}}{D} , \frac{\lpr{2C p^s}{D}}{D}  \right) \] lies in $(0,1)^3$, and has coordinate sum equal to either $1$ or $2$.  
\begin{enumerate}[(1)]
\item  If $\norm{\vv{z}_s} = 2$ for all $1 \leq s \leq \mathcal{O}$, then $\ftlm{\vv{a}} = 2/D$.  
\item  Otherwise, if $1 \leq \mu \leq \mathcal{O}$ is the minimal integer for which $ \norm{\vv{z}_{\mu}} = 1$, then 
\[ \ftlm{\vv{a}} = \frac{2}{D} - \frac{1}{p^{\mu}} \cdot \min \left \{  \frac{\lpr{2Ap^{\mu} }{D}}{AD}, \frac{\lpr{2Bp^{\mu} }{D}}{BD}, \frac{\lpr{2Cp^{\mu} }{D}}{CD} \right \}.  \]
\end{enumerate}
In particular,  $\ftlm{\vv{a}} = 2/ D$ for all $p \equiv 1 \bmod D$.
\end{cor}

\begin{proof}
By Remark \ref{truncations of rationals: rem}, 
\begin{equation}
\label{truncation when rational: e}
\tr{\vv{u}}{s} = \vv{u} - \frac{\vv{z}_s}{p^s},
\end{equation}
and Remark \ref{norms of truncations: R} and the assumption that $p \nmid D$ then imply that $\vv{z}_s \in (0,1)^3$, and that either $\norm{\vv{z}_s} = 1$ or $\norm{\vv{z}_s} = 2$.  If $\norm{\vv{z}_s} =2$ for every $1 \leq s \leq \mathcal{O}$, then the same holds for all $s \geq 1$.  In this case, \eqref{truncation when rational: e} and Corollary \ref{ft when n=3: C} imply that $\ftlm{\vv{t}} = 2/D$.  Otherwise, Corollary \ref{ft when n=3: C} instead implies that 
\[ \ftlm{\vv{a}} = \max \left \{ \frac{\tr{\frac{2A}{D}}{{\mu}}}{A}, \frac{\tr{\frac{2B}{D}}{{\mu}}}{B}, \frac{\tr{\frac{2C}{D}}{{\mu}}}{C} \right \}. \] 
The claimed formula for $\ftlm{\vv{a}}$ follows from applying Remark \ref{truncations of rationals: rem} to simplify this expression, and the last assertion is justified by the observation that if $p \equiv 1 \bmod D$, then $\vv{z}_s = \vv{u}$ for all $s \geq 1$.
\end{proof}

\begin{rem}[$F$-pure thresholds as a function of the class of $p$] 
\label{rem:  variation of fpt with p}
In the statement of Corollary \ref{n=3 integral case: C}, the value of $\mu$ depends only on the class of $p$ modulo $\norm{\vv{a}}$.  Thus, is this context, Corollary \ref{n=3 integral case: C} shows that for every unit $u$ modulo $\norm{\vv{a}}$, there exist a positive integer $\mu(u)$ and a nonnegative integer $\mathcal{E}(u)$ such that 
\[ \ftlm{\vv{a}} = \frac{2}{\norm{\vv{a}}} - \frac{\mathcal{E}(u)}{p^{\mu(u)}} \] whenever $p \equiv u \bmod \norm{\vv{a}}$.

Though we will not elaborate on this here, this description of the $F$-pure threshold of the polynomial $\ell^{\vv{a}}$, when combined with certain well-known facts  (see, \eg \cite[Proposition~1.9]{mustata+takagi+watanabe.F-thresholds} or \cite[Key Lemma~3.1]{hernandez.F-purity_of_hypersurfaces}) allows us to give a positive answer to \cite[Problem~3.10]{mustata+takagi+watanabe.F-thresholds} in this setting.
\end{rem}

\section{\texorpdfstring{$\bm F$-pure thresholds of homogeneous polynomials in two variables}
	{F-pure thresholds of homogeneous polynomials in two variables}}\label{s:  fpt}

Let $G$ be a non-constant form in $\kk[x,y]$ and $\idealb=\ideal{U,V}$, where $U,V\in \kk[x,y]$ are non-constant relatively prime forms. 
Extending $\kk$, if necessary, we  write $ G=\ell^\vv{a}$, where $\ell=(\ell_1,\ldots,\ell_n)$ is a collection of pairwise prime linear forms
	and $\vv{a}\in \NNpos^n$.
Let $\lambda=\deg UV/\deg G$. 
As before, we shall say that ``$\ftb(G)$ is determined by a critical point $\vv{c}$'' if 
	$\ftb(G)=\max\{ c_1/a_1, \ldots, c_n/a_n\} <\lambda$ or, equivalently, 
	$\vv{c}<\lambda \vv{a}$; see Proposition~\ref{prop: CPs vs normalized points} and Theorem~\ref{thm: CPs and the FT function}.
According to Theorem~\ref{thm: CPs and the FT function}, we have three mutually exclusive possibilities  for $\ftb(G)=\ftlb{\vv{a}}$: 

\begin{enumerate}[(A)] 	
	\item	$\ftb(G)$ is determined by a critical point in $\NN^n$. 
		This is the case if and only if $\tr{\lambda\vv{a}}{0}\in \Upper$.
	\item	$\ftb(G)$ is determined by a critical point not in $\NN^n$. 
		This is the case if and only if $\tr{\lambda\vv{a}}{0}\notin \Upper$ but $\tr{\lambda\vv{a}}{e}\in \Upper$ for some $e\ge 1$.
	\item $\ftb(G)$ is determined by the trivial region: $\ftb(G)=\lambda$.
\end{enumerate}

\begin{rem}\label{rem: truncation for FTs}
	In case (B), taking $q=p^e>1$ to be the least power of $p$ such that $q\vv{c}\in \NN^n$, 
		Remark~\ref{rem: bound on distance between FT and "LCT"} shows that
		\begin{equation*}
			0<\lambda-\ftb(G)\le \frac{n-2}{q\deg G}< \frac{1}{q}.
		\end{equation*}
	If $\ftb(G)\in \QQ_{q}$ (which will be the case, \eg when $G$ is square free), then the above inequalities
		and Remark~\ref{Characterizations of truncations: R} show that $\ftb(G)=\tr{\lambda}{e}$.
	This reproduces a result obtained by N\'u\~nez-Betancourt, Witt, Zhang, and the first author,
		through completely different methods \cite{hernandez+others.fpt_quasi-homog.polys}.
	For more on this, see Theorem~\ref{thm: FPTs of quasi-homogeneous polynomials}.
\end{rem}
	
\begin{exmp}
	In the above remark, the conclusion that $\ftb(G)$ is a truncation of $\lambda$ does not hold under the looser assumption that 
		$\ftb(G)\in \QQ_{p^\infty}$. 
	If $G=(xy)^{49}((x+y)(x+2y)(x+4y))^{13}\in \FF_7[x,y]$ and $\idealm=\ideal{x,y}$, for instance, 
		then $\ftm(G)=4/343\in \QQ_{7^3}$ (determined by the critical point $(4,4,1,1,1)/7$ associated with $\ell=(x,y,x+y,x+2y,x+4y)$ 
		and $\idealm$), while $\tr{\lambda}{3}=\tr{2/137}{3}=5/343$.
\end{exmp}

We now specialize to the case where $\idealb=\idealm=\ideal{x,y}$.
Recall that $\ftm(G)$ is the \emph{$F$-pure threshold} of $G$, denoted by $\fpt(G)$.

\begin{rem}\label{rem: integral CPs w.r.t max ideal}	
		Proposition~\ref{prop: characterization of CP}\iref{item: second characterization of CP}  
		shows that the only critical points with integer coordinates associated with $\ell$ and $\idealm$ 
		are $\canvec_1,\ldots,\canvec_n$.
	Moreover, those are the only critical points  that have \emph{some} positive
		integer coordinate.
	Indeed, if $\vv{c}$ is a critical point and $c_i\in \NNpos$, then $\vv{c}\ge \canvec_i$, and therefore $\vv{c}=\canvec_i$.
\end{rem}

The scarcity of positive integer coordinates in critical points observed above has interesting consequences.
Case (A) becomes relegated to a  ``degenerate case'' where the multiplicity of some $\ell_i$ in $G$ is 
	too large---$\fpt(G)$ is determined by $\canvec_i$ if and only if $\lambda \vv{a}=(2/\deg G)\vv{a}>\canvec_i$ 
	if and only if $a_i>(\deg G)/2$, in which case $\fpt( G)=\max\{ 0, 1/a_i \} = 1/a_i$.
Note that the same conclusion is reached if $a_i=(\deg G)/2$ (so $\lambda\vv{a}\ge \canvec_i$), 
	but in that case $\fpt(G)= 1/a_i = \lambda$, so this falls actually under case (C).

In case (B), no nonzero coordinate of the critical point $\vv{c}$ is integral, so the minimal denominator of $\fpt(G)=\max\{c_i/a_i\}$ is of the form 
	$kp^e$, with $e\ge 1$ and $k$ a factor of one of the multiplicities $a_i$.
To put this observation in context, we must digress momentarily with some characteristic~0 considerations.

\begin{defn}\label{defn: good/bad prime}
	Let $G_0\in \QQ[x,y]$ be a non-constant form.
	We say that a prime~$p$ is a \emph{good prime} associated with $G_0$ if there exists a reduction modulo~$p$ of $G_0$ in $\FF_p[x,y]$, 
		which we denote by $G_p$, and the factorization of $G_p$ over $\closure{\FF}_p$ is similar to the factorization of $G_0$ over $\CC$, 
		in the sense that those factorizations have the same number of pairwise prime linear factors and the same multiplicities.  
	If that is not the case, we say that $p$ is a \emph{bad prime}.
\end{defn}

\begin{lem}\label{lem: finiteness of bad primes} 
	There are at most finitely many bad primes associated with a fixed non-constant form $G_0 \in \QQ[x,y]$.
\end{lem}

\begin{proof}  
	The proof relies on the following facts. 
	Given a finitely generated $\ZZ$-algebra $A$ containing $\ZZ$, for all but finitely many primes $p$ there exists a maximal ideal 
		of $A$ containing $p$.  
	Moreover, if $\mathfrak{M}$ is a maximal ideal of $A$ containing a prime number $p$, then $A/\mathfrak{M}$ 
		is a finite field of characteristic $p$.  
	For a justification of these facts, see, \eg \cite[Corollary~3.2]{hernandez.thesis}.

	We now proceed with the proof. 
	Suppose $G_0$ factors over $\CC$ as 
		\[ G_0 = \gamma \cdot x^ly^m(x-\alpha_1y)^{k_1}\cdots(x-\alpha_r y)^{k_r},\] 
		where $\gamma,\alpha_1, \ldots, \alpha_r\in\CC^\times$ and the $\alpha_i$ are distinct.  
	Let $A$ be the $\ZZ$-algebra generated by $\gamma$ and the $\alpha_i$, together with 
		$\prod_{i<j}(\alpha_i-\alpha_j)^{-1}\cdot(\gamma\cdot \alpha_1\cdots\alpha_r)^{-1}$.  
	According to the facts cited above, given a prime $p \gg 0$, there exists a maximal ideal $\mathfrak{M}$ of $A$ containing $p$.  
	Furthermore, for such $p$ and $\mathfrak{M}$, the field $\kk=A/\mathfrak{M}$ is a finite extension of $\FF_p$ 
		over which the image of $G_0$ has a factorization similar to that of $G_0$, because, by design, the images of $\gamma$ 
		and the $\alpha_i$  and $\alpha_i-\alpha_j$ in $\kk$ are nonzero, since they are invertible.
	So $p$ is a good prime.
\end{proof}

\begin{defn}\label{defn: degeneracy}
	A form $F$ in two variables over some field $K$ is \emph{degenerate}, of \emph{degeneracy type $m$},  if it has a linear factor 
		over $\closure{K}$ with multiplicity $m>(\deg F)/2$.  Note that each degenerate form has a unique degeneracy type.
\end{defn}

Let $G_0\in \QQ[x,y]$ be a non-constant form.  
As discussed in the introduction, the log canonical threshold of $G_0$, denoted by $\lct(G_0)$, is an 
	invariant measuring the singularity of $G_0$ at the origin, and is defined via a log resolution of singularities.  
In the context of this paper, the most important property of $\lct(G_0)$ is that $\lim_{p \to \infty} \fpt(G_p) = \lct(G_0)$.  

If $G_0$ is degenerate, of degeneracy type $m$, then for each good prime $p$ the computation of $\fpt(G_p)$ falls under case (A), 
	and $\fpt(G_p)=1/m$.
Consequently, $\lct(G_0)=\lim_{p\to \infty}\fpt(G_p)=1/m$.
If $G_0$ is non-degenerate, then for each good prime $p$ the computation of $\fpt(G_p)$ falls under cases (B) or (C), 
	and the inequalities in Remark~\ref{rem: truncation for FTs} show that $\lct(G_0)=\lim_{p\to \infty}\fpt(G_p)=\lambda=2/\deg(G_0)$.
Thus, in the paragraph before Definition~\ref{defn: good/bad prime} we have shown the following result, which provides 
	an affirmative answer to Question~\ref{question: Schwede} in the two-variable homogeneous setting.

\begin{thm}\label{thm: p in denominator}
	Let $G_0\in \QQ[x,y]$ be a non-constant form. 
	Let $p$ be a good prime associated with $G_0$, and let $G_p$ be the image of 
		$G_0$ in $\FF_p[x,y]$.
	If $\fpt(G_p)\ne \lct(G_0)$, then the minimal denominator of $\fpt(G_p)$ is of the form $kp^e$, where $e\ge 1$ and 
		$k$ divides the multiplicity of some linear factor \textup(over $\CC$\textup) of $G_0$.
	\qed
\end{thm}

\begin{rem}\label{rem: weakening}
	If we weaken the notion of good prime, requiring only that~$G_0$ and~$G_p$ be both non-degenerate or both degenerate, 
		of the same degeneracy type, then an alternate version of the above theorem still holds, where the conclusion states that 
		\emph{$k$ divides the multiplicity of some linear factor \textup(over $\closure{\FF}_p$\textup) of $G_p$}.
\end{rem}

\begin{exmp}
\label{example: lowcharacteristic}
	Let $G_0=x(x+y)(x+6y)$.
	Then $G_5=x(x+y)^2$, so $\fpt(G_5)=1/2\ne 2/3=\lct(G_0)$, and yet $\fpt(G_5)$ has  a denominator prime to 5. 
	This shows the need for requiring $p$ to be a good prime in Theorem~\ref{thm: p in denominator}.
	That result may otherwise not hold when the factorizations of $G_0$ and $G_p$ are 
		``too different''.
\end{exmp}

Going back to the three cases discussed earlier in this section, while case (A) is clearly delimited, distinguishing between cases (B) and (C) is delicate.
For that, it is useful to know an upper bound for the denominator of the critical point 
	that determines $\fpt(G)$ in case (B).
When $\deg G$ is prime to $p$ we have such a bound---a consequence of the 
	``forbidden intervals'' theorem of Blickle, Musta{\c{t}}\u{a}, and Smith \cite[Proposition~4.3]{blickle+mustata+smith.F-thresholds_hyper},
	generalized by the first author  \cite[Proposition~4.8]{hernandez.F-purity_of_hypersurfaces}, 
	which states that for any $\beta\in (0,1)_{q}$  there are no $F$-pure thresholds 
	of hypersurfaces in characteristic $p$ in the interval $(\beta, \beta q/(q-1))$.

\begin{lem}\label{lem: forbidden interval}
	Let  $\lambda=a/b\in (0,1]\cap\QQ$.
	Suppose $b$  is prime to $p$, and let $\mu$ be the multiplicative order of $p$ in $(\ZZ/b\ZZ)^\times$.
	Then no $F$-pure threshold of a polynomial over a field of characteristic $p$ lies in the interval $(\tr{\lambda}{\mu},\lambda)$.
\end{lem}

\begin{proof}
	Let  $q=p^\mu$ and $k= (q-1)\lambda$; then $k\in\NN$ and 
		\[\lambda=\frac{k}{q-1}=\frac{k}{q}\left(1+\frac{1}{q}+\frac{1}{q^2}+\cdots\right).\]
	Since $k<q$, the above equation shows that $\tr{\lambda}{\mu}=k/q$, so $\lambda=\tr{\lambda}{\mu}q/(q-1)$,
		and the ``forbidden intervals'' theorem gives the result.
\end{proof}

\begin{prop}
	Let $G\in \kk[x,y]$ be a form of degree $d>0$.
	Suppose $\fpt(G)<\lambda=2/d$ and the minimal denominator of $\lambda$ is prime to $p$.
	Let~$\mu$ be the multiplicative order of $p$ modulo that denominator.
	Then $\fpt(G)$ is determined by a critical point with coordinates in $\QQ_{p^\mu}$.
\end{prop}

\begin{proof}
	Write $G=\ell^\vv{a}$, as in the beginning of this section, and let $\vv{c}$ be the critical point that determines $\fpt(G)$.  
	Lemma~\ref{lem: forbidden interval} gives us the following inequalities:
		\[\max \left \{ \frac{c_1}{a_1}, \ldots, \frac{c_n}{a_n} \right \}=\fpt(G)\le \tr{\lambda}{\mu}<\lambda=\frac{2}{d}.\]
	As $\max\{c_i/a_i\} \cdot \vv{a} \in \boundary [\vv{c},\vvv{\infty})$ and 
		$\frac{2}{d} \cdot \vv{a} = \frac{2}{\norm{\vv{a}}} \cdot \vv{a} \in \boundary \Trivial$, multiplying  
		each of the terms in the above inequalities by $\vv{a}$ shows that $\tr{\lambda}{\mu}\vv{a}$ lies in $\oball{\Dlm(\vv{c})}{\vv{c}}$,
		the region where the function $\Dlm$ attached to $\ell$ and $\idealm$ is determined by $\vv{c}$.
	Because $\tr{\lambda}{\mu}\vv{a}\in \QQ_{p^\mu}^n$, Remark~\ref{rem: don't need to go to finer mesh to find local max}
		tells us that $\vv{c}\in \QQ_{p^\mu}^n$ as well. 
\end{proof}

We summarize the above observations in the following theorem.

\begin{thm}\label{thm: summary homogeneous case}
	Let $G\in \kk[x,y]$ be a form of degree $d>0$.
	Write $G=\ell_1^{a_1}\cdots\ell_n^{a_n}$, where the $\ell_i$ are pairwise prime linear forms in $\closure{\kk}[x,y]$
		and $a_i>0$, for each $i$. 
	\begin{enumerate}[(1)]
		\item If $a_i\ge d/2$, for some $i$, then $\fpt(G)=1/a_i$.
		\item	Suppose $a_i< d/2$, for each $i$.
		Set $\ell=(\ell_1,\ldots,\ell_n)$ and $\vv{a}=(a_1,\ldots,a_n)$. 
		Then exactly one of the following holds\textup:
		\begin{itemize}[wide=0pt]
			\item	$\tr{\frac{2}{d} \cdot \vv{a}}{e}$ lies in the upper region attached to $\ell$ and $\idealm=\ideal{x,y}$, 
					for some $e\ge 1$. 
				Thus, there exists a unique
					critical point $\vv{c}\in \QQ_{p^e}^n$ associated with $\ell$ and $\idealm$
					such that $\vv{c}\le \tr{\frac{2}{d} \cdot \vv{a}}{e}$, and  
					\[ \fpt(G)= \max \left \{ \frac{c_1}{a_1}, \ldots, \frac{c_n}{a_n} \right \} .\]
				The critical point $\vv{c}$ has no nonzero integer coordinates, 
					and thus the minimal denominator of $\fpt(G)$ is of the form $kp^m$, where $1\le m\le e$ and 
					$k$ is a factor of some multiplicity $a_i$.
				Moreover, if the minimal denominator of $2/d$ is prime to $p$, then the above holds for some $e$ no greater than   
					the multiplicative order of $p$ modulo that denominator.
				Finally, if $\fpt(G)\in \QQ_{p^e}$ \textup(\eg if $G$ is square free\textup), then $\fpt(G)=\tr{2/d}{e}$.					
			\item $\fpt(G)=2/d$.
				\qed
		\end{itemize}
	\end{enumerate}
\end{thm}

\begin{exmp}\label{exmp: computation of fpts}
	Consider the forms $G_1=x^2y^2(x^2+2xy+3y^2)^7$ and $G_2=x^2y^2(x^2+2xy+3y^2)$ in $\mathbb{F}_{25}[x,y]$.
	Let $\ell_1=x$, $\ell_2=y$, and $\ell_3\ell_4=x^2+2xy+3y^2$, and  set $\ell=(\ell_1,\ell_2,\ell_3,\ell_4)$. 
	Finally, let $\vv{a}_1=(2,2,7,7)$, $\vv{a}_2=(2,2,1,1)$, $\lambda_1=2/\deg G_1=1/9$, and 
		$\lambda_2=2/\deg G_2=1/3$.	
	Since the multiplicative order of 5 (mod 9) is 6, to find $\fpt(G_1)$ we look for a truncation $\tr{\lambda_1\vv{a}_1}{e}$
		with $e\le 6$ that lies in the upper region $\Upper$ attached to~$\ell$ and~$\idealm$.
	We find that $\tr{\lambda_1\vv{a}_1}{3}=(27, 27, 97, 97)/125$ lies in $\Upper$, and is itself a critical point.
	Thus $\fpt(G_1)=  \max \bigl\{  \frac{27}{125 \cdot 2}, \frac{97}{125 \cdot 7}   \bigr \} = \frac{97}{875}$.
	As for $G_2$, the multiplicative order of~5 (mod 3) is 2, but $\tr{\lambda_2\vv{a}_2}{2}$ does not
		lie in $\Upper$, so $\fpt(G_2)$ is determined by the trivial region:
		$\fpt(G_2)=\lambda_2=1/3$.
	Figure~\ref{fig: 2D slice} shows a density plot of a section of the function $\Dlm$ attached to $\ell$ and $\idealm$, 
		together with the lines spanned by~$\vv{a}_1$ and~$\vv{a}_2$.
	\begin{figure}
		\centering
		\includegraphics[width=.6\textwidth]{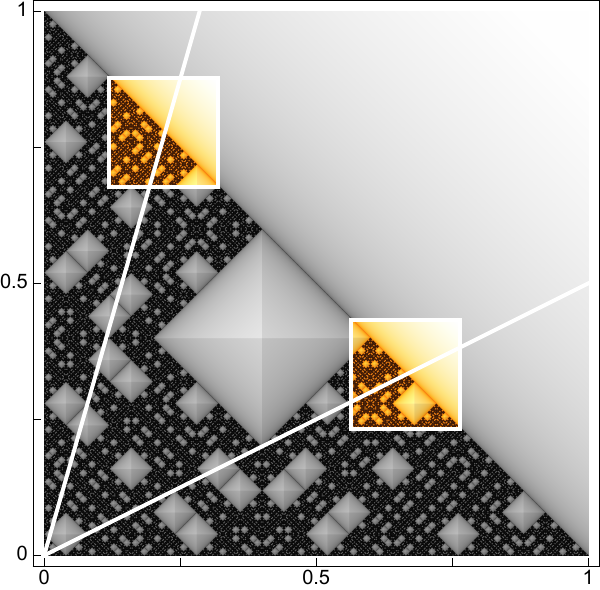} \\[5mm]
		\includegraphics[width=.45\textwidth]{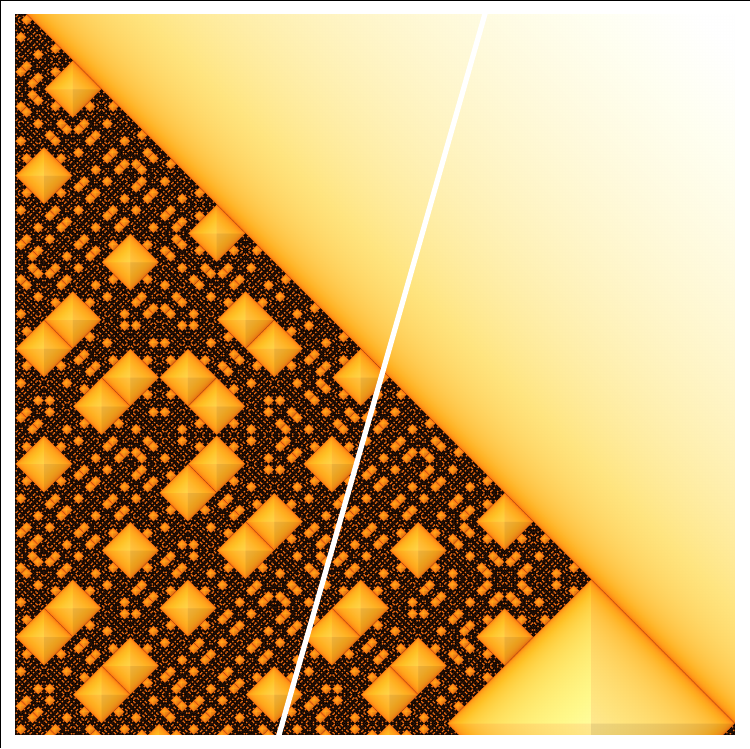}\hskip .1\textwidth
		\includegraphics[width=.45\textwidth]{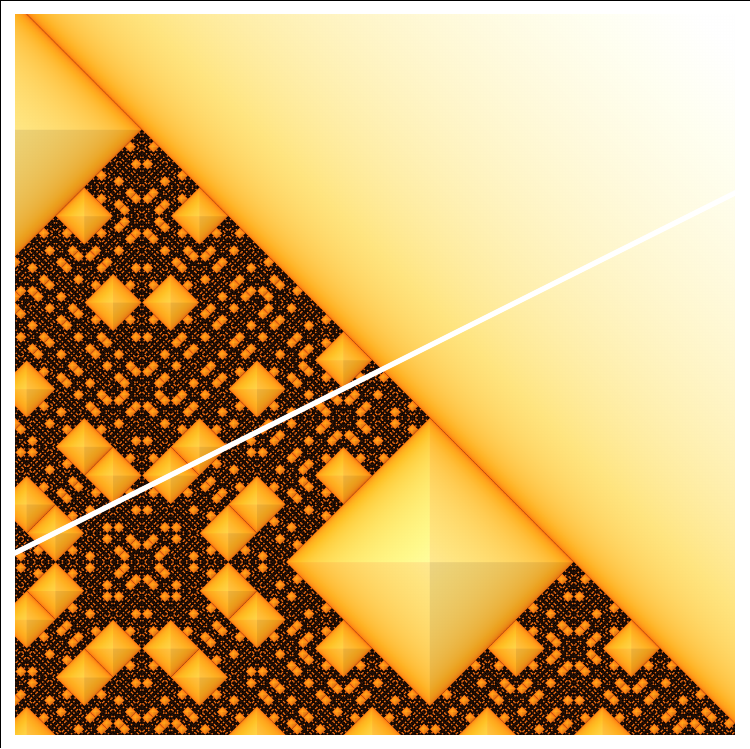}
		\caption{The two dimensional section $(u,v)\mapsto\Dlm(u,u,v,v)$ of $\Dlm$ and the lines spanned by $\vv{a}_1$ and $\vv{a}_2$}
		\label{fig: 2D slice}
	\end{figure}
\end{exmp}

The method used in the above example can be applied to any form in two variables, and leads to an efficient algorithm
	that has been implemented by the second author in the \emph{Macaulay2} \cite{M2} package \emph{PosChar} \cite{poschar}.
More details about this algorithm and its implementation are presented in Appendix~\ref{appendix: algorithm}.

We close this section highlighting a key difference between $F$-pure thresholds  and $F$-thresholds with respect to ideals $\idealb\ne \idealm$, 
	which is that in the latter setting the analogue of Remark~\ref{rem: integral CPs w.r.t max ideal} does not hold.
Integral critical points tend to abound (see, for instance, Example~\ref{exmp: 2D example})
	and, as the following example will show, critical points may have both nonzero integral coordinates and non-integral 
	coordinates---so $F$-thresholds with a denominator prime to $p$ may arise from non-integral critical points.
		
\begin{exmp}
	Let $\ell=(x,y,x+y,x+2y)\in \FF_5[x,y]^4$ and $\idealb=\ideal{x,y^2}\subseteq \FF_5[x,y]$.
	Then $\vv{c}=(2/5,3/5,4/5,1)$ is a critical point associated with $\ell$ and $\idealb$,
		and it is easy to produce forms whose $F$-thresholds with respect to $\idealb$ are
		determined by the last coordinate of $\vv{c}$ and have a denominator prime to 5.
	For instance, if $G=x^7y^{10}(x+y)^{13}(x+2y)^{16}$, then 
		$\ftb(G)=\max \bigl\{ \frac{2}{5 \cdot 7}, \frac{3}{5 \cdot 10},\frac{4}{5 \cdot 13},\frac{1}{16} \bigr \}=\frac{1}{16}$.
\end{exmp} 

\section{\texorpdfstring{$\bm F$-pure thresholds of quasi-homogeneous polynomials in two variables}
	{F-pure thresholds of quasi-homogeneous polynomials in two variables}}\label{s: QH case}

If $K$ is an arbitrary field, we endow $K[X,Y]$  with a non-standard $\NN$-grading where $\deg X=u$, $\deg Y= v$, and $uv\ne 0$.
We shall refer to the homogeneous elements under this grading as \emph{quasi-homogeneous polynomials} and reserve the 
	terms \emph{form} and \emph{homogeneous polynomial} for polynomials that are homogeneous under the standard grading. 
We extend the notion of degeneracy (see Definition~\ref{defn: degeneracy}) to this setting by saying that a quasi-homogeneous 
	polynomial $f\in K[X,Y]$ is \emph{degenerate}, of \emph{degeneracy type $m$}, if it has an irreducible factor over $\closure{K}$ with 
	multiplicity $m>\deg(f)/(u+v)$.

\begin{notation}
	Given $f\in K[X,Y]$ and an irreducible polynomial $h\in \closure{K}[X,Y]$, $\mult_h(f)$ denotes the multiplicity of $h$ in $f$, that is, 
		the largest $m\in \NN$ (possibly~0) such that $h^m$ divides $f$ in $\closure{K}[X,Y]$. 
\end{notation}

\begin{defn}
	Let $g_0\in \QQ[X,Y]$ be a non-constant quasi-homogeneous polynomial.
	A prime $p$ is a \emph{good prime} associated with $g_0$ if the following hold: 
	\begin{enumerate}[(1)]
		\item there exists a reduction modulo~$p$ of $g_0$ in $\FF_p[X,Y]$, denoted by $g_p$; 
		\item the factorization of $g_p$ over $\closure{\FF}_p$ is similar to the factorization of $g_0$ over $\CC$, 
				in the sense that those factorizations have the same number of pairwise prime irreducible factors and the same multiplicities;
		\item $\mult_X(g_p)=\mult_X(g_0)$ and $\mult_Y(g_p)=\mult_Y(g_0)$.  
	\end{enumerate}
	If that is not the case, then $p$ is a \emph{bad prime}.
	Remark~\ref{rem: finiteness of bad primes QH} will show that there exist at most finitely many bad primes 
		associated with a fixed $g_0$.
\end{defn}

Our first and main goal in this section will be to extend Theorem~\ref{thm: p in denominator} to the quasi-homogeneous setting: 

\begin{thm}\label{thm: p in denominator QH}
	Let $g_0\in \QQ[X,Y]$ be a non-constant quasi-homogeneous polynomial. 
	Let $p$ be a good prime associated with $g_0$, and let $g_p$ be the image of $g_0$ in $\FF_p[X,Y]$.
	If $\fpt(g_p)\ne \lct(g_0)$, then the minimal denominator of $\fpt(g_p)$ is of the form $kp^e$, where $e\ge 1$ and 
		$k$ is a factor of one of the following\textup: $\mult_X(g_0)\cdot \deg X$, $\mult_Y(g_0)\cdot \deg Y$, or $\mult_h(g_0)$, 
		where $h$ is some irreducible factor \textup(over $\CC$\textup) of $g_0$ other than $X$ or $Y$.
\end{thm}

\begin{rem}
	As it is the case with Theorem~\ref{thm: p in denominator} (see Remark~\ref{rem: weakening}),
		an alternate version of the above result, where, in the conclusion,
		multiplicities in $g_0$ are replaced with multiplicities in $g_p$,
		may be obtained under a looser notion of good prime, that 
		only requires that $g_0$ and $g_p$ be both non-degenerate or both degenerate, of the same degeneracy type. 
\end{rem}

Fix $g_0$ as in the statement of Theorem~\ref{thm: p in denominator QH} and a good prime $p$; let $\kk=\closure{\FF}_p$ and 
	$g=g_p$, the image of $g_0$ in $\kk[X,Y]$.
Since Theorem~\ref{thm: p in denominator QH} is already known for standard homogenous polynomials, we may assume that $g_0$  
	is not homogeneous---so $g_0$ is not a monomial, and thus neither is $g$, and  $u\ne v$. 
In fact we assume, by possibly changing the grading, that $u$ and $v$ are coprime. 
The following proposition will allow us to extend our methods to the quasi-homogeneous setting:

\begin{prop}\label{prop: relating H and QH cases}
	Let $\psi: \kk[X,Y]\rightarrow \kk[x,y]$ be the map $f(X,Y) \mapsto f(x^u,y^v)$.
	Set $G=\psi(g)$ and $\idealb=\ideal{x^u,y^v}$.
	Then $\fpt(g)=\ftb(G)$.
\end{prop}

\begin{proof}  
	Because of the description of the $F$-threshold of a polynomial given in Discussion~\ref{discussion:  FT description for polynomials}, 
		it suffices to show that, for each $a \in \NN$ and $q = p^e$, we have that $g^a \in \ideal{X,Y}^{[q]}$ 
		if and only if $G^a \in \idealb^{[q]}$.  
	Let $A=\kk[x^u,y^v]$ and $B=\kk[x,y]$.
	Let~$\ideala$ be the ideal of $A$ generated by $x^u$ and $y^v$.
	Then $\psi$ induces an isomorphism from $\kk[X,Y]$ to $A$, mapping $g$ to $G$ and $\idealm=\ideal{X,Y}$ to $\ideala$, 
		and hence, $g^a \in \idealm^{[q]}$ if and only if $G^a \in \ideala^{[q]}$.  
	But $A$ is a direct summand of  $B$ as an $A$-module, and therefore $\idealb^{[q]}\cap A=(\ideala^{[q]}B)\cap A=\ideala^{[q]}$.
	As $G\in A$,  we see that $G^a\in \ideala^{[q]}$ if and only if $G^a\in \idealb^{[q]}$, which allows us to conclude the proof.
\end{proof}

\begin{lem}\label{lem: QH factorization}
	The polynomial $g$ can be factored as
		\[g=\xi \cdot X^{j_1}Y^{j_2}\cdot \prod_{i=1}^{m}(X^v-\mu_i Y^u)^{k_i},\]
		for some  $j_1,j_2\in \NN$, $m,k_1,\ldots, k_m\in \NNpos$, 
		$\xi\in \kk^\times$,  and distinct $\mu_1,\ldots,\mu_m\in \kk^\times$.
\end{lem}

\begin{rem}\label{rem: finiteness of bad primes QH}
	The polynomial $g_0$, of course, has a similar factorization over $\CC$, and the argument used in the proof of 
		Lemma~\ref{lem: finiteness of bad primes} can be adapted to show that there exist at most finitely
		many bad primes associated with $g_0$.
\end{rem}

\begin{proof}
	Write $g$ as $X^{j_1}Y^{j_2}h$, where $h$ is a quasi-homogeneous polynomial prime to $XY$.
	As $g$ is not a monomial, $h$ has at least terms $X^a$ and $Y^b$, with $au=bv$.
	Since $u$ and $v$ are coprime, $v\mid a$ and $u\mid b$.
	Now suppose $h$ also has a term $X^cY^d$.
	Then $cu+dv=au$, so $v\mid c-a$; but $v\mid a$ as well, so we conclude that $v\mid c$.
	Similarly we find that $u\mid d$, so in each term of $h$ the exponents of~$X$ and~$Y$ are divisible by $v$ and $u$, 
		respectively, and therefore $h=H(X^v,Y^u)$ for some form~$H$.
	The result is then obtained by factoring $H$ into linear forms.
\end{proof}

In the remainder of this section, we adopt the notation introduced in Proposition~\ref{prop: relating H and QH cases}.
In our proof of Theorem~\ref{thm: p in denominator QH}, we assume that $j_1j_2\ne 0$; our argument can be adapted to 
	handle the other cases, which are simpler.
As $u = \deg X$ and $v = \deg Y$ are coprime, $p$ can only divide one of them;
	we assume that $u$ is prime to $p$ and 
	write $v=w\bar{q}$, where $w$ is prime to $p$ and $\bar{q}$ is a power of $p$ (possibly~1).
Then, as $\kk = \closure{\kk}$, there exist unique $\mu_i^{1/\bar{q}} \in \kk$ with 
	$\bigl( \mu_i^{1/\bar{q}} \bigr)^{\bar{q}} = \mu_i$, so that 
	\begin{equation} 
	\label{eq: first factorization of G}
	G=\psi(g)=\xi \cdot  x^{uj_1}y^{vj_2}\cdot \prod_{i=1}^{m}\left(x^{uw}-\mu_i^{1/\bar{q}} y^{uw}\right)^{\bar{q}k_i},
	\end{equation}
	where the factors $x^{uw}-\mu_i^{1/\bar{q}} y^{uw}$ are square free and pairwise prime.
As our methods will depend on factoring $G$ into a product of linear forms, it will be necessary to factor each  
	$x^{uw}-\mu_i^{1/\bar{q}} y^{uw}$ into a product of linear forms.  
Let $\zeta$ be a primitive $(uw)$th root of unity in $\kk$, and let $\nu_i$ be a $(uw)$th root of 
	$\mu_i^{1/\bar{q}}$ in $\kk$, for each $i$; 
	then $x^{uw}-\mu_i^{1/\bar{q}} y^{uw}=\prod_{j=1}^{uw}(x-\nu_i\zeta^jy)$.  
Substituting this into~\eqref{eq: first factorization of G} produces the following factorization of $G$ into a product of linear forms:  		
	\begin{equation*}
		G= \xi \cdot  x^{uj_1}y^{vj_2}\cdot \prod_{i=1}^{m}  \prod_{j=1}^{uw}(x-\nu_i\zeta^jy)^{\bar{q} k_i}.
	\end{equation*}
Let $n=2+muw$, the number of pairwise prime linear factors of $G$.
It will be convenient to label the $n$ linear factors of $G$ and the canonical basis vectors of $\RR^n$
	in a non-standard way.
Let $\ell_1=x$, $\ell_2=y$, and for each $i\in\{1,\ldots, m\}$ and $j\in \{1,\ldots,uw\}$ let $\ell_{i,j}=x-\nu_i\zeta^jy$; set  
	\[\ell=(\ell_1,\ell_2,\ell_{1,1},\ldots,\ell_{1,uw},\ldots,\ell_{m,1},\ldots,\ell_{m,uw}).\]
For each $\ell_{i,j}$, let $\vv{f}_{i,j}$ be the corresponding  canonical basis vector of $\RR^n$, and 
	set $\vv{f}_i=\sum_{j=1}^{uw}\vv{f}_{i,j}$, so that $\ell^{\vv{f}_i}= \prod_{j=1}^{uw}(x-\nu_i\zeta^jy) = x^{uw}-\mu_i^{1/\bar{q}} y^{uw}$, by definition.
The canonical basis vectors corresponding to $\ell_1$ and $\ell_2$ will be denoted by the usual $\canvec_1$ and $\canvec_2$.

We now explore some symmetries in the critical points associated with $\ell$ and $\idealb=\ideal{x^u,y^v}$ coming from the special shape of $G$.

\begin{lem}\label{lem: cyclic permutations}
	Suppose $\vv{c}=c_1\canvec_1+c_2\canvec_2+\sum_{i,j}c_{i,j}\vv{f}_{i,j}$ is a critical point.
	Let $\vv{c}'$ be the point obtained from $\vv{c}$ by replacing each set of coordinates $c_{i,1},c_{i,2},\ldots,c_{i,uw}$ 
		with $c_{i,uw},c_{i,1},\ldots,c_{i,uw-1}$.
	Then $\vv{c}'$ is also a critical point. 
\end{lem}

\begin{proof}
	Suppose $\vv{c}\in \QQ_{q}^n$.
	Then the $\kk$-automorphism of $\kk[x,y]$ that maps $x\mapsto x$ and $y\mapsto \zeta y$ transforms  
		$\ell^{q\vv{c}}$ into a constant multiple of $\ell^{q\vv{c}'}$, while fixing $\idealb$.	
\end{proof}

\begin{lem}\label{lem: CPs have blocks of identical coordinates}
	Let $\vv{t}=t_1\canvec_1+t_2\canvec_2+\sum_{i=1}^m t_i^*\vv{f}_i\in\boundary \Trivial$.
	Suppose $\vv{c}$ is a critical point such that $\vv{c}<\vv{t}$.
	Then $\vv{c}=c_1\canvec_1+c_2\canvec_2+\sum_{i=1}^m c_i^*\vv{f}_i$, for some $c_i,c_i^*\in\QQ_{p^\infty}$.
\end{lem}	

\begin{proof}
	Write $\vv{c}=c_1\canvec_1+c_2\canvec_2+\sum_{i,j}c_{i,j}\vv{f}_{i,j}$, and 
		construct $\vv{c}'$ as in Lemma~\ref{lem: cyclic permutations}.
	Then $\vv{c}'$ is a critical point and $\vv{c}'<\vv{t}$ as well.
	Since there can be no more than one critical point lying under a point in $\boundary \Trivial$ 
		(see Remark~\ref{rem: uniqueness of CP under point in boundary of trivial region}), 
		we conclude that $\vv{c}=\vv{c}'$.
	Iterating this process, we see that for each $i\in\{1,\ldots,m\}$ the coordinates $c_{i,1},\ldots,c_{i,uw}$ of $\vv{c}$ are all equal, 
		and the result follows.
\end{proof}

\begin{lem}\label{lem: analysis of block CPs}
	Suppose $\vv{c}=\alpha\canvec_1+\beta\canvec_2+\sum_{i=1}^m \gamma_i\vv{f}_i\in \QQ_q^n$ is a critical point.
	Then $\alpha/u$ and $\beta/w$  both lie in $\QQ_q$.
\end{lem}

\begin{proof}
	Suppose $\alpha\ne 0$.
	Write $\alpha=a/q$, $\beta=b/q$, and $\gamma_i=c_i/q$. 
	As $\vv{c}\in \Upper$,  
		\begin{equation}\label{eq: c is in upper region}
			\ell^{q\vv{c}}=x^ay^b\cdot \prod_{i=1}^m\left(x^{uw}-\mu_i^{1/\bar{q}}y^{uw}\right)^{c_i}\in \idealb^{[q]}=\ideal{x^{uq},y^{vq}}.
		\end{equation}
	Since $\vv{c}$ is a critical point, there is a monomial $M=x^{a+iuw}y^{b+juw}$ in the support of $\ell^{q\vv{c}}$ such that 
		$M/x\notin \idealb^{[q]}$.
	Clearly  $a+iuw\le uq$, and 
		if the inequality were strict, then~\eqref{eq: c is in upper region} would imply that $b+juw\ge vq$, and 
		$M/x$ would be in $\idealb^{[q]}$, a contradiction.
	So $a+iuw=uq$, whence $\alpha/u=a/(uq)=1-iw/q\in \QQ_q$.
	The argument showing that $\beta/w\in \QQ_q$ is analogous. 
\end{proof}

\begin{cor}\label{cor: analysis of block CPs}
	Suppose $\vv{c}=\alpha\canvec_1+\beta\canvec_2+\sum_{i=1}^m \gamma_i\vv{f}_i$ is a critical point.
	\begin{enumerate}[(1)]
		\item If $\alpha$ is a positive integer, then $\vv{c}=u\canvec_1$.
		\item If $\beta$ is a positive integer, then either $\vv{c}=v\canvec_2$ or $\bar{q}\nmid \beta$.
	\end{enumerate}
\end{cor}

\begin{proof}
	If $\alpha\in \NNpos$, then, bearing in mind that $u$ is prime to $p$, Lemma~\ref{lem: analysis of block CPs} shows that $u\mid \alpha$.
	Thus, $\vv{c}\ge u\canvec_1$, and since $u\canvec_1$ is a critical point we must have $\vv{c}=u\canvec_1$. 	
	Likewise, if $\beta\in \NNpos$,  then $w\mid \beta$.
	If $\bar{q}\mid \beta$ as well, then $v=w\bar{q}\mid \beta$, and arguing as before we conclude that $\vv{c}=v\canvec_2$. 
\end{proof}

We are now ready to conclude the proof of Theorem~\ref{thm: p in denominator QH}. 

\begin{proof}[Proof of Theorem~\ref{thm: p in denominator QH}]
	Recall that $\fpt(g)=\ftb(G)$, by Proposition~\ref{prop: relating H and QH cases}.  
	Set $\lambda=(u+v)/\deg G$ and write $G=\xi\ell^\vv{a}$, where 
		$\vv{a}=uj_1\canvec_1+vj_2\canvec_2+\sum_{i=1}^m \bar{q}k_i \vv{f}_i$. 
	Consider the following cases: 	
	\begin{itemize}[wide=0pt]
		\item	\textbf{$\bm{\lambda k_i> 1}$ for some $\bm i$.}
		This situation is not common---$\lambda k_i> 1$ is equivalent to $\deg G<(u+v)k_i$ and, as $\deg G\ge uvk_i$, this requires
			either $u$ or $v$ to be 1. 
		In this case, $\bar{q}\vv{f}_i$ is a critical point lying under $\lambda \vv{a}$, so
			$\fpt(g)= \max\bigl\{0,\frac{\bar{q}}{\bar{q}k_i}\bigr\} =\frac{1}{k_i}$.
		\item \textbf{$\bm{\lambda j_i> 1}$ for some $\bm i$.}
		Suppose $\lambda j_1> 1$ (the other case is analogous). 
		Then $\lambda \vv{a}$ lies above the critical point $u\canvec_1$, so 
			$\fpt(g)= \max\bigl\{ \frac{u}{u j_1}, 0\bigr\} = \frac{1}{j_1}$.
		\item	\textbf{$\bm{\lambda j_i\le 1}$ and $\bm{\lambda k_i\le 1}$, for each $\bm i$.}
		If $\ftb(G)$ is determined by a critical point $\vv{c}$, then $\vv{c}<\lambda \vv{a}$, 
			so $\vv{c}=\alpha\canvec_1+\beta\canvec_2+\sum_{i=1}^m \gamma_i\vv{f}_i$, for some $\alpha,\beta,\gamma_i\in\QQ_{p^\infty}$, 
				by Lemma~\ref{lem: CPs have blocks of identical coordinates}.
		The inequalities $\lambda j_i\le 1$ ($i=1,2$) ensure that $\vv{c}$ is neither $u\canvec_1$ nor $v\canvec_2$, and thus 
			Corollary~\ref{cor: analysis of block CPs} shows that $\alpha$ is not a positive integer, and that if $\beta$ is a positive integer, then 
			$\bar{q}\nmid \beta$.
		The inequalities $\lambda k_i\le 1$ ($i=1,\ldots,m$), on the other hand, ensure that $\gamma_i<\lambda \bar{q}k_i\le \bar{q}$. 
		Being determined by $\vv{c}$, $\fpt(g)=\ftb(G)$ equals the maximum among 
			$\alpha/(uj_1)$, $\beta/(\bar{q}wj_2)$, and $\gamma_i/(\bar{q}k_i)$ 
			($i=1,\ldots,m$), hence its minimal denominator has the desired form.
		Alternatively, $\ftb(G)$ may be determined by the trivial region, and  $\fpt(g)=\ftb(G)=\lambda$.
	\end{itemize}
	
	We now allow $p$ to vary, to find $\lct(g_0)$ in each of the above cases.
	Note that the conditions ``$\lambda j_i>1$'' and ``$\lambda k_i>1$'' are equivalent to degeneracy conditions 
		on $g_p$, which are inherited from $g_0$, and thus independent of the choice of the good prime~$p$.
	It follows that $\lct(g_0)=\fpt(g_p)$ in the first two cases.
	As for the last case, note that if $p\nmid v$ then $\alpha$ and $\beta$ cannot be nonzero integers, by Corollary~\ref{cor: analysis of block CPs}, 
		and $\gamma_i<\bar{q}=1$, so $\vv{c}\notin \NN^n$.
	Thus,  $\fpt(g_p)$ is either $\lambda$ or is determined by a non-integral critical point, for all $p\gg 0$,
		and  the inequalities in Remark~\ref{rem: truncation for FTs} show that $\lct(g_0)=\lim_{p\to \infty}\fpt(g_p)=\lambda$. 
	So we have shown 
		that the minimal denominator of $\fpt(g_p)$ has the required form whenever $\fpt(g_p)\ne \lct(g_0)$.
\end{proof}

A byproduct of our proof is the following:

\begin{cor}
	If $g_0$ is non-degenerate, then $\lct(g_0)=(u+v)/\deg(g_0)$.
	If $g_0$ is degenerate, of degeneracy type $m$, then $\lct(g_0)=1/m$. 
	\qed
\end{cor}

The following theorem was recently proved by N\'u\~nez-Betancourt, Witt, Zhang, and the first author.

\begin{thm}[{\cite[Theorem~4.4]{hernandez+others.fpt_quasi-homog.polys}}]\label{thm: FPTs of quasi-homogeneous polynomials}
	Let $g\in \kk[X,Y]$ be a non-constant quasi-homogeneous polynomial that is square free over~$\closure{\kk}$ .
	Set $\lambda=(u+v)/\deg g$. 
	Then either $\fpt(g)=\min\{1,\lambda\}$ or $\fpt(g)=\tr{\lambda}{e}$, for some $e\ge 1$.
\end{thm}

We conclude this section and the paper with a simple proof of this theorem 
	under the additional assumption that $u$ and $v$ are prime to $p$.

\begin{proof}
	In view of  Theorem~\ref{thm: summary homogeneous case}, we may assume that $g$ is not homogeneous---so $g$ 
		is not a monomial and $u\ne v$.
	As before, we assume that $u$ and $v$ are coprime and $\kk=\closure{\kk}$.
	Using Lemma~\ref{lem: QH factorization} and the assumption that $g$ is square free, we write 
		\[G=\psi(g)= \xi\cdot x^{ju} y^{kv}\cdot\prod_{i=1}^{m}(x^{uv}-\mu_i y^{uv}),\]
		where $j,k\in \{0,1\}$ and the factors $x^{uv}-\mu_i y^{uv}$ are square free and pairwise prime. 
	We consider the case $j=k=1$; the other cases are analogous.
	Set $\lambda=(u+v)/\deg G$ and $\idealb=\ideal{x^u,y^v}$ and, adopting the setup used earlier 
		(minding that here $v=w$ and $\bar{q}=1$), write $G=\xi\ell^\vv{a}$, 
		where $\vv{a}=u\canvec_1+v\canvec_2+\sum_{i=1}^m \vv{f}_i$.
	
	If $\ftb(G)\ne \lambda$, then $\ftb(G)$ is determined by a critical point $\vv{c}<\lambda \vv{a}$.
	Since $\lambda <1$,  we find that $\vv{c}\not\in \NN^n$,\footnote{
		Note that if $x$ or $y$ are not factors of $g$ (\ie $jk=0$), then $\lambda$ may---in some rare instances---be 
		$\ge 1$, in which case $\ftb(G)$ is determined  by an integral critical point and  $\ftb(G)=1$.} 
		and Lemmata~\ref{lem: CPs have blocks of identical coordinates} and~\ref{lem: analysis of block CPs} show that
			$\vv{c}=\alpha\canvec_1+\beta\canvec_2+\sum_{i=1}^m \gamma_i\vv{f}_i$ for some $\alpha,\beta,\gamma_i\in \QQ_{q}$,
			where $q=p^e>1$, 
			and $\ftb(G)=\max\{\alpha/u,\beta/v,\gamma_i\} \in \QQ_q$.
	Remark~\ref{rem: truncation for FTs} then shows that $\ftb(G)=\tr{\lambda}{e}$.
	As $\fpt(g)=\ftb(G)$, by Proposition~\ref{prop: relating H and QH cases}, we have shown that either $\fpt(g)=\lambda=\min\{1,\lambda\}$ or 
		$\fpt(g)=\tr{\lambda}{e}$, for some $e\ge 1$.
\end{proof}
 
\section*{Acknowledgements}

This paper was written while the second author was visiting the University of Utah, during a sabbatical leave.
He wishes to thank the University of Utah for the hospitality and inspiring environment, and Anurag Singh for making this happen. 
The first author gratefully acknowledges support from the NSF through a Mathematical Sciences Research Postdoctoral Fellowship.
The authors would like to thank the anonymous referees for the thorough  reading and the valuable comments, corrections, and suggestions. 

\appendix
\section{The algorithm}\label{appendix: algorithm}

In this appendix we present an algorithm to compute $F$-pure thresholds of forms in two variables
	and discuss some of the practical issues surrounding its implementation. 
The first issue one faces is factoring the form:
	this factorization often happens in very large field extensions that cannot be handled by the computer. 
We dodge this issue here, assuming that a factorization is known from the start: $G=\ell_1^{a_1}\ldots\ell_n^{a_n}=\ell^\vv{a}$.
A na\"ive use of Theorem~\ref{thm: summary homogeneous case} then leads to Algorithm~\ref{naive alg}.

\ \\

\begin{algorithm}[H]
\DontPrintSemicolon
\Indp
\medskip
\SetKwInput{Input}{Input}\SetKwInput{Output}{Output}
 \Input{$\ell=(\ell_1,\ldots,\ell_n)$, $\vv{a}=(a_1,\ldots,a_n)$}
 \Output{The $F$-pure threshold of $G=\ell_1^{a_1}\ldots\ell_n^{a_n}$}
\medskip
\lIf{$a_i\ge \norm{\vv{a}}/2$, for some $i$}{\KwRet{$1/a_i$}}
 $\lambda\leftarrow 2/\norm{\vv{a}}$\;
 \lIf{$\fpt(G)=\lambda$}{\KwRet{$\lambda$}}\label{line: fpt test}
 $e\leftarrow 1$\;
 \lWhile{$\tr{\lambda\vv{a}}{e}\not\in \Upper$}{$e\leftarrow e+1$}\label{line: while}
Locate the critical point $\vv{c}\in \QQ^n_{p^e}$ such that $\vv{c}\le \tr{\lambda\vv{a}}{e}$\;\label{line: find CP}
\KwRet{ $\max \big\{ \frac{c_1}{a_1}, \ldots, \frac{c_n}{a_n} \big \}$}
\caption{Na\"ive FPT Algorithm}\label{naive alg}
\medskip
\end{algorithm}

\newpage 

\begin{rem}
A few remarks are in order:
\begin{itemize}
	\item	In line~\ref{line: fpt test} we test if $\fpt(G)=\lambda$. 
		Testing whether $\fpt(G)$ equals any given rational number can be done rather efficiently 
			by a method of Schwede, which compares the ``non-$F$-pure ideals'' of \cite{fujino+schwede+takagi.non-lc-ideal-sheaves} 
			and test ideals.
		This was implemented by Schwede in the \emph{Macaulay2} package \emph{PosChar} \cite{poschar}, 
			in the command \texttt{isFPTPoly}.
	\item Checking if $\tr{\lambda\vv{a}}{e}\in \Upper$ in line~\ref{line: while} boils down to checking ideal membership
		(specifically, $\ell^{p^e\tr{\lambda\vv{a}}{e}}\in \idealm^{[p^e]}$), so it can be easily implemented.
	\item The `while' loop in line~\ref{line: while} is guaranteed to end, by Theorem~\ref{thm: summary homogeneous case}(2).
	\item In line~\ref{line: find CP}, we search for a critical point  $\vv{c}\le \tr{\lambda\vv{a}}{e}$, guaranteed to exist by
			Theorem~\ref{thm: summary homogeneous case}(2).
		This is done in the most na\"ive way, by successively subtracting $1/p^e$ from the coordinates of $\tr{\lambda\vv{a}}{e}$, until 
			a minimal point of $\Upper\cap \QQ^n_{p^e}$ is found.
		See Proposition~\ref{prop: characterization of CP} and the comments after its proof.
\end{itemize}
\end{rem}

\begin{discussion}\label{discussion: zooming in}
Although there is no question that Algorithm~\ref{naive alg} works in theory, it does not fare well in practice.
Directly checking if $\tr{\lambda\vv{a}}{e}\in \Upper$ (\ie checking if $\ell^{p^e\tr{\lambda\vv{a}}{e}}\in \idealm^{[p^e]}$) for increasingly large $e$
	(line~\ref{line: while}) can quickly lead to impractical computations 
	involving polynomials of extremely large degrees.
To get around this problem, write the componentwise non-terminating base $p$ expansion of $\lambda\vv{a}$: 
	\[\lambda\vv{a}=\frac{\vv{d}_1}{p}+\frac{\vv{d}_2}{p^2}+\frac{\vv{d}_3}{p^3}+\cdots.\]
Set $\idealb_0\coloneqq \idealm$, and successively compute $\idealb_{e}\coloneqq(\idealb_{e-1}^{[p]}:\ell^{\vv{d}_e})$.
Using the flatness of the Frobenius over $\kk[x,y]$ we see that
	\[\idealb_{e}=(\idealm^{[p^{e}]}:\ell^{p^{e-1}\vv{d}_1+p^{e-2}\vv{d}_2+\cdots+\vv{d}_e})=
	(\idealm^{[p^{e}]}:\ell^{p^e\tr{\lambda\vv{a}}{e}}),\]
	so $\idealb_e=\ideal{1}$ if and only if $\tr{\lambda\vv{a}}{e}\in \Upper$. 
When computing the ideals $\idealb_e$ we never raise polynomials to powers greater than $p$.
Moreover, it can be shown that each $\idealb_e$ can be generated by two forms whose degrees add up to at most $n$.
So we can check whether $\tr{\lambda\vv{a}}{e}\in \Upper$ for arbitrarily large $e$ without ever having to deal with  large degree polynomials.

Searching for a critical point $\vv{c}\le \tr{\lambda\vv{a}}{e}$ is also impractical if $e$ is large.
In our improved algorithm, Algorithm~\ref{alg}, we use the 
	above ideas in the search for critical points, relying on the next lemma, which relates 
	regions and critical points with respect to different ideals.
\end{discussion}

To simplify our language, in what follows we shall refer to a ``critical point associated with $\ell$ and $\idealb$'' simply as a ``$\idealb$-critical point'', 
	and denote the upper region attached to $\ell$ and $\idealb$ by $\Upper_\idealb$, unless $\idealb=\idealm$, 
	in which case we shall simply denote the upper region by the usual $\Upper$. 
We shall also use the term ``critical point under $\vv{u}$'' for a critical point $\vv{c}$ such that $\vv{c}\le \vv{u}$.

\begin{lem}\label{lem: zooming in and out}
	Let $\idealb$ be an ideal generated by two non-constant relatively prime forms in $\kk[x,y]$.
	Let $q$ be a power of $p$, $\vv{k}\in \NN^n$, and $\idealb'=(\idealb^{[q]}:\ell^\vv{k})$.
	Finally, let $\vv{u}\in (\QQnn)^n_{p^\infty}$ and $\vv{v}=(\vv{u}+\vv{k})/q$.
	Then\textup:
	\begin{enumerate}[(1)]
		\item $\vv{u}\in \Upper_{\idealb'}\ \iff\ \vv{v}\in \Upper_\idealb$.
		\item If $\vv{v}$ is a $\idealb$-critical point, 
				then $\vv{u}$ is a $\idealb'$-critical point.
			The converse holds if $u_i>0$ whenever $v_i>0$, and in particular when $\vv{u}$ has positive coordinates.
	\end{enumerate}
\end{lem}

\begin{proof}
	The first point follows from the fact that $(\idealb^{[q]}:\ell^\vv{k})^{[q']}=(\idealb^{[qq']}:\ell^{q'\vv{k}})$, 
		due to the flatness of the Frobenius over $\kk[x,y]$.
	The second point follows from the first and the characterization of critical points given in 
		Proposition~\ref{prop: characterization of CP}\iref{item: first characterization of CP discrete}.
\end{proof}

The following corollary shows how the above lemma will be used in Algorithm~\ref{alg}.

\begin{cor}\label{cor: zooming in and out}
	Adopt the notation introduced in Discussion~\ref{discussion: zooming in}.
	Suppose $\idealb_e=\ideal{1}$, so that $\tr{\lambda\vv{a}}{e}\in \Upper$.
	For each $j\in \NN$ with $j<e$, let 
		\[\vv{u}_j=p^j(\tr{\lambda\vv{a}}{e}-\tr{\lambda\vv{a}}{j})=\frac{\vv{d}_{j+1}}{p}+\frac{\vv{d}_{j+2}}{p^2}+\cdots+\frac{\vv{d}_{e}}{p^{e-j}}.\]
	Then\textup:
	\begin{enumerate}[(1)]
		\item	$\vv{u}_j\in \Upper_{\idealb_j}$ and, in particular, $\vv{u}_{e-1}=\vv{d}_{e}/p\in \Upper_{\idealb_{e-1}}$.
		\item If $\vv{c}_j\in \QQ^n_{p^{e-j}}$ is a $\idealb_{j}$-critical point under $\vv{u}_j$ with positive 
				coordinates, then $\vv{c}_j/p^{j}+\tr{\lambda \vv{a}}{j}$ is an $\idealm$-critical point under $\tr{\lambda\vv{a}}{e}$.
	\end{enumerate} 
\end{cor}

\begin{proof}
	Set $\vv{k}=p^{j}\tr{\lambda \vv{a}}{j}$.
	Then $\tr{\lambda \vv{a}}{e}=(\vv{u}_j+\vv{k})/p^{j}$ and $(\idealm^{[p^{j}]}:\ell^\vv{k})=\idealb_{j}$, and 
		the corollary follows easily from Lemma~\ref{lem: zooming in and out}. 
\end{proof}

We are now ready to present our improved algorithm---see Algorithm~\ref{alg} on page~\pageref{alg}---and prove its correctness.

\begin{algorithm}
\Indp
\SetInd{2em}{0em}
\DontPrintSemicolon
\medskip
\SetKwInput{Input}{Input}\SetKwInput{Output}{Output}
 \Input{$\ell=(\ell_1,\ldots,\ell_n)$, $\vv{a}=(a_1,\ldots,a_n)$}
 \Output{The $F$-pure threshold of $G=\ell_1^{a_1}\ldots\ell_n^{a_n}$}
\medskip
 \lIf{$a_i\ge \norm{\vv{a}}/2$, for some $i$}{\KwRet{$1/a_i$}}\label{line: endpoint 1}
 $\lambda\leftarrow 2/\norm{\vv{a}}$\;
 $m\leftarrow \text{minimal denominator of }\lambda$\;
 \If{
 	$p\mid m$
}
{
 	\leIf{$\fpt(G)=\lambda$}{\KwRet{$\lambda$}}{$\mu\leftarrow \infty$}\label{line: endpoint 2}
 }
 \lElse{$\mu\leftarrow $ multiplicative order of $p$ modulo $m$}\label{label: line 7}
 $\idealb_0\leftarrow \idealm$\;
 $e\leftarrow 0$\;
 \While{$\idealb_e\ne \ideal{1}$ and $e < \mu$\label{line: while 1}}{ 
	$e\leftarrow e+1$\;
	$\vv{d}_e\leftarrow e$th digit of the unique componentwise non-terminating base $p$ expansion of $\lambda\vv{a}$\;
	$\idealb_{e}\leftarrow (\idealb_{e-1}^{[p]}:\ell^{\vv{d}_e})$\;
}
\lIf{$\idealb_e\ne \ideal{1}$}{\KwRet{$\lambda$}}\label{line: endpoint 3}
$j\leftarrow e-1$\;
$\vv{v}_j\leftarrow \vv{d}_{e}/p$\;
$\vv{c}_j\leftarrow $ $\idealb_{j}$-critical point  under $\vv{v}_j$\label{line: search CP 1}\; 
\While{some component of $\vv{c}_j$ is $0$ and $j> 0$\label{line: while 2}}{ 
	$\vv{v}_{j-1}\leftarrow (\vv{c}_{j}+\vv{d}_{j})/p$\;
	$j\leftarrow j-1$\;
	$\vv{c}_j\leftarrow$ $\idealb_{j}$-critical point under $\vv{v}_j$\label{line: CP search}\; 
}
$\vv{c}\leftarrow\vv{c}_j/p^{j}+\tr{\lambda\vv{a}}{j}$\label{line: almost there}\;  
\KwRet{ $\max \big\{ \frac{c_1}{a_1}, \ldots, \frac{c_n}{a_n} \big \}$}
\caption{Improved FPT Algorithm}\label{alg}
\medskip
\end{algorithm}

\begin{thm}
	Algorithm~\ref{alg} works.
\end{thm}

\begin{proof}
	The algorithm may terminate prematurely in lines~\ref{line: endpoint 1}, \ref{line: endpoint 2}, or~\ref{line: endpoint 3}.
	If it terminates in line~\ref{line: endpoint 1}, then it returns the correct output, by Theorem~\ref{thm: summary homogeneous case}(1).
	If it terminates in line~\ref{line: endpoint 2}, then it returns the correct output for obvious reasons.	
	If the algorithm gets past line~\ref{label: line 7}, then
		at that point $\lambda=2/\deg(G)$ and $m$ is the minimal denominator of $\lambda$, and one of the following holds: 
	\begin{itemize}
		\item	$p\mid m$, $\fpt(G)\ne \lambda$, and $\mu=\infty$;
		\item $p\nmid m$, and $\mu$ is the multiplicative order of $p$ modulo $m$; it could still be the case that $\fpt(G)=\lambda$.
	\end{itemize}
	The `while' loop in line~\ref{line: while 1} will then compute the ideals $\idealb_e$ introduced in Discussion~\ref{discussion: zooming in}, 
		until $\idealb_e=\ideal{1}$ (\ie $\tr{\lambda\vv{a}}{e}\in \Upper$) or $e=\mu$.
	This loop terminates---this is clear if $\mu<\infty$, and if $\mu=\infty$, then we know that $\fpt(G)\ne \lambda$, 
		so some truncation $\tr{\lambda\vv{a}}{e}$ lies in $\Upper$, by Theorem~\ref{thm: summary homogeneous case}(2).
	If $\idealb_e\ne\ideal{1}$ at the end of this loop, 
		the aforementioned theorem allows us to conclude that $\fpt(G)=\lambda$; the algorithm 
		terminates in line~\ref{line: endpoint 3}, returning the correct output.
		
	Suppose the algorithm gets past line~\ref{line: endpoint 3}.
	At that point, we have $\tr{\lambda \vv{a}}{e}\in \Upper$, so $\vv{v}_{e-1}=\vv{d}_{e}/p\in \Upper_{\idealb_{e-1}}$, 
		by Corollary~\ref{cor: zooming in and out}(1).
	In line~\ref{line: search CP 1} we look  for a $\idealb_{e-1}$-critical point $\vv{c}_{e-1}\in \QQ^n_p$  
		under $\vv{v}_{e-1}$.
	If $\vv{c}_{e-1}$ has positive coordinates (so the `while' loop in line~\ref{line: while 2}
		is bypassed), then $\vv{c}=\vv{c}_{e-1}/p^{e-1}+\tr{\lambda\vv{a}}{e-1}$, computed in line~\ref{line: almost there}, 
		is an $\idealm$-critical point under $\tr{\lambda \vv{a}}{e}$, 
		by Corollary~\ref{cor: zooming in and out}(2),
		so that in this situation the algorithm returns the correct output.  
	
	It remains to examine what happens when $\vv{c}_{e-1}$ has some zero coordinate. 
	In the `while' loop in line~\ref{line: while 2}, points $\vv{v}_{j}$ and $\vv{c}_{j}$	($j=e-2, e-3,\ldots$) are
		constructed so that $\vv{v}_{j-1}= (\vv{c}_{j}+\vv{d}_{j})/p$ and 
		$\vv{c}_j$ is a $\idealb_{j}$-critical point under $\vv{v}_j$.
	We claim that, for each $0\le j<e$, 
		$\vv{v}_j\in \Upper_{\idealb_j}$ (so it makes sense to look for a critical point $\vv{c}_j$ under $\vv{v}_j$ in 
			line~\ref{line: CP search}) and $\vv{v}_j\le \vv{u}_j\coloneqq p^j(\tr{\lambda\vv{a}}{e}-\tr{\lambda\vv{a}}{j})$.
	This follows from an easy inductive argument using Lemma~\ref{lem: zooming in and out} 
		and the recursions $\vv{u}_{j-1}=(\vv{u}_{j}+\vv{d}_{j})/p$
		and $\idealb_j=(\idealb_{j-1}^{[p]}:\ell^{\vv{d}_j})$.
	The `while' loop may terminate when a point $\vv{c}_j$ with positive coordinates is found. 
	That point is a $\idealb_j$-critical point under $\vv{v}_j$, and therefore under $\vv{u}_j$.
	The point $\vv{c}$ computed in line~\ref{line: almost there} is thus an $\idealm$-critical point under $\tr{\lambda\vv{a}}{e}$, 
		by Corollary~\ref{cor: zooming in and out}(2), and the algorithm returns  the correct output. 
	Alternatively, the `while' loop may terminate when $j=0$, and 
		$\vv{c}=\vv{c}_0$ is a $\idealb_0$-critical point under $\vv{u}_0$.
	But $\idealb_0=\idealm$ and $\vv{u}_0=\tr{\lambda\vv{a}}{e}$, 
		so again the algorithm returns the correct output.
\end{proof}

Algorithm~\ref{alg} could be improved a bit by using the full strength of Lemma~\ref{lem: zooming in and out}(2) in 
	the stopping condition for the `while' loop in line~\ref{line: while 2}.
Comparing Algorithms~\ref{naive alg} and~\ref{alg}, the reader will notice that in Algorithm~\ref{alg} we are trying 
	to avoid having to test whether $\fpt(G)=\lambda$, by performing that test only when $p$ divides $m$.
This is because that test is often slow when the degree of $G$ is large.
If, on the other hand, a factorization of $G$ is not know from the start, then that test should be the first thing done in the algorithm.

Algorithm~\ref{alg} has been implemented by the second author in the \emph{Macaylay2}  package \emph{PosChar} \cite{poschar},
	in the \texttt{FPT2VarHomog} command.  
To illustrate its use, below we show a \emph{Macaulay2} session that computes the example given in the introduction:

\lstset{basicstyle=\small\ttfamily}
\begin{lstlisting}
i1 : installPackage("PosChar");
i2 : kk=GF(ZZ/5[a]/ideal(a^3+a+1));
i3 : kk[x,y];
i4 : L={x,y,x+y,x+a*y,x+a^2*y,x+a^3*y};
i5 : FPT2VarHomog(L,{420,419,417,390,402,438})
       46636216675556057485911762783799675605705641779512143
o5 = --------------------------------------------------------
     57968817327716179454988321140262996777892112731933593750
o5 : QQ
\end{lstlisting}
If a form in two variable is given, \texttt{FPT2VarHomog} will try to factor the form in an appropriate extension of the coefficient field
	and then compute its $F$-pure threshold using Algorithm~\ref{alg}:
\begin{lstlisting}
i6 : ZZ/2[x,y];
i7 : G=x^10*y^3+x^9*y^4+x^6*y^7+x^4*y^9+x^3*y^10+x*y^12+y^13;
i8 : FPT2VarHomog(G)
      315
o8 = ----
     2048
o8 : QQ

\end{lstlisting}
We invite the reader to download the package and try that command, and welcome any suggestions or bug reports.

\bibliographystyle{amsalpha}
\footnotesize
\bibliography{bibdatabase}
	
\end{document}